\pdfoutput=1
\documentclass[11pt]{article}

\usepackage[letterpaper,margin=1in]{geometry}
\usepackage{amsmath}
\usepackage{amsthm}
\usepackage{amssymb}
\usepackage{mathtools}
\usepackage{bm}
\usepackage{booktabs}
\usepackage{longtable}
\usepackage{graphicx}
\usepackage{tikz}
\usepackage{natbib}
\usepackage{aliascnt}
\usepackage[hidelinks]{hyperref}
\usepackage{cleveref}

\newcommand{\BlackBox}{\rule{1.5ex}{1.5ex}}

\newtheorem{example}{Example}
\newtheorem{theorem}{Theorem}

\newaliascnt{lemma}{theorem}
\newtheorem{lemma}[lemma]{Lemma}
\aliascntresetthe{lemma}

\newaliascnt{proposition}{theorem}
\newtheorem{proposition}[proposition]{Proposition}
\aliascntresetthe{proposition}

\newaliascnt{remark}{theorem}
\newtheorem{remark}[remark]{Remark}
\aliascntresetthe{remark}

\newaliascnt{corollary}{theorem}
\newtheorem{corollary}[corollary]{Corollary}
\aliascntresetthe{corollary}

\newaliascnt{definition}{theorem}

\aliascntresetthe{definition}

\newaliascnt{conjecture}{theorem}

\aliascntresetthe{conjecture}

\newaliascnt{axiom}{theorem}

\aliascntresetthe{axiom}

\newtheorem{assumption}{Assumption}

\crefname{theorem}{Theorem}{Theorems}
\Crefname{theorem}{Theorem}{Theorems}
\crefname{lemma}{Lemma}{Lemmas}
\Crefname{lemma}{Lemma}{Lemmas}
\crefname{proposition}{Proposition}{Propositions}
\Crefname{proposition}{Proposition}{Propositions}
\crefname{corollary}{Corollary}{Corollaries}
\Crefname{corollary}{Corollary}{Corollaries}
\crefname{definition}{Definition}{Definitions}
\Crefname{definition}{Definition}{Definitions}
\crefname{remark}{Remark}{Remarks}
\Crefname{remark}{Remark}{Remarks}
\crefname{example}{Example}{Examples}
\Crefname{example}{Example}{Examples}
\crefname{conjecture}{Conjecture}{Conjectures}
\Crefname{conjecture}{Conjecture}{Conjectures}
\crefname{axiom}{Axiom}{Axioms}
\Crefname{axiom}{Axiom}{Axioms}
\crefname{assumption}{Assumption}{Assumptions}
\Crefname{assumption}{Assumption}{Assumptions}
\crefformat{equation}{#2(#1)#3}
\Crefformat{equation}{#2(#1)#3}
\crefrangeformat{equation}{#3(#1)#4--#5(#2)#6}
\Crefrangeformat{equation}{#3(#1)#4--#5(#2)#6}
\crefmultiformat{equation}{#2(#1)#3}{ and~#2(#1)#3}{, #2(#1)#3}{, and~#2(#1)#3}
\Crefmultiformat{equation}{#2(#1)#3}{ and~#2(#1)#3}{, #2(#1)#3}{, and~#2(#1)#3}

\newenvironment{keywords}
  {\begin{quote}\small\noindent\textbf{Keywords:}\ }
  {\end{quote}}

\newcommand{\ep}{\varepsilon}

\newcommand{\R}{\mathbb{R}}
\newcommand{\E}{\mathbb{E}}

\newcommand{\bbS}{\mathbb{S}}
\newcommand{\bbZ}{\mathbb{Z}}
\newcommand{\bbN}{\mathbb{N}}
\newcommand{\bbP}{\mathbb{P}}

\newcommand{\bbT}{\mathbb{T}}
\newcommand{\bbB}{\mathbb{B}}

\DeclareMathOperator*{\argmin}{arg\,min}

\DeclareMathOperator{\spn}{span}

\DeclareMathOperator{\Ran}{Ran}

\DeclareMathOperator{\Tr}{Tr}
\DeclareMathOperator{\Var}{Var}
\DeclareMathOperator{\Bias}{Bias}
\DeclareMathOperator{\Cov}{Cov}

\newcommand{\mr}{\mathrm}

\newcommand{\mca}{\mathcal}

\newcommand{\caH}{\mathcal{H}}

\newcommand{\mcaN}{\mathcal{N}}

\newcommand{\xk}[1]{\left(#1\right)}
\newcommand{\zk}[1]{\left[#1\right]}

\providecommand{\ang}[1]{\left\langle{#1}\right\rangle}

\providecommand{\abs}[1]{\left\lvert{#1}\right\rvert}
\providecommand{\norm}[1]{\left\lVert{#1}\right\rVert}

\providecommand{\dd}{~\mathrm{d}}

\newcommand{\given}{\,|\,}

\newcommand{\falling}[2]{(#1)_{#2}}

\begin{document}

\title{Exact Generalization Error Curves of Kernel Ridge Regression for Functional Moment Estimation}

\author{
Yinan Ding\\
\small Zhili College, Tsinghua University\\
\small Haidian District, Beijing 100084, China\\
\small \texttt{dingyn23@mails.tsinghua.edu.cn}
\and
Yicheng Li\thanks{Corresponding author.}\\
\small KLATASDS-MOE, School of Statistics\\
\small East China Normal University\\
\small 3663 North Zhongshan Road, Shanghai 200062, China\\
\small \texttt{ycli@sfs.ecnu.edu.cn}
}

\date{\today}

\maketitle

\begin{abstract}
  Kernel ridge regression is a standard method for functional data analysis, but its exact behavior is less understood.
  We study tensor-product kernel ridge regression for estimating the \(r\)-th moment function of a random function based on noisy discrete observations.
  The formulation includes mean estimation, covariance estimation, and higher-order moment estimation in a single framework.
  Our main result gives a precise \(1+o_{\bbP}(1)\) expansion for the \(L^2\) error at each admissible regularization parameter.
  The expansion consists of bias and three variance terms corresponding respectively to variation across the independent sample paths, latent signal variation at each sample point, and variation from measurement errors,
  identifying the refined error structure underlying functional data.
  As applications, we show that KRR attains the minimax rate for source smoothness \(s\le2\) but becomes suboptimal in the sparse regime for \(s>2\) due to saturation.
  A technical ingredient is a set of concentration inequalities for \(U\)-statistics suited to the dependent product structure of functional observations.
\end{abstract}


\section{Introduction}

In recent years, functional data analysis (FDA) has become a powerful statistical framework for analyzing high-dimensional data,
where the basic observation is a random function or curve rather than a finite-dimensional vector.
These curves are usually observed only through noisy and discretized evaluations
\citep{ramsay1997FunctionalDataAnalysis,wang2016FunctionalDataAnalysis,hsing2015TheoreticalFoundationsFunctional}.
A fundamental problem in FDA is to recover population features of the underlying random function from these discrete observations, such as the mean and covariance functions
\citep{rice1991EstimatingMeanCovariance,yao2005FunctionalDataAnalysis,li2010UniformConvergenceRates}.
These population features are often used for downstream tasks such as functional principal component analysis (FPCA) \citep{hall2006PropertiesPrincipalComponent} and functional regression.

As one of the most popular approaches in FDA, kernel methods,
particularly kernel ridge regression (KRR),
have been widely studied in terms of minimax optimality for mean and covariance estimation
\citep{cai2011OptimalEstimationMean,cai2010NonparametricCovarianceFunction,gupta2026MinimaxOptimalEstimation}.
In these works, two main sources of error are often identified: the finite number of independent curves and the finite number of point evaluations per curve.
Consequently, a phase transition \citep{cai2011OptimalEstimationMean,cai2010NonparametricCovarianceFunction,zhang2016SparseDenseFunctional,gupta2026MinimaxOptimalEstimation} has been established between the sparse observation regime, where each curve is observed only a few times, and the dense observation regime, where each curve is observed many times.

However, such theory about optimal rates does not describe the full regularization path of a kernel estimator,
and only provides an upper bound on the generalization error up to constant or logarithmic factors.
It remains unknown how the generalization error changes with the regularization parameter,
how bad the error can be when the regularization parameter is not optimally tuned,
and how different sources of error contribute to the total error.
Knowing the answers to these questions would contribute to a comprehensive understanding of the behavior of KRR for functional data, and would also be useful for practical applications, such as generalized cross-validation and hyperparameter tuning.
Thus, determining the exact generalization error of KRR, not merely its upper bound, is of both theoretical and practical interest.

In this paper, we unify mean, covariance, and higher-order moment estimation in a single framework of tensor-product KRR.
We provide an exact characterization of the generalization error of the KRR estimator in the form of a \(1+o_{\bbP}(1)\) expansion,
which is valid for all admissible regularization parameters.
This expansion identifies refined contributions from different sources of error, including curve-level variation, latent point variation, measurement noise, and bias.
Simulation studies confirm that the expansion accurately predicts the generalization error curve.

\subsection{Problem Formulation}

Let \(\mca{I}\) be a compact metric input space equipped with a Borel
probability measure \(\rho\), e.g., \(\mca{I}=[0,1]\) with the Lebesgue
measure. Let \(X(\cdot): \mca{I} \to \R\) be a random function, and
\(X_1,\ldots,X_n\) be independent copies of \(X\). For each curve \(X_i\), the
i.i.d. sampling locations \(t_{ij} \in\mca{I}\), \(1\le j\le m\), are drawn
from \(\rho\), and we observe
\begin{equation}
  \label{eq:intro-observation-model}
  \begin{aligned}
    y_{ij}=X_i(t_{ij})+\ep_{ij},
    \qquad
    1\le i\le n,\quad 1\le j\le m,
  \end{aligned}
\end{equation}
where the design points and sample paths are mutually independent.
The measurement errors \( \ep_{ij} \) are mutually independent and are also independent of all
sample paths and all design points, satisfying \(\E\ep_{ij}=0\) and \(\E\ep_{ij}^2=\sigma^2\).

For a fixed integer \(r\ge1\), the estimation target is the \(r\)-th moment function
\begin{equation}
  \label{eq:intro-moment-target}
  \begin{aligned}
    \mu_r(t_1,\ldots,t_r)
    =
    \E(X(t_1)\cdots X(t_r)),
    \qquad (t_1,\ldots,t_r)\in \mca{I}^r.
  \end{aligned}
\end{equation}
By convention, we set \(\mu_0=1\). We also use the covariance kernel of the
\(r\)-fold product process. For
\(\bm{t}=(t_1,\ldots,t_r)\) and \(\bm{s}=(s_1,\ldots,s_r)\), define
\[
  \Sigma_r(\bm{t};\bm{s})
  \coloneqq
  \Cov\left(
        \prod_{a=1}^r X(t_a),
        \prod_{a=1}^r X(s_a)
  \right).
\]
Equivalently, as a function on \(\mca{I}^{2r}\), \(\Sigma_r=\mu_{2r}-\mu_r \otimes\mu_r\),
where \((\mu_r \otimes\mu_r)(\bm{t};\bm{s})=\mu_r(\bm{t})\mu_r(\bm{s})\).
In particular, the mean function is given by \(r=1\):
\[
  m(t)=\mu_1(t)=\E X(t).
\]
The covariance function is obtained from the first two moment functions:
\[
  C(t,s)
  =
  \Cov(X(t),X(s))
  =
  \Sigma_1(t;s)
  =
  \mu_2(t,s)-\mu_1(t)\mu_1(s).
\]
Higher-order moment functions are also useful for describing distributional
features of functional data beyond the mean and covariance; see, for example,
recent work on functional moment regression
\citep{li2026FunctionalMomentsRegression}.
In terms of the generalization error, we are interested in the \(L^2\) excess risk with respect to the product measure \(\rho^{\otimes r}\):
\[
  \E\left[
      \norm{\hat{\mu}_r-\mu_r}_{L^2(\mca{I}^r,\rho^{\otimes r})}^2
  \right]
  =
  \E\left[
      \int_{\mca{I}^r}
      (\hat{\mu}_r(\bm{t})-\mu_r(\bm{t}))^2
      \dd\rho^{\otimes r}(\bm{t})
  \right].
\]

\subsection{Our Contributions}

In this paper, we propose a unified tensor-product KRR estimator \(\hat{\mu}_{r,\lambda}\) for the \(r\)-th moment function \(\mu_r\) from discretely observed functional data.
The estimator is built from products of observations from the same curve, with a convention that removes the diagonal noise bias \citep{cai2010NonparametricCovarianceFunction,terada2026TheoryNonparametricCovariance}.
We refer to \Cref{subsec:moment-krr-estimator} for details of the construction.

Our main result provides the following deterministic equivalent.
Write
\(\mca{T} \coloneqq (t_{ij})\) for the collection
of design points.
Under mild assumptions and for the regularization parameter \(\lambda\) in the admissible range \cref{eq:main-admissible-lambda},
we prove that
\begin{align*}
  \operatorname{MSE}_{r,\lambda}(\mca{T})
  \coloneqq &
  \E\left[
      \norm{\hat{\mu}_{r,\lambda}-\mu_r}_{L^2(\mca{I}^r,\rho^{\otimes r})}^2
      \middle|
      \mca{T}
  \right]                                   \\
  =  &
  (1+o_{\bbP}(1))
  \left(
    \frac{1}{n}V_{\mathrm{curve}}(\lambda;\Sigma_r)
  +
    \frac{1}{mn}V_{\mathrm{point}}(\lambda;\Sigma_r)
  +
    \frac{\sigma^2}{mn}V_{\mathrm{noise}}(\lambda;\mu_{2r-2})
  \right)                                   \\
  & +(1+o_{\bbP}(1))B(\lambda;\mu_r)^2 .
\end{align*}
Here \(B(\lambda;\mu_r)\) and the three \(V\)-terms are deterministic quantities, stated precisely in \Cref{sec:main-theorem}.
The term \(B(\lambda;\mu_r)^2\) is the squared bias;
the term \(V_{\mathrm{curve}}\) captures variation across the \(n\) independent sample paths;
the term \(V_{\mathrm{point}}\) captures the fluctuation from the \(mn\) latent values \(X_i(t_{ij})\);
the term \(V_{\mathrm{noise}}\) accounts for the \(mn\) measurement errors.

The expansion is precise even in constant factors, and it provides a complete characterization of the generalization error curve,
showing how the error changes with the regularization parameter \(\lambda\).
This viewpoint is related to recent learning curve analysis for KRR and spectral algorithms
\citep{bordelon2020SpectrumDependentLearning,cui2022GeneralizationErrorRates,li2023AsymptoticLearningCurves,li2024GeneralizationErrorCurves}.
These existing works focus on nonparametric regression where the observations are independent pairs with independent noise.
In comparison, we study functional data, where different curves are independent but different sampling points on the same curve share the same random function and are therefore correlated.
This difference leads to a substantially different and more complex error structure, which is reflected in the three variance contributions above.

We formulate the source condition using the interpolation spaces \([H]^s\).
This allows misspecification, namely the target need not belong to the RKHS when \(s<1\).
Such assumptions are standard in KRR and spectral regularization theory
\citep{caponnetto2007OptimalRatesRegularized,lin2020OptimalRatesSpectral,fischer2020SobolevNormLearning,zhang2024OptimalityMisspecifiedSpectral},
and have recently been used for mean and covariance estimation with spectral regularization \citep{gupta2026MinimaxOptimalEstimation}.
After optimizing \(\lambda\), our expansion shows that KRR attains the minimax rate for smoothness \(s\le2\), recovering the sparse--dense phase transition known from FDA minimax theory for mean and covariance estimation.
When \(s>2\), it instead reveals the saturation of KRR in the sparse regime, in line with the qualification limit of Tikhonov regularization \citep{li2024SaturationEffectKernel}.

The main contributions are as follows.
\begin{itemize}
  \item We formulate general moment function estimation from discretely observed functional data as tensor-product KRR.
  The construction includes mean, covariance, and higher-order moment estimation in one framework.

  \item We prove a \(1+o_{\bbP}(1)\) generalization error expansion for KRR.
  The expansion consists of four deterministic terms:
  the squared bias \(B^2\) and variance coefficients \(V_{\mathrm{curve}}\), \(V_{\mathrm{point}}\), and \(V_{\mathrm{noise}}\).
  This expansion identifies the refined error structure of functional data that was vague in the literature.

  \item
  As applications of the expansion,
  we show that KRR attains the minimax rate for \(s\le2\) in both the sparse and dense regimes.
  For \(s>2\), we prove its qualification-two saturation and resulting suboptimality in the sparse regime.

  \item On the technical side, we establish a Bernstein-type concentration inequality for operator-valued \(U\)-statistics.
  This inequality is of independent interest and may be useful for other problems in functional data analysis.
\end{itemize}

\subsection{Related Work}

\paragraph{Functional data and moment estimation.}
Functional data analysis has a broad literature; see, for example,
\citet{ramsay1997FunctionalDataAnalysis},
\citet{wang2016FunctionalDataAnalysis}, and
\citet{hsing2015TheoreticalFoundationsFunctional}. A central problem is to
estimate the mean and covariance functions from noisy discrete observations.
Traditional approaches are often based on local smoothing or spline methods
\citep{rice1991EstimatingMeanCovariance,yao2005FunctionalDataAnalysis,li2010UniformConvergenceRates}.
Kernel methods have also been widely used
\citep{cai2011OptimalEstimationMean,cai2010NonparametricCovarianceFunction,gupta2026MinimaxOptimalEstimation}.
For independent design, the sparse--dense transition was made explicit for
mean estimation by \citet{cai2011OptimalEstimationMean} and for covariance
estimation by \citet{cai2010NonparametricCovarianceFunction}; related
phenomena appear in FPCA and in unified sparse/dense asymptotics
\citep{hall2006PropertiesPrincipalComponent,zhang2016SparseDenseFunctional}.
Later work developed covariance estimators with structural constraints such as positive semidefiniteness and low rank \citep{wong2019NonparametricOperatorregularizedCovariance,wang2022LowRankCovarianceFunction}.
\citet{gupta2026MinimaxOptimalEstimation} studied
spectral regularization for mean and covariance estimation
under misspecified source conditions.
More general covariance least-squares estimators have also been analyzed through oracle inequalities
\citep{terada2026TheoryNonparametricCovariance}.
These works mainly give optimized rates for mean and covariance estimators,
while our target is to determine the exact generalization error for a general moment function.

\paragraph{KRR, misspecification, and learning curves.}
KRR is a central method in nonparametric regression and statistical learning
\citep{steinwart2008SupportVectorMachines,caponnetto2007OptimalRatesRegularized}.
It is also a basic example of spectral algorithms. Misspecified settings,
where the target may lie outside the RKHS but inside an interpolation space,
have been studied for KRR and spectral algorithms
\citep{fischer2020SobolevNormLearning,lin2020OptimalRatesSpectral,zhang2024OptimalityMisspecifiedSpectral};
for discretely observed functional data, closely related source conditions are
used in \citet{gupta2026MinimaxOptimalEstimation}.
Beyond optimized rates, a more recent line of work studies learning curves, that
is, how the error changes with the regularization parameter.
\citet{bordelon2020SpectrumDependentLearning} used methods from Gaussian
processes and statistical physics to derive spectrum-dependent learning curves
for kernel regression and wide neural networks.
\citet{cui2022GeneralizationErrorRates} studied KRR under Gaussian design and
characterized the crossover from effectively noiseless to noisy regimes.
\citet{li2023AsymptoticLearningCurves} later gave rigorous matching
asymptotic rates for the KRR learning curve. More recently,
\citet{li2024GeneralizationErrorCurves} obtained a \(1+o_{\bbP}(1)\)
deterministic equivalent for analytic spectral algorithms, which keeps the
leading constant factors.
Compared with these works, our main new feature is the independent-design
functional-data sampling scheme. The observations are grouped by curves, and
the latent signal fluctuation and measurement noise enter the learning curve in
different ways; this leads to the three variance terms in the main
expansion.

\paragraph{\texorpdfstring{\(U\)}{U}-statistic concentration.}
The concentration estimates needed in this paper are not for a single pooled
\(U\)-statistic. A classical \(k\)-th order \(U\)-statistic based on
independent observations \(Z_1,\ldots,Z_N\) is
\[
  \frac{1}{\binom{N}{k}}
  \sum_{1\le i_1<\cdots<i_k \le N}
  h(Z_{i_1},\ldots,Z_{i_k}).
\]
In our analysis the statistic is instead averaged first
within each curve and then over curves:
\[
  \frac{1}{n}\sum_{i=1}^n
  \frac{1}{m(m-1)\cdots(m-k+1)}
  \sum_{j_1,\ldots,j_k \text{ distinct}}
  h(Z_{i,j_1},\ldots,Z_{i,j_k}).
\]
Here \(h\) may be scalar-, vector-, or operator-valued. Classical scalar
inequalities for \(U\)-statistics go back to
\citet{hoeffding1963ProbabilityInequalitiesSums}.
\citet{arcones1995BernsteintypeInequalityUstatistics} used the Hoeffding
decomposition to obtain Bernstein-type bounds for nondegenerate scalar
\(U\)-statistics and \(U\)-processes, with the first projection playing the
leading role.  \citet{terada2026TheoryNonparametricCovariance} recently used a
modified Talagrand inequality for second-order scalar \(U\)-statistics
to obtain oracle inequalities for covariance estimation.

For Hilbert-valued \(U\)-statistics,
\citet{giraudo2025ExponentialInequalityHilbertvalued} proved exponential
inequalities by combining martingale inequalities, decoupling, and the
Hoeffding decomposition. Operator-valued inequalities use a different set of
tools. The matrix case is based on the matrix Laplace-transform method
\citep{tropp2015IntroductionMatrixConcentration}, and
\citet{minsker2017ExtensionsBernsteinsInequality} extended this approach to
self-adjoint operators on separable Hilbert spaces with bounds involving
effective rank.
By decoupling and the Hoeffding decomposition, \citet{sriperumbudur2022ApproximateKernelPCA}
derived a second-order operator-valued \(U\)-statistic inequality,
but its proof contains a gap in the derivation of the moment-generating function.
Rather than relying on that result, we give an independent proof by
inserting Hoeffding's blocking representation into the Tropp--Minsker matrix
Laplace-transform argument, and obtain a general \(k\)-th-order inequality.

 \section{Preliminaries}
\label{sec:preliminaries}

\subsection{RKHS}
\label{subsec:rkhs-preliminaries}

Let \(k:\mca{I}\times\mca{I}\to\R\) be a continuous positive-definite kernel,
and let \(\caH\) be its associated separable reproducing kernel Hilbert space.
We write \(k_x \coloneqq k(x,\cdot)\).
The reproducing property gives
\( f(x)=\ang{f,k_x}_{\caH}, \) for \( f\in\caH,\ x\in\mca{I}. \)
Since \(\mca{I}\) is compact and \(k\) is continuous,
we assume that
\( \kappa^2
  \coloneqq
  \sup_{x\in\mca{I}} k(x,x)
  <\infty . \)
Consequently,
\[
  \norm{f}_{L^2(\mca{I},\rho)}
  \le
  \norm{f}_{L^\infty(\mca{I},\rho)}
  \le
  \kappa\norm{f}_{\caH},
  \qquad f\in\caH,
\]
so \(\caH\) is naturally and continuously embedded into
\(L^2(\mca{I},\rho)\).

The kernel integral operator \(T:L^2(\mca{I},\rho)\to L^2(\mca{I},\rho)\) is
defined by
\begin{equation}
    (Tf)(x)
  \coloneqq
  \int_{\mca{I}} k(x,x')f(x') \dd\rho(x').
\end{equation}
It is positive, self-adjoint, compact, and trace class
\citep{steinwart2012MercersTheoremGeneral}.
By Mercer's
theorem \citep{steinwart2008SupportVectorMachines}, we have the eigendecomposition
\begin{equation}
  T e_j = \lambda_j e_j,\quad
  k(x,x')
  =
  \sum_{j=1}^\infty \lambda_j e_j(x)e_j(x'),
\end{equation}
where \((e_j)_{j\ge1}\) is an orthonormal system in \(L^2(\mca{I},\rho)\) chosen
with continuous representatives, and the positive eigenvalues \((\lambda_j)_{j\ge1}\) are repeated according to multiplicity.
The latter series converges absolutely and uniformly on \(\mca{I}\times\mca{I}\).

Following the literature \citep[e.g.][]{caponnetto2007OptimalRatesRegularized,zhang2024OptimalityMisspecifiedSpectral},
we assume standard polynomial eigenvalue decay for the kernel integral operator \(T\) as follows.
This assumption is satisfied by many Matérn and Sobolev kernels under standard domain and design conditions,
and is connected to the smoothness of the kernel and the underlying function space.

\begin{assumption}[Eigenvalue Decay]
  \label{ass:prelim-eigen-decay}
  There exist \(\beta>1\) and constants
  \(0<c_{\mathrm{eig}}\le C_{\mathrm{eig}}<\infty\) such that
  \[
    c_{\mathrm{eig}}j^{-\beta}
    \le
    \lambda_j
    \le
    C_{\mathrm{eig}}j^{-\beta},
    \qquad j\ge1.
  \]
\end{assumption}

We further introduce the interpolation space of the RKHS \citep{steinwart2012MercersTheoremGeneral}.
For \(s\ge0\), define the fractional power
\[
  T^s f
  =
  \sum_{j=1}^\infty \lambda_j^s \ang{f,e_j}_{L^2} e_j
\]
on its natural domain.  The interpolation space \([H]^s\) is
\[
  [H]^s
  \coloneqq
  \Ran(T^{s/2})
  =
  \left\{
  \sum_{j=1}^{\infty} a_j \lambda_j^{s/2} e_j:
  (a_j)_{j\ge1}\in\ell^2
  \right\}
  \subseteq L^2(\mca{I},\rho).
\]
For \(f=\sum_{j=1}^{\infty} a_j \lambda_j^{s/2} e_j\) and \(g=\sum_{j=1}^{\infty} b_j \lambda_j^{s/2} e_j\),
define the inner product on \([H]^s\) by
\[
  \ang{f,g}_{[H]^s}
  \coloneqq
  \sum_{j=1}^{\infty} a_j b_j
  =
  \ang{T^{-s/2} f,T^{-s/2} g}_{L^2}.
\]
In particular, \([H]^0\subseteq L^2(\mca{I},\rho)\) and \([H]^1=\caH\).
The parameter \(s\) is the smoothness index on this interpolation scale: larger
\(s\) requires faster decay of the coefficients in the eigenbasis, and
therefore represents higher smoothness.

To establish the precise generalization error to the greatest extent possible, we need to impose additional regularity conditions on the RKHS.
We use the regular RKHS condition introduced in \citet{li2024GeneralizationErrorCurves}.
Let \((\theta_m)_{m\ge1}\) be the decreasing sequence of distinct positive eigenvalues of \(T\) and \(V_m\) be the eigenspace associated with \(\theta_m\).
Set \(d_m=\dim(V_m)\).
Let $k_m$ be the projection kernel onto the eigenspace \(V_m\).
Concretely, taking an orthonormal basis \(\{e_{m,\ell}:1\le\ell\le d_m\}\) of \(V_m\),
\( k_m(x,x')
  \coloneqq
  \sum_{\ell=1}^{d_m} e_{m,\ell}(x)e_{m,\ell}(x') \),
which is independent of the chosen orthonormal basis of \(V_m\).
We make the following regularity assumption on the RKHS.

\begin{assumption}[Regular RKHS]
  \label{ass:prelim-regular-rkhs}
  There exists \(M>0\) such that, for every \(N\ge1\),
  \[
    \sup_{x\in\mca{I}}
    \sum_{m=1}^{N} k_m(x,x)
    \le
    M\sum_{m=1}^{N} d_m .
  \]
\end{assumption}

This assumption is satisfied by many commonly used kernels.
In particular, if the eigenfunctions are uniformly bounded, then the RKHS is regular.
Moreover, this regularity condition is also closely connected to the embedding property of the RKHS.
For \(\alpha>0\), say that \(\caH\) has the embedding property of order
\(\alpha\) if \([H]^\alpha\) is continuously embedded into
\(L^\infty(\mca{I},\rho)\).
Equivalently, the embedding norm
\( M_\alpha
  \coloneqq
  \norm{[H]^\alpha\hookrightarrow L^\infty(\mca{I},\rho)} \)
is finite.
The embedding index is defined as
\[
  \alpha_0
  \coloneqq
  \inf\{\alpha>0: M_\alpha<\infty\}.
\]
It is easy to see that $\alpha = 1$ satisfies the embedding property, and \citet[Lemma~10]{fischer2020SobolevNormLearning} has shown that \(\alpha_0 \ge \frac{1}{\beta}\).
By the \(L^\infty\)-embedding characterization of
\citet[Theorem~9]{fischer2020SobolevNormLearning}, the same quantity can be written as
\[
  M_\alpha^2
  =
  \left\|
  \sum_{j=1}^{\infty} \lambda_j^\alpha e_j^2
  \right\|_{L^\infty(\rho)}
  =
  \operatorname*{ess sup}_{x\in\mca{I}}
  \sum_{j=1}^{\infty} \lambda_j^\alpha e_j(x)^2 .
\]
According to Lemma~\ref{lem:appendix-supnorm}, the embedding index satisfies \(\alpha_0=\frac{1}{\beta}\)
under Assumptions~\ref{ass:prelim-eigen-decay} and \ref{ass:prelim-regular-rkhs}.
It can also be found in \citet[Proposition~2.2]{li2024GeneralizationErrorCurves}.
We remark that if we only have the embedding property of order \(\alpha_0 \leq 1\),
then the error decomposition still holds but for a slightly restricted range of \(\lambda\) depending on \(\alpha_0\).

We next introduce a lower analogue of the regular RKHS condition, which is not needed for the proof of the main theorem but is used to obtain a lower bound for \(V_{\mathrm{point}}\) in the noiseless case.

\begin{assumption}[Lower regularity]
  \label{ass:prelim-ae-lower-regularity}
  The spectral projection kernels satisfy
  \[
    \liminf_{N\to\infty}
    \frac{\sum_{m=1}^N k_m(x,x)}
    {\sum_{m=1}^N d_m}
    >
    0
    \qquad
    \text{for }\rho\text{-a.e. }x\in\mca I.
  \]
\end{assumption}

A uniform lower bound \(e_i(x)^2\ge c\), valid for all sufficiently large
\(i\) and all \(x\in\mca{I}\), immediately implies
Assumption~\ref{ass:prelim-ae-lower-regularity}.
We adapt the following examples from \citet[Examples~2.1--2.3 and Appendix~B.1]{li2024GeneralizationErrorCurves} to show that several broad classes of regular RKHSs also satisfy the lower regularity counterpart \cref{ass:prelim-ae-lower-regularity}.

\begin{example}[Shift-invariant periodic kernels]\label{ex:lower-regular-periodic}
  Let \(\mca{I}=\bbT^d=[-\pi,\pi)^d\), and let \(\rho\) be the uniform measure
  on \(\bbT^d\).  Suppose \(k(x,y)=h(x-y)\) is shift-invariant.  The Fourier
  basis \(\{\phi_\ell(x)=\exp(\sqrt{-1}\ang{\ell,x}):\ell\in\bbZ^d\}\) gives
  eigenfunctions of \(T\).  Equivalently, over \(\R\), one may use the
  corresponding sine-cosine basis.  For every eigenspace, the spectral
  projection has constant diagonal, and therefore \(k_m(x,x)=d_m\) for all
  \(x\in\bbT^d\).  Hence the RKHS is regular and satisfies
  Assumption~\ref{ass:prelim-ae-lower-regularity}.  If the
  Fourier eigenvalues satisfy
  \(\lambda_\ell \asymp (1+\norm{\ell}^2)^{-\alpha}\), then
  \(\caH\) is equivalent to the Sobolev space \(H^\alpha(\bbT^d)\), and
  \([H]^s\) is equivalent to \(H^{s\alpha}(\bbT^d)\).
\end{example}

\begin{example}[Dot-product kernels on the sphere]\label{ex:lower-regular-sphere}
  Let \(\mca{I}=\bbS^d\) be the \(d\)-dimensional sphere equipped with the
  uniform measure.  Consider a dot-product kernel
  \(k(x,y)=h(\ang{x,y})\), where \(h\) is a function on \([-1,1]\).  By the
  Funk-Hecke formula \citep[Theorem~1.2.9]{dai2013ApproximationTheoryHarmonic},
  the spherical harmonics of a fixed degree have a common eigenvalue under
  \(T\).  If
  \(d_q\) denotes the dimension of the degree-\(q\) spherical harmonics, then
  \(d_q=\binom{q+d}{q}-\binom{q-2+d}{q-2}\asymp q^{d-1}\).
  \citet[Example~2.2 and Appendix~B.1.1]{li2024GeneralizationErrorCurves}
  showed that this RKHS is regular.  The same formula for the
  spherical-harmonic projection kernels also gives
  Assumption~\ref{ass:prelim-ae-lower-regularity}; see
  Appendix~\ref{app:omitted-proofs}.
\end{example}

\begin{example}[Dot-product kernels on the ball]\label{ex:lower-regular-ball}
  Let \(d\ge2\), let \(\mca{I}=\bbB^d=\{x\in\R^d:\norm{x}\le1\}\), and let \(\rho\) be
  proportional to the classical weight
  \(W(x)=(1-\norm{x}^2)^{-1/2}\).  Consider again a dot-product kernel
  \(k(x,y)=h(\ang{x,y})\).  The analogue of the Funk-Hecke formula on the ball
  \citep[Theorem~11.1.9]{dai2013ApproximationTheoryHarmonic} shows that \(T\)
  acts by a scalar on the space of orthogonal polynomials of degree exactly
  \(q\), whose dimension is \(\binom{q+d-1}{q}\).  As in
  \citet[Example~2.3]{li2024GeneralizationErrorCurves}, this gives a regular
  RKHS.  The same projection-kernel calculation also gives
  Assumption~\ref{ass:prelim-ae-lower-regularity}; see
  Appendix~\ref{app:omitted-proofs}.
\end{example}

\subsection{Moment KRR Estimator}
\label{subsec:moment-krr-estimator}

Before defining the estimator, we first introduce some notation for tensor products.
For a probability measure \(\rho\), \(\rho^{\otimes r}\) denotes the product
measure on \(\mca{I}^r\).  If \(H\) is a Hilbert space, \(H^{\otimes r}\)
denotes its \(r\)-fold Hilbert tensor product; if \(A\) is a linear operator
on \(H\), \(A^{\otimes r}\) denotes the \(r\)-fold tensor-product operator
acting on \(H^{\otimes r}\).
We use the convention \(H^{\otimes 0}=\R\) and interpret an empty product as \(1\).

The observation model is given in \cref{eq:intro-observation-model}.
Fix
\(r\ge1\), assume \(m\ge r\), and let
\(\mu_r\) be the moment target defined in
\cref{eq:intro-moment-target}.
We write \( \mca{T}=(t_{ij})_{1\le i\le n,1\le j\le m} \)
for all design points.
On \(\mca{I}^r\), we use the \(r\)-fold tensor-product RKHS
\[
  \caH^{\otimes r}
  =
  \underbrace{\caH\otimes\cdots\otimes\caH}_{r\text{ times}},
\]
whose reproducing kernel is
\[
  (\bm{x},\bm{x}')
  \mapsto
  \prod_{\ell=1}^{r} k(x_\ell,x'_\ell),
  \qquad
  \bm{x}=(x_1,\ldots,x_r),\quad
  \bm{x}'=(x'_1,\ldots,x'_r).
\]
We write
\[
  K(\bm{x})
  \coloneqq
  k_{x_1} \otimes\cdots\otimes k_{x_r} \in\caH^{\otimes r},
\]
so that \(f(\bm{x})=\ang{f,K(\bm{x})}_{\caH^{\otimes r}}\) for
\(f\in\caH^{\otimes r}\).
The corresponding population tensor-product operator is
\[
  T_r f(\bm{x})
  \coloneqq
  \int_{\mca{I}^r} f(\bm{x}')\prod_{\ell=1}^{r} k(x_\ell,x'_\ell) \dd\rho^{\otimes r}(\bm{x}'),
  \qquad f\in\caH^{\otimes r}.
\]

We denote the index set of distinct coordinates by
\[
  I_m^r
  \coloneqq
  \{(j_1,\ldots,j_r)\in\{1,\ldots,m\}^r:
  j_a \ne j_b \text{ whenever } a\ne b\},
  \qquad
  \falling{m}{r} \coloneqq |I_m^r|.
\]
We also set \(I_m^0=\{()\}\) and \((m)_0=1\).
For \(\bm{j}=(j_1,\ldots,j_r)\in I_m^r\), define the shorthand notation
\[
  t_{i\bm{j}}
  \coloneqq
  (t_{ij_1},\ldots,t_{ij_r}),
  \qquad
  Y_{i\bm{j}}
  \coloneqq
  \prod_{\ell=1}^{r} y_{ij_\ell}.
\]
Since the indices in \(\bm{j}\) are distinct, the corresponding measurement errors are distinct and independent, and we have
\[
  \E(Y_{i\bm{j}} \given \mca{T})
  =
  \mu_r(t_{i\bm{j}}).
\]
For
\(f\in\caH^{\otimes r}\), define the empirical squared loss over the \(r\)-th moment observations as
\begin{equation}
  \mca{L}_{n,m}^{(r)}(f)
  \coloneqq
  \frac{1}{n}\sum_{i=1}^{n}
  \frac{1}{\falling{m}{r}}\sum_{\bm{j}\in I_m^r}
  \zk{Y_{i\bm{j}}-f(t_{i\bm{j}})}^2
  .
\end{equation}
The moment KRR estimator is
\begin{equation}
  \label{eq:krr-objective}
  \hat{\mu}_{r,\lambda}
  \coloneqq
  \argmin_{f\in\caH^{\otimes r}}
  \zk{
  \mca{L}_{n,m}^{(r)}(f)
  +
  \lambda\norm{f}_{\caH^{\otimes r}}^2
  }.
\end{equation}
Using the representer theorem, \cref{eq:krr-objective} is solved explicitly as
\begin{equation}
  \label{eq:hat-mu}
  \hat{\mu}_{r,\lambda}
  =
  (\hat{T}_r+\lambda I)^{-1}\hat{\zeta}_r,
\end{equation}
where
\begin{equation}
  \label{eq:hat-zeta}
  \hat{\zeta}_r
  \coloneqq
  \frac{1}{n}\sum_{i=1}^{n}
  \frac{1}{\falling{m}{r}}\sum_{\bm{j}\in I_m^r}
  Y_{i\bm{j}} K(t_{i\bm{j}}) \in \caH^{\otimes r},
\end{equation}
and \( \hat{T}_r \) is an operator on \(\caH^{\otimes r}\) defined by
\begin{equation}
  \label{eq:hat-T}
    \hat{T}_r
  \coloneqq
  \frac{1}{n}\sum_{i=1}^{n}
  \frac{1}{\falling{m}{r}}\sum_{\bm{j}\in I_m^r}
  K(t_{i\bm{j}})\otimes K(t_{i\bm{j}}).
\end{equation}
Here \(u\otimes v\) denotes the rank-one operator
\(f\mapsto \ang{f,v}_{\caH^{\otimes r}}u\).  Thus \(\hat{T}_r\) is the
empirical counterpart of \(T_r\).
When \(r=1\) and \(r=2\), the estimator \(\hat{\mu}_{r,\lambda}\) corresponds to the estimation of the mean and second moment functions, respectively; the covariance function is obtained after centering \citep{cai2011OptimalEstimationMean,cai2010NonparametricCovarianceFunction}.
 \section{Main Results}
\label{sec:main-results}

We measure performance through the conditional generalization error
\[
  \operatorname{MSE}_{r,\lambda}(\mca{T})
  \coloneqq
  \E\left[
      \norm{\hat{\mu}_{r,\lambda}-\mu_r}_{L^2(\mca{I}^r,\rho^{\otimes r})}^2
      \given
      \mca{T}
  \right].
\]
While this formulation is convenient for our analysis,
the same results hold for the unconditional version or the high-probability version without taking expectations over \(\mca{T}\),
but proving either version would involve technical complications beyond our scope.
The remaining assumption concerns the regularity and nondegeneracy of the signal
process. To state it, write the tensor eigen expansion of the target in \(L^2(\mca{I}^r,\rho^{\otimes r})\) as
\[
  e_{\bm{i}}
  \coloneqq
  e_{i_1} \otimes\cdots\otimes e_{i_r},
  \qquad
  \Lambda_{\bm{i}}
  \coloneqq
  \prod_{a=1}^r \lambda_{i_a},
  \qquad
  \mu_r
  =
  \sum_{\bm{i}\in\bbN^r} c_{\bm{i}} e_{\bm{i}}.
\]

\begin{assumption}[Moment and Source Conditions]
  \label{ass:source}
  There are constants \(s>1/\beta\), \(s_1>1/\beta\), and \(s_2>1/\beta\)
  such that the following conditions hold.
  \begin{enumerate}
    \renewcommand{\labelenumi}{\textup{(\roman{enumi})}}
    \renewcommand{\theenumi}{\roman{enumi}}
    \item For every \(s'<s\), \(\mu_r \in ([H]^{s'})^{\otimes r}\).
    \item \(\mu_r \ne0\). Moreover, if \(s<2\),
    then \(\sum_{\Lambda_{\bm{i}}<\lambda}c_{\bm{i}}^2=\Omega(\lambda^s)\) as \(\lambda\downarrow0\).
    \item For \(q = 2, 4, \ldots, 2r-4\), \(\mu_q\) is bounded by a positive
    number \(C_{\mu_q}\), i.e., \(\abs{\mu_q}\le C_{\mu_q}\).
    \item \(\Sigma_r \in ([H]^{s_1})^{\otimes 2r}\) and \(\mu_{2r-2} \in ([H]^{s_2})^{\otimes 2r-2}\).
    By the embedding property, there exist \(C_{\Sigma_r}, C_{\mu_{2r-2}}>0\)
    such that \(\abs{\Sigma_r}\le C_{\Sigma_r}\) and
    \(\abs{\mu_{2r-2}}\le C_{\mu_{2r-2}}\). Moreover,
    \(\Sigma_r \neq0\) and \(\mu_{2r-2} \neq0\).
  \end{enumerate}
\end{assumption}

\begin{remark}[Moment condition]
    The following stronger moment condition is sufficient for
    Assumption~\ref{ass:source}\textup{(i)}, \textup{(iii)}, and the
    regularity and boundedness requirements in
    Assumption~\ref{ass:source}\textup{(iv)}. Suppose that there exists \(s>1/\beta\)
    such that, for every \(s'<s\),
    \[
      X\in [H]^{s'}\; a.s.
      \qquad
      \text{and}
      \qquad
      \E\norm{X}_{[H]^{s'}}^{2r}<\infty.
    \]
    Indeed, this condition gives the source regularity of \(\mu_r\) by taking
    expectations in the tensor expansion of \(X^{\otimes r}\). It also implies
    the boundedness of the lower even moments in
    Assumption~\ref{ass:source}\textup{(iii)} through the \(L^\infty\) embedding
    of \([H]^{s'}\) for any \(s'\in(1/\beta,s)\). Finally, choosing
    \(s_1,s_2 \in(1/\beta,s)\) yields the regularity and boundedness of
    \(\Sigma_r\) and \(\mu_{2r-2}\) required in
    Assumption~\ref{ass:source}\textup{(iv)}.
\end{remark}

\begin{example}[Exact power law \(\mu_r\)]\label{ex:exact-power-law-mu}
  Suppose \Cref{ass:prelim-eigen-decay} holds and
  \(c_{\bm{i}}^2 \asymp \Lambda_{\bm{i}}^{s+\frac{1}{\beta}}\) for some
  \(s>\frac{1}{\beta}\). Then \(\mu_r\) satisfies
  Assumption~\ref{ass:source}\textup{(i)} and \textup{(ii)}.
  Similar forms are also considered in \citet{cui2022GeneralizationErrorRates} and \citet{li2023AsymptoticLearningCurves}.
\end{example}

\subsection{Main Theorem}\label{sec:main-theorem}

Let us define the deterministic squared bias by
\begin{equation}
  B(\lambda;\mu_r)^2
  \coloneqq
  \norm{
    \lambda(T_r+\lambda I)^{-1}\mu_r
  }_{L^2(\mca{I}^r,\rho^{\otimes r})}^2,
  \qquad
  \text{where}
  \quad
  T_r=T^{\otimes r}.
\end{equation}
The deterministic variance terms are more complicated.
Write \(T_{r,\lambda} \coloneqq T_r+\lambda I\).
Let
\(\bm{t}=(t_1,\ldots,t_r)\), \(\bm{t}'=(t_1',\ldots,t_r')\), and
\(\tilde{\bm{t}}=(t_1,\ldots,t_{r-1})\),
\(\tilde{\bm{t}}'=(t_1',\ldots,t_{r-1}')\).
Denote by \( \E_{\rho} \) the expectation with respect to independent \(\rho\)-distributed random variables
\(t_1,\ldots,t_r,t_1',\ldots,t_r'\) and \(\xi\).
Write \( \ang{\cdot,\cdot}_{L^2} \) for the \(L^2\) inner product in \(L^2(\mca{I}^r,\rho^{\otimes r})\).
Define the three variance terms as
\begin{equation}
  \begin{aligned}
    V_{\mathrm{curve}}(\lambda;\Sigma_r)
    & \coloneqq
    \E_{\rho}\Bigl[
      \Sigma_r(\bm{t};\bm{t}')
      \ang{
        T_{r,\lambda}^{-1} K(\bm{t}),
        T_{r,\lambda}^{-1} K(\bm{t}')
      }_{L^2}
      \Bigr],
    \\
    V_{\mathrm{point}}(\lambda;\Sigma_r)
    & \coloneqq
    r \E_{\rho}\Bigl[
      \Sigma_r(\xi,\tilde{\bm{t}};\xi,\tilde{\bm{t}}')
      \ang{
        T_{r,\lambda}^{-1} K(\xi,\tilde{\bm{t}}),
        T_{r,\lambda}^{-1} K(\xi,\tilde{\bm{t}}')
      }_{L^2}
      \Bigr],
    \\
    V_{\mathrm{noise}}(\lambda;\mu_{2r-2})
    & \coloneqq
    r \E_{\rho}\Bigl[
      \mu_{2r-2}(\tilde{\bm{t}};\tilde{\bm{t}}')
      \ang{
        T_{r,\lambda}^{-1} K(\xi,\tilde{\bm{t}}),
        T_{r,\lambda}^{-1} K(\xi,\tilde{\bm{t}}')
      }_{L^2}
      \Bigr].
  \end{aligned}
\end{equation}
Now, we are ready to state the main theorem of this paper.

\begin{theorem}
  \label{thm:main}
  Suppose Assumptions~\ref{ass:prelim-eigen-decay},
  \ref{ass:prelim-regular-rkhs}, and \ref{ass:source} hold,
  \(\sigma^2>0\), and
  \(\lambda=\lambda(n,m)\to0\). If
  \begin{equation}
    \label{eq:main-admissible-lambda}
    \begin{gathered}
      \lambda=\Omega((mn)^{-\theta})
      \quad\text{for some }0<\theta<\beta,
      \qquad
      \left(\log\frac{1}{\lambda}\right)^{r-1}=o(m),
    \end{gathered}
  \end{equation}
  then
  \begin{equation}
    \label{eq:main-exact-expansion}
    \begin{aligned}
      \operatorname{MSE}_{r,\lambda}(\mca{T})
      = &
      (1+o_{\bbP}(1))
      \left(
        \frac{1}{n}V_{\mathrm{curve}}(\lambda;\Sigma_r)
      +
        \frac{1}{mn}V_{\mathrm{point}}(\lambda;\Sigma_r)
      +
        \frac{\sigma^2}{mn}V_{\mathrm{noise}}(\lambda;\mu_{2r-2})
      \right) \\
      & +
      (1+o_{\bbP}(1))B(\lambda;\mu_r)^2 .
    \end{aligned}
  \end{equation}
\end{theorem}

If the observations are noiseless, i.e., \(\sigma^2=0\), then the last variance term in \cref{eq:main-exact-expansion} vanishes. In this case, we have the following result.

\begin{proposition}[Noiseless case]
  \label{cor:main-noiseless}
  Suppose Assumptions~\ref{ass:prelim-eigen-decay}--\ref{ass:source} hold, and \(y_{ij}=X_i(t_{ij})\), equivalently \(\sigma^2=0\). If
  \(\lambda=\lambda(n,m)\to0\) satisfies \cref{eq:main-admissible-lambda},
  then
  \begin{equation}
    \label{eq:main-noiseless-expansion}
    \operatorname{MSE}_{r,\lambda}(\mca{T})
    =
    (1+o_{\bbP}(1))
    \left(
      \frac{1}{n}V_{\mathrm{curve}}(\lambda;\Sigma_r)
    +
      \frac{1}{mn}V_{\mathrm{point}}(\lambda;\Sigma_r)
    \right)
    +
    (1+o_{\bbP}(1))B(\lambda;\mu_r)^2 .
  \end{equation}
\end{proposition}
 \section{Discussion}
\label{sec:discussion}

\subsection{Consequences for Minimax Rates}

Optimizing the regularization parameter \(\lambda\) in the deterministic
equivalent \cref{eq:main-exact-expansion} or
\cref{eq:main-noiseless-expansion} gives the rate achieved by moment KRR.
Suppose that \(\mu_r\in([H]^s)^{\otimes r}\) for some \(s>1/\beta\), and
set \(\tilde{s} \coloneqq \min(s,2)\).
By \Cref{lem:bias-source-upper}, the deterministic squared bias satisfies
\[
  B(\lambda;\mu_r)^2=O(\lambda^{\tilde{s}}).
\]
On the variance side,
\Cref{lem:appendix-I0-expectations,lem:appendix-I1-expectations} give
\[
  V_{\mathrm{curve}}(\lambda;\Sigma_r)
  =
  \Theta(1),
  \qquad
  V_{\mathrm{point}}(\lambda;\Sigma_r)
  =
  O(\lambda^{-1/\beta}),
  \qquad
  V_{\mathrm{noise}}(\lambda;\mu_{2r-2})
  =
  \Theta(\lambda^{-1/\beta}).
\]
If \(\sigma^2>0\), the noise term contributes
\(\lambda^{-1/\beta}/(mn)\).
If \(\sigma^2=0\), Assumption~\ref{ass:prelim-ae-lower-regularity} gives the
corresponding lower bound
\(V_{\mathrm{point}}(\lambda;\Sigma_r)\asymp\lambda^{-1/\beta}\); see
\Cref{lem:appendix-I0-vpoint-lower}. Combining these estimates with
\cref{eq:main-exact-expansion} in the noisy case or with
\cref{eq:main-noiseless-expansion} in the noiseless case, one may choose
\[
  \lambda_\ast\ \text{so that}\quad
  \begin{cases}
    \lambda_\ast \asymp (mn)^{-\beta/(\tilde{s}\beta+1)},
    & m\lesssim n^{1/(\tilde{s}\beta)}, \\[0.4em]
    m^{-\beta} \lesssim\lambda_\ast \lesssim n^{-1/\tilde{s}},
    & m\gtrsim n^{1/(\tilde{s}\beta)} .
  \end{cases}
\]
With this choice, the upper rate becomes
\begin{equation}
  \label{eq:optimal-upper}
  \operatorname{MSE}_{r,\lambda_\ast}(\mca T)
  \lesssim_{\bbP}
  \frac{1}{n}
  +
  (mn)^{-\tilde{s}\beta/(\tilde{s}\beta+1)}
  \lesssim
  \begin{cases}
    (mn)^{-\tilde{s}\beta/(\tilde{s}\beta+1)},
  & m\lesssim n^{1/(\tilde{s}\beta)}
    \quad\text{(sparse)},               \\[0.4em]
    n^{-1},
    & m\gtrsim n^{1/(\tilde{s}\beta)}
    \quad\text{(dense)} .
  \end{cases}
\end{equation}

This rate is known to be minimax optimal for \(r=1\) and \(r=2\)~\citep{cai2011OptimalEstimationMean,cai2010NonparametricCovarianceFunction}, as long as \( s \leq 2 \).
Moreover, we have the following minimax lower bound for arbitrary fixed \(r\ge1\).
Let \( \mca P_{r,s}(R) \)
denote the class of probability measures \(P\) of \( X \) such that \( X\in[H]^s \) almost surely and
\( \E_P\norm{X}_{[H]^s}^{2r}\le R^{2r} \).
Write \(\mu_r(P)\coloneqq\E_P X^{\otimes r}\).

\begin{theorem}
  \label{thm:minimax-lower-bound}
  Suppose Assumption~\ref{ass:prelim-eigen-decay} holds.
  For fixed $r \geq 1$, \( s > 1/\beta \), \( R > 0 \), and \(\sigma^2>0\),
  there exists a constant \(c>0\) such that
  \[
    \inf_{\widehat\mu}
    \sup_{P\in\mca P_{r,s}(R)}
    \E_P
    \norm{
      \widehat\mu-\mu_r(P)
    }_{L^2(\mca I^r,\rho^{\otimes r})}^2
    \ge
    c \xk{\frac{1}{n}
    +
      (mn)^{-s\beta/(s\beta+1)}}
  \]
  for all sufficiently large \(n\), uniformly over integers \(m\ge r\). The
  infimum is over all estimators \( \widehat\mu \) based on the discretely observed curves, and
  \(\E_P\) averages over the latent curves, design points, and Gaussian errors.
\end{theorem}

Combining \cref{thm:minimax-lower-bound} and \cref{eq:optimal-upper},
we see that the KRR estimator achieves the minimax rate as long as \(s\le2\).
When \( s>2 \), the KRR estimator saturates and is not optimal; see \Cref{subsec:discussion-sparse-saturation} for more details.
We also remark that,
thanks to our refined analysis,
there is no logarithmic gap between the upper and lower bounds, which is often present in the literature~\citep{cai2010NonparametricCovarianceFunction,gupta2026MinimaxOptimalEstimation}.

\subsection{Deterministic Equivalent}

\begin{figure}[t]
  \centering
  \includegraphics[width=0.9\linewidth]{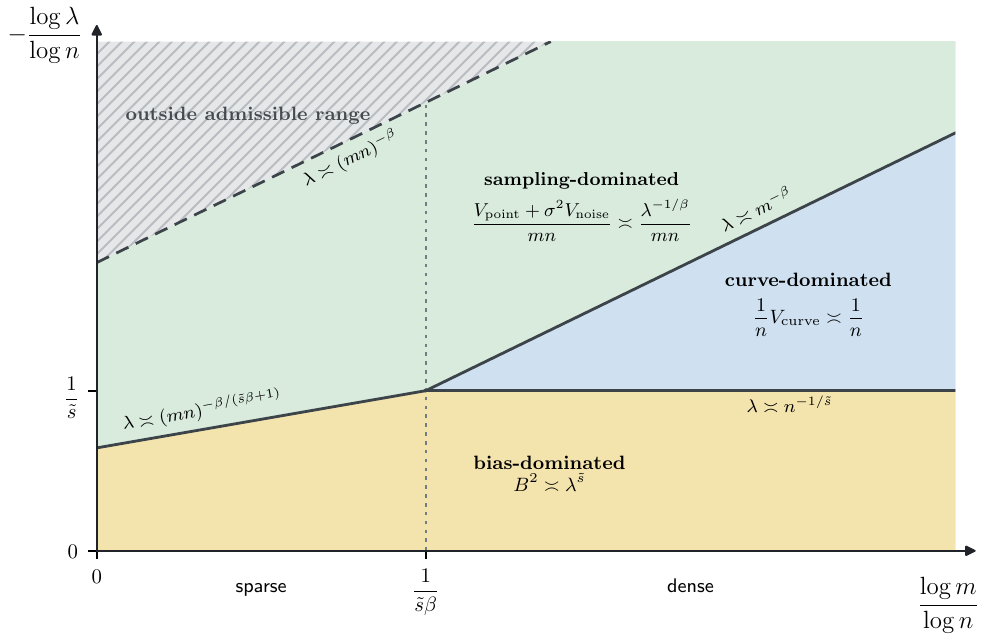}
  \caption{Phase diagram for the three leading terms in the MSE.}
  \label{fig:phase-diagram}
\end{figure}

The main result gives a deterministic equivalent for the
generalization error at each regularization parameter satisfying
\cref{eq:main-admissible-lambda}. In the noisy case this is
\cref{eq:main-exact-expansion}; in the noiseless case this is
\cref{eq:main-noiseless-expansion}.
The interpretation of the four terms in the expansion is as follows:
\begin{itemize}
  \item \(V_{\mathrm{curve}}\) measures variation across the \(n\)
  independent random functions \(X_i\);
  \item \(V_{\mathrm{point}}\) measures variation from the
  \(mn\) latent values \(X_i(t_{ij})\);
  \item \(V_{\mathrm{noise}}\) measures variation from the \(mn\)
  measurement errors \(\varepsilon_{ij}\);
  \item \(B^2\) is the deterministic equivalent of the squared bias.
\end{itemize}
It should be noted that \(V_{\mathrm{point}}\) is caused by the random function
\(X_i\) evaluated at the design points, rather than by the randomness of the design points
alone. When \(r=1\), \(V_{\mathrm{point}}\) is contributed by the diagonal covariance
\(C(t,t)=\Cov(X(t),X(t))\).
For general \(r\), it is contributed by the covariance due to one common sampling point between two product blocks.

Viewed as a function of \(\lambda\), the deterministic equivalent gives an asymptotically exact generalization error curve.
This curve shows how the error depends on the regularization parameter \(\lambda\), the number of curves \(n\), the number of observations per curve \(m\), and the noise level \(\sigma^2\).
It also shows which population objects enter the leading terms: \(\mu_r\) through the bias, \(\Sigma_r\)
through the randomness of the underlying process \(X\), and
\(\mu_{2r-2}\) through the measurement noise.
When \(r=1\), the convention \(\mu_0=1\) implies that \(V_{\mathrm{noise}}(\lambda;\mu_0)\) does not depend on the latent process \(X\), as in classical KRR \citep{li2023AsymptoticLearningCurves}.
As \(\lambda\) decreases, the
deterministic bias term \(B(\lambda;\mu_r)^2\) decreases, while the variance terms \(V_{\mathrm{point}}\) and \(V_{\mathrm{noise}}\) grow at the same power order.
Our result extends the exact generalization error results for kernel ridge regression in the classical regression setting \citep{li2023AsymptoticLearningCurves,li2024GeneralizationErrorCurves}, where the structure of discretely sampled functional data leads to several variance terms rather than a simple bias-variance trade-off.

The deterministic equivalent gives more information than a traditional minimax-rate statement in functional data analysis~\citep{cai2011OptimalEstimationMean,cai2010NonparametricCovarianceFunction}.
Minimax bounds identify the best order after optimizing \(\lambda\), but the leading constants and the separate variance sources are
usually no longer visible.
Here these constants are retained in the
deterministic coefficients \(V_{\mathrm{curve}}\), \(V_{\mathrm{point}}\),
\(V_{\mathrm{noise}}\), and \(B\).
Thus the curve gives an explicit bias--variance trade-off for every admissible \(\lambda\), not only the optimized rate.
Moreover, existing minimax results often mix variation from latent signals \(X_i(t_{ij})\) and variation from measurement errors \(\varepsilon_{ij}\) into one sampling term.
In comparison, our result delicately separates different sources of variation and makes the contribution of each source explicit.
In particular, even when \(\sigma^2=0\), the variation from latent signals remains.

\subsection{Saturation Effect}
\label{subsec:discussion-sparse-saturation}

The saturation effect of KRR is a well-known phenomenon in the classical regression setting \citep{bauer2007RegularizationAlgorithmsLearning,li2024SaturationEffectKernel}:
when the smoothness of the target function exceeds a certain threshold,
KRR fails to achieve the minimax rate no matter how the regularization parameter is chosen.
Using the deterministic equivalent,
we can show that a similar phenomenon exists for discretely observed functional data.

\begin{corollary}[Saturation effect]
  \label{cor:sparse-saturation}
  Suppose either Assumptions~\ref{ass:prelim-eigen-decay},
  \ref{ass:prelim-regular-rkhs}, and \ref{ass:source} hold with
  \(\sigma^2>0\), or Assumptions~\ref{ass:prelim-eigen-decay},
  \ref{ass:prelim-regular-rkhs}, \ref{ass:prelim-ae-lower-regularity}, and
  \ref{ass:source} hold with \(\sigma^2=0\). Suppose also that, for some
  \(s>2\), \( \mu_r \in ([H]^s)^{\otimes r}. \)
  If \(m\lesssim n^{1/(2\beta)}\), then for every \(\lambda=\lambda(n,m)\to0\) satisfying
  \cref{eq:main-admissible-lambda},
  \[
    \operatorname{MSE}_{r,\lambda}(\mca T)
    =
    \Omega_\bbP \left((mn)^{-2\beta/(2\beta+1)}\right).
  \]
\end{corollary}

For \(s>2\), \cref{cor:sparse-saturation} limits the KRR rate to \( (mn)^{-2\beta/(2\beta+1)} \),
while the minimax lower rate in \cref{thm:minimax-lower-bound} remains \( (mn)^{-s\beta/(s\beta+1)} \).
Hence KRR is no longer minimax optimal in the sparse regime when \(s>2\).
This saturation threshold \( s=2 \) reflects the qualification of Tikhonov regularization known in the literature \citep{bauer2007RegularizationAlgorithmsLearning,li2024SaturationEffectKernel}.
We also note that this effect is visible only in the sparse regime, since the dense regime is governed by the parametric rate \(n^{-1}\).

\subsection{Concentration Inequalities for U-Statistics}

For functional data,
the KRR estimator involves empirical quantities that are averages within curves, such as \cref{eq:hat-zeta,eq:hat-T}.
This naturally leads to concentration problems for \(U\)-statistics.
While \(U\)-statistics are well-studied, the existing results in the FDA literature are often too loose for our purpose of establishing a deterministic equivalent.
For example, naive control of the \(U\)-statistic structure often gives a polynomial dependence on the failure probability \(\delta\), which is too coarse.
Moreover, our analysis also involves operator-valued \(U\)-statistics, which are less studied in the literature.
To this end, we develop a series of \(U\)-statistic concentration inequalities.
The inequalities have the Bernstein form and cover scalar-valued, vector-valued, and operator-valued \(U\)-statistics of general order.
See Appendix~\ref{app:concentration} for more details.
These inequalities are of independent interest and may be useful in other problems in functional data analysis.

 \section{Simulations}
\label{sec:experiments}

In the numerical simulations, we consider the second-moment estimation problem with a one-dimensional Sobolev RKHS.
We demonstrate that our theoretical prediction in \Cref{thm:main} closely matches the empirical error curve.

\medskip\noindent\textbf{Experiment Setup.}
The input space is \(\mca I=[0,1]\) with the uniform measure.
We use the Brownian kernel \(k(x,y)=\min(x,y)\)~\citep{hsing2015TheoreticalFoundationsFunctional},
whose RKHS is the first-order Sobolev space \( H
    =
    \left\{
    f\in \mathrm{AC}[0,1]:
    f(0)=0,\ f'\in L^2[0,1]
    \right\} \).
For this kernel, the eigenvalues of the integral operator satisfy
\(\lambda_j \asymp j^{-2}\), so \(\beta=2\).

The random function is a centered rank-one Gaussian process.  Let
\(Z\sim N(0,1)\) and set
\[
    X(t)=Zg(t),
    \qquad
    g(t)=
    \sum_{j=1}^{\infty}
    \lambda_j^{s/2} j^{-1/2} e_j(t),
\]
where \((\lambda_j,e_j)_{j\ge1}\) is the Mercer system of \(T\) and \(s=0.7\).
The population model has \(\mu_1=0\), and the second moment equals the
covariance function
\[
    C(x,y)
    =
    \mu_2(x,y)
    =
    g(x)g(y)
    =
    \sum_{i,j=1}^{\infty} \lambda_i^{s/2} \lambda_j^{s/2}(ij)^{-1/2} e_i(x)e_j(y).
\]
Since \(s>1/\beta\), one can verify that \(X\) satisfies Assumption~\ref{ass:source}.
Discrete noisy observations are generated according to the model in \Cref{sec:preliminaries} with \(\sigma=1\).
According to \Cref{lem:appendix-I0-expectations,lem:appendix-I0-vpoint-lower,lem:appendix-I1-expectations,ex:bias-borderline-coefficients},
\[
    \frac{1}{n}V_{\mathrm{curve}}
    \asymp
    \frac{1}{n},
    \qquad
    \frac{1}{mn}V_{\mathrm{point}}
    \asymp
    \frac{\lambda^{-1/2}}{mn},
    \qquad
    \frac{\sigma^2}{mn}V_{\mathrm{noise}}
    \asymp
    \frac{\lambda^{-1/2}}{mn},
    \qquad
    B(\lambda;\mu_2)^2
    \asymp
    \lambda^{0.7} \ln\frac{1}{\lambda}.
\]

\medskip\noindent\textbf{Simulation.}
For each repetition, we fit the estimator for all values of \(\lambda\) over the same training data.
We use \(30\) regularization parameters between \(5\times10^{-5}\) and \(2\times10^{-2}\).
The empirical \(L^2\) error is computed by Simpson's rule.
\(50\) repetitions are performed for each configuration of \((n,m)\).
We consider two sample-size configurations: \((n,m)=(100,10)\) and \((n,m)=(10,40)\).
They reflect the sparse regime and the dense regime, respectively.

\begin{figure}[t]
    \centering
    \includegraphics[width=.9\linewidth]{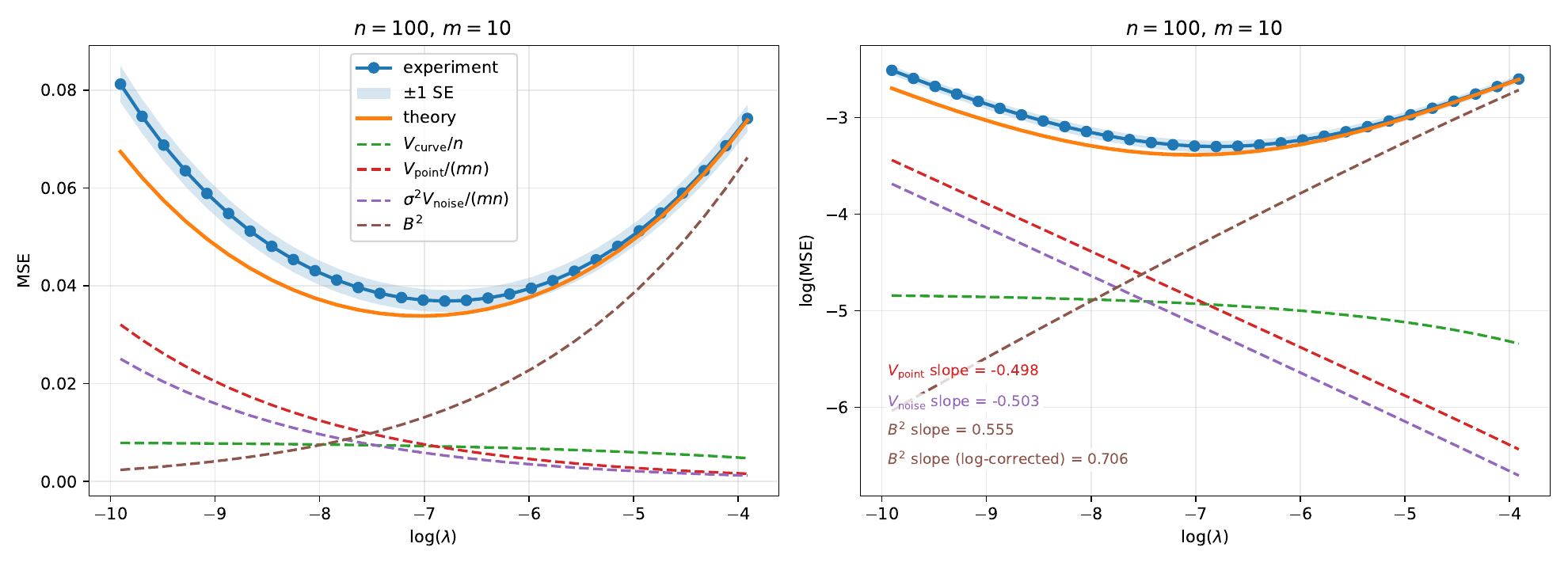}
    \includegraphics[width=.9\linewidth]{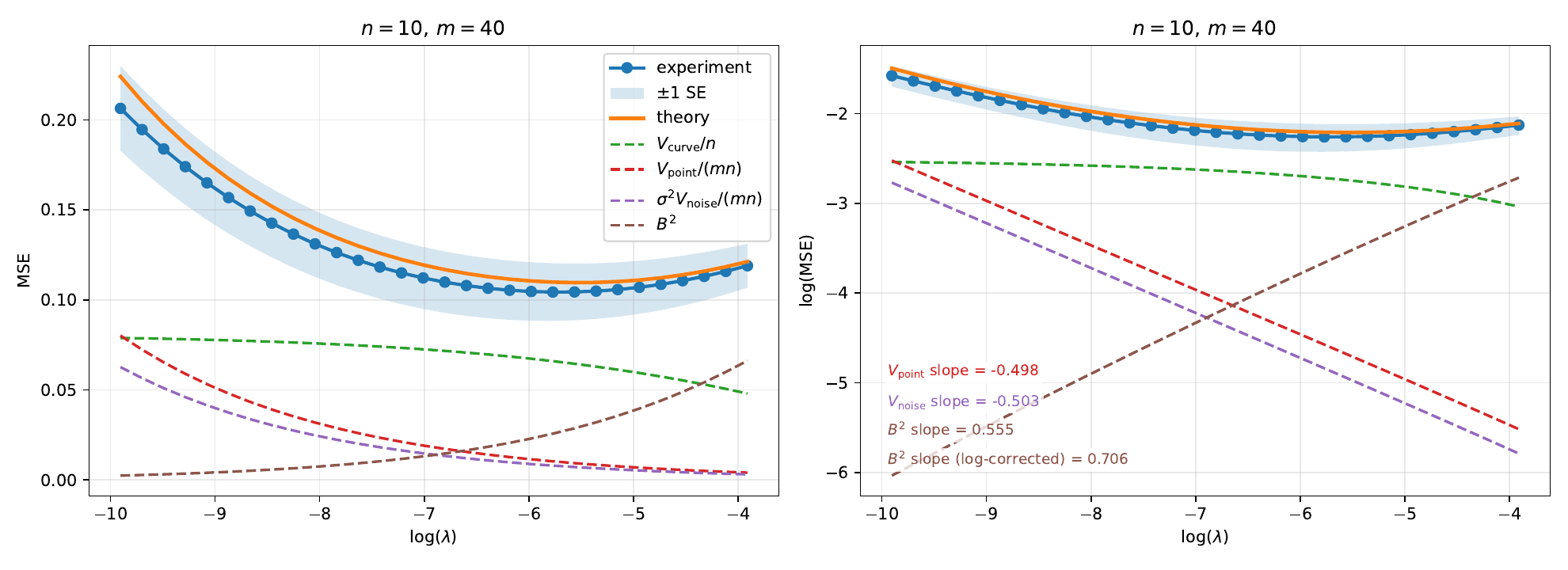}
    \caption{
        Generalization error curves for second-moment estimation.
        The upper and lower rows correspond to the sparse and dense sampling
        regimes, respectively.
        Within each row, the left and right panels use
        the same setting, and the right panel displays the vertical axis on a
        logarithmic scale for better visualization of the slopes.
    }
    \label{fig:r2-lambda-sweeps}
\end{figure}

The results are shown in \Cref{fig:r2-lambda-sweeps}.
The empirical error curve is U-shaped, and it is close to the deterministic curve given by \Cref{thm:main},
showing the predictive power of the asymptotically exact expansion.
Moreover, we fit the slopes of the log-log plots of the three variance terms and the bias term.
With \(\beta=2\), both
\(V_{\mathrm{point}}(\lambda;\Sigma_2)/(mn)\) and \(\sigma^2 V_{\mathrm{noise}}(\lambda;\mu_2)/(mn)\) should scale as
\(\lambda^{-1/\beta}=\lambda^{-1/2}\), and the fitted slopes in the log-log
plots are consistent with this prediction.
For the bias, if we fit the slope of \(\log B(\lambda;\mu_2)^2\) against \(\log\lambda\), the fitted exponent
is smaller than \(0.7\), because the theoretical order contains an additional logarithmic factor \(\ln(1/\lambda)\).
After fitting \(\log B(\lambda;\mu_2)^2-\log\log(1/\lambda)\) against \(\log\lambda\) to remove the logarithmic factor, the fitted exponent is approximately \(0.706\), matching the predicted smoothness index \(s=0.7\).

Comparing the two sampling regimes, we can see that the sparse sampling regime has a sharper U-shape than the dense sampling regime.
In the sparse regime, the variations from the latent signal and the measurement error are both large, so the total variance is dominated by these two terms.
In the dense regime,
the variations from the latent signal and the measurement error are reduced, while
\(V_{\mathrm{curve}}(\lambda;\Sigma_2)/n\) becomes the main variance
contribution over a large range of \(\lambda\).
Overall, these results demonstrate that the theoretical prediction in \Cref{thm:main} closely matches the empirical error curve.
 \section{Proofs}
\label{sec:proofs}

\subsection{Proof Overview}
\label{sec:proof-overview}
The proof begins with the traditional
bias-variance decomposition. Let
\begin{gather}
    \label{eq:tilde-zeta}
    \tilde{\zeta}_r
    \coloneqq
    \E[\hat{\zeta}_r \mid \mca{T}]
    =
    \frac{1}{n}\sum_{i=1}^{n}
    \frac{1}{\falling{m}{r}}\sum_{\bm{j}\in I_m^r}
    \mu_r(t_{i\bm{j}})K(t_{i\bm{j}}),\\
    \label{eq:tilde-mu}
    \tilde{\mu}_{r,\lambda}
    \coloneqq
    \E[\hat{\mu}_{r,\lambda} \mid \mca{T}]
    =
    (\hat{T}_r+\lambda I)^{-1}\tilde{\zeta}_r,
\end{gather}
where we use \(\E[Y_{i\mathbf{j}}\mid \mca{T}]=\mu_r(t_{i\mathbf{j}})\)
on the distinct coordinates \(\mathbf{j}\in I_m^r\). For brevity,
we write \(\hat{T}_{r,\lambda} \coloneqq \hat{T}_r+\lambda I\) and \(T_{r,\lambda} \coloneqq T_r+\lambda I\). Using the reproducing property, we have
\[
    \tilde{\zeta}_r
    =
    \hat{T}_r \mu_r,
    \qquad
    \tilde{\mu}_{r,\lambda}
    =
    \hat{T}_{r,\lambda}^{-1} \hat{T}_r \mu_r.
\]
Adding and subtracting \(\tilde{\mu}_{r,\lambda}\) gives the conditional
bias-variance decomposition
\begin{align}
    \operatorname{MSE}_{r,\lambda}(\mca{T})
     & =
    \E\bigl[
    \|\hat{\mu}_{r,\lambda}-\tilde{\mu}_{r,\lambda}\|_{L^2(\mca{I}^r)}^2
     \big|
    \mca{T}
    \bigr]
    +
    \|\tilde{\mu}_{r,\lambda}-\mu_r\|_{L^2(\mca{I}^r)}^2. \label{eq:bias-variance}
\end{align}
The cross term vanishes because
\(\E[\hat{\mu}_{r,\lambda}-\tilde{\mu}_{r,\lambda} \mid \mca{T}]=0\).
We denote the two terms in \cref{eq:bias-variance} by
\(\Var(\lambda)\) and \(\Bias(\lambda)^2\), respectively.
Substituting \cref{eq:hat-mu}, \cref{eq:hat-zeta}, \cref{eq:tilde-mu}, and \cref{eq:tilde-zeta} into \cref{eq:bias-variance} gives
\begin{equation}\label{eq:variance-master}
    \Var(\lambda)
    =
    \E\left[
    \left\|
    \frac{1}{n}\sum_{i=1}^n \frac{1}{(m)_r}\sum_{\mathbf{j}\in I_m^r}
    \bigl(Y_{i\mathbf{j}}-\mu_r(t_{i\mathbf{j}})\bigr)
    \hat{T}_{r,\lambda}^{-1} K(t_{i\mathbf{j}})
    \right\|_{L^2(\mca{I}^r)}^2
     \middle|
    \mca{T}
    \right],
\end{equation}
\begin{equation}\label{eq:bias-master}
    \Bias(\lambda)^2
    =
    \|\lambda\hat{T}_{r,\lambda}^{-1} \mu_r\|_{L^2(\mca{I}^r)}^2.
\end{equation}
The remaining proof is organized as follows. First, \Cref{sec:variance-expansion}
expands the variance term \(\Var(\lambda)\) in powers of \(\sigma^2\):
\[
    \Var(\lambda)=\sum_{d=0}^r \hat{V}_d,
\]
and shows that the terms with \(d\ge 2\) are higher-order and negligible.
\Cref{sec:variance-v0,sec:variance-v1} estimate the two leading
variance terms \(\hat{V}_0\) and \(\hat{V}_1\),
which represent the variance term in the noiseless case
and the first-order correction to the variance term due to noise, respectively.
We convert \(\hat{V}_0\) and \(\hat{V}_1\) into population traces, which
turns long sums into compact operator quantities and makes the variance analysis
more structured. \Cref{sec:variance-summary} summarizes the variance estimates,
and \Cref{sec:bias-term} estimates the bias term.
The appendices provide the supporting tools: \Cref{app:concentration}
for \(U\)-statistic concentration, \Cref{app:regular-rkhs} for estimates
associated with regular RKHSs, \Cref{app:i0,app:i1} for the detailed trace bounds
corresponding to \(\hat{V}_0\) and \(\hat{V}_1\),
\Cref{app:omitted-proofs} for omitted arguments in the main text,
and \Cref{app:notation} for notation.

\subsection{Variance Expansion in Powers of \texorpdfstring{\(\sigma^2\)}{sigma2}}
\label{sec:variance-expansion}

We now return to the conditional variance term
\[
    \Var(\lambda)
    =
    \E\left[
    \left\|
    \frac{1}{n}\sum_{i=1}^n \frac{1}{(m)_r}\sum_{\mathbf{j}\in I_m^r}
    \bigl(Y_{i\mathbf{j}}-\mu_r(t_{i\mathbf{j}})\bigr)
    \hat{T}_{r,\lambda}^{-1} K(t_{i\mathbf{j}})
    \right\|_{L^2(\mca{I}^r)}^2
     \middle|
    \mca{T}
    \right].
\]
Expanding the square and using the independence across curves,
\begin{equation}\label{eq:variance-double-sum}
    \Var(\lambda)
    =
    \frac{1}{n^2(m)_r^2}
    \sum_{i=1}^n
    \sum_{\mathbf{j},\mathbf{k}\in I_m^r}
    E_i(\mathbf{j},\mathbf{k})
    \left\langle
    \hat{T}_{r,\lambda}^{-1} K(t_{i\mathbf{j}}),
    \hat{T}_{r,\lambda}^{-1} K(t_{i\mathbf{k}})
    \right\rangle_{L^2},
\end{equation}
where
\[
    E_i(\mathbf{j},\mathbf{k})
    \coloneqq
    \E\bigl[
        (Y_{i\mathbf{j}}-\mu_r(t_{i\mathbf{j}}))
        (Y_{i\mathbf{k}}-\mu_r(t_{i\mathbf{k}}))
         \big|
        \mca{T}
        \bigr].
\]
Fix \(i\), \(\mathbf{j}=(j_1,\dots,j_r)\in I_m^r\), and \(\mathbf{k}=(k_1,\dots,k_r)\in I_m^r\).
Recall that
\[
    y_{ij}=X_{i}(t_{ij})+\ep_{ij},
    \qquad
    Y_{i\mathbf{j}}=\prod_{a=1}^r y_{ij_a},
    \qquad
    Y_{i\mathbf{k}}=\prod_{b=1}^r y_{ik_b}.
\]
Let
\[
    J(\mathbf{j}) \coloneqq \{j_1,\dots,j_r\},
    \qquad
    J(\mathbf{k}) \coloneqq \{k_1,\dots,k_r\},
    \qquad
    q(\mathbf{j},\mathbf{k}) \coloneqq |J(\mathbf{j})\cap J(\mathbf{k})|.
\]
For a subset \(A\subseteq J(\mathbf{j})\cap J(\mathbf{k})\), let
\[
    \bm{t}^{A}_{i,\mathbf{j},\mathbf{k}}
\]
denote the \((2r-2|A|)\)-tuple obtained from
\[
    (t_{ij_1},\dots,t_{ij_r},t_{ik_1},\dots,t_{ik_r})
\]
by deleting both copies of \(t_{i\ell}\) for every \(\ell\in A\).
Because \(\mu_q\) is symmetric in its arguments, the order of this tuple is irrelevant.
Expanding the products according to the repeated noise variables gives the
following formula.

\begin{lemma}[Expansion of \(E_i(\mathbf{j},\mathbf{k})\)]
    \label{lem:variance-E-expansion}
    With the notation above,
    \begin{equation}\label{eq:b-polynomial}
        E_i(\mathbf{j},\mathbf{k})
        =
        \Sigma_r(t_{i\mathbf{j}};t_{i\mathbf{k}})
        +
        \sum_{d=1}^{q(\mathbf{j},\mathbf{k})}
        \sigma^{2d}
        \sum_{\substack{A\subseteq J(\mathbf{j})\cap J(\mathbf{k})\\|A|=d}}
        \mu_{2r-2d} \bigl(\bm{t}^{A}_{i,\mathbf{j},\mathbf{k}}\bigr).
    \end{equation}
\end{lemma}

Substituting \Cref{lem:variance-E-expansion} into \cref{eq:variance-double-sum}, we obtain the explicit decomposition
\begin{equation}\label{eq:variance-polynomial}
    \Var(\lambda)
    =
    \sum_{d=0}^r \hat{V}_d,
\end{equation}
where
\begin{equation}\label{eq:Id-definition}
    \hat{V}_0
    \coloneqq
    \frac{1}{n^2(m)_r^2}
    \sum_{i=1}^n
    \sum_{\mathbf{j},\mathbf{k}\in I_m^r}
    \Sigma_r(t_{i\mathbf{j}};t_{i\mathbf{k}})
    \left\langle
    \hat{T}_{r,\lambda}^{-1} K(t_{i\mathbf{j}}),
    \hat{T}_{r,\lambda}^{-1} K(t_{i\mathbf{k}})
    \right\rangle_{L^2},
\end{equation}
and, for \(1\le d\le r\),
\begin{equation}\label{eq:Id-positive-definition}
    \hat{V}_d
    \coloneqq
    \frac{\sigma^{2d}}{n^2(m)_r^2}
    \sum_{i=1}^n
    \sum_{\mathbf{j},\mathbf{k}\in I_m^r}
    \left(
    \sum_{\substack{A\subseteq J(\mathbf{j})\cap J(\mathbf{k})\\|A|=d}}
    \mu_{2r-2d} \bigl(\bm{t}^{A}_{i,\mathbf{j},\mathbf{k}}\bigr)
    \right)
    \left\langle
    \hat{T}_{r,\lambda}^{-1} K(t_{i\mathbf{j}}),
    \hat{T}_{r,\lambda}^{-1} K(t_{i\mathbf{k}})
    \right\rangle_{L^2}.
\end{equation}
This is the exact meaning of the notation:
\(\hat{V}_d\) is precisely the contribution of order \(\sigma^{2d}\) in the variance expansion.

The leading terms of the variance will come from \(d=0\) and \(d=1\).
For the terms with \(d\ge 2\), a rough estimate is enough.

Set \(C_{\mu_0} \coloneqq 1\).  For even \(q\) with \(2\le q\le 2r-2\), use the
constants \(C_{\mu_q}\) from
\Cref{ass:source}\textup{(iii)}--\textup{(iv)}.

Now fix \(d\ge 2\).
If \(q(\mathbf{j},\mathbf{k})<d\), then the inner sum in \cref{eq:Id-positive-definition} is empty.
If \(q(\mathbf{j},\mathbf{k})\ge d\), then there are exactly \(\binom{q(\mathbf{j},\mathbf{k})}{d}\) subsets \(A\subseteq J(\mathbf{j})\cap J(\mathbf{k})\) with \(|A|=d\), so
\begin{equation}\label{eq:overlap-d-bound}
    \left|
    \sum_{\substack{A\subseteq J(\mathbf{j})\cap J(\mathbf{k})\\|A|=d}}
    \mu_{2r-2d} \bigl(\bm{t}^{A}_{i,\mathbf{j},\mathbf{k}}\bigr)
    \right|
    \le
    \binom{q(\mathbf{j},\mathbf{k})}{d} C_{\mu_{2r-2d}}.
\end{equation}
Write
\[
    \mcaN_p^{(r)}(\lambda)
    \coloneqq
    \sum_{i_1,\dots,i_r=1}^\infty
    \bigl(\frac{\Lambda_{\mathbf{i}}}{\Lambda_{\mathbf{i}}+\lambda}\bigr)^p,
    \qquad
    \Lambda_{\mathbf{i}} \coloneqq \prod_{a=1}^r \lambda_{i_a},
\]
we have the following bound, which is the empirical counterpart of
\Cref{lem:appendix-phi-T}.
\begin{lemma}\label{lem:variance-phi-hatT}
    If \(v_r \coloneqq \frac{\lambda^{-1/\beta}(\ln \frac{1}{\lambda})^{r}}{mn}\to 0\),
    then with probability tending to one,
    \begin{equation}\label{eq:phi-hatT-bound-main}
        \sup_{\bm{t}\in \mca{I}^r}
        \|\hat{T}_{r,\lambda}^{-1} K(\bm{t})\|_{L^2}^2
        \le
        C M^r \mcaN_2^{(r)}(\lambda).
    \end{equation}
\end{lemma}

If \(\lambda=\Omega((mn)^{-\theta})\) for some \(\theta<\beta\), then
\[
    v_r
    =
    O \left(
    (mn)^{\theta/\beta-1}
    \left(\ln(mn)\right)^r
    \right)
    =
    o(1),
\]
so this condition is a convenient sufficient condition for
\(v_r \to0\).

Using \cref{eq:overlap-d-bound} and \cref{eq:phi-hatT-bound-main} in \cref{eq:Id-positive-definition}, we get, for \(d\ge 2\), with probability tending to one,
\[
    |\hat{V}_d|
    \le
    \frac{\sigma^{2d} C_{\mu_{2r-2d}} C M^r \mcaN_2^{(r)}(\lambda)}{n^2(m)_r^2}
    \sum_{i=1}^n
    \sum_{\mathbf{j},\mathbf{k}\in I_m^r}
    \binom{q(\mathbf{j},\mathbf{k})}{d}.
\]
The sum over \(i\) contributes a factor \(n\), hence
\begin{equation}\label{eq:Id-count-start}
    |\hat{V}_d|
    \le
    \frac{\sigma^{2d} C_{\mu_{2r-2d}} C M^r \mcaN_2^{(r)}(\lambda)}{n(m)_r^2}
    \sum_{\mathbf{j},\mathbf{k}\in I_m^r}
    \binom{q(\mathbf{j},\mathbf{k})}{d}.
\end{equation}

\begin{lemma}\label{lem:variance-overlap-count}
    For every \(0\le d\le r\),
    \begin{equation}\label{eq:variance-overlap-count}
        \sum_{\mathbf{j},\mathbf{k}\in I_m^r}
        \binom{q(\mathbf{j},\mathbf{k})}{d}
        =
        \sum_{q=d}^r
        \binom{r}{q}^2 q! (m)_{2r-q}\binom{q}{d}.
    \end{equation}
\end{lemma}

Since \(q\ge d\) in every summand on the right-hand side of
\cref{eq:variance-overlap-count} implies
\((m)_{2r-q}\le (m)_{2r-d}\), this gives
\[
    \sum_{\mathbf{j},\mathbf{k}\in I_m^r}
    \binom{q(\mathbf{j},\mathbf{k})}{d}
    \le
    C_r (m)_{2r-d},
\]
where \(C_r\) depends only on \(r\).
Plugging this into \cref{eq:Id-count-start} on the same high-probability event,
\[
    |\hat{V}_d|
    \le
    \frac{C_r \sigma^{2d} C_{\mu_{2r-2d}} M^r \mcaN_2^{(r)}(\lambda)}{n}
    \frac{(m)_{2r-d}}{(m)_r^2}.
\]
Now
\[
    \frac{(m)_{2r-d}}{(m)_r^2}
    =
    \frac{(m-r)_{r-d}}{(m)_r}
    =
    \frac{1}{(m)_d}
    \frac{(m-r)_{r-d}}{(m-d)_{r-d}}
    \le
    \frac{1}{(m)_d},
\]
so for \(d\ge 2\),
\[
    \frac{(m)_{2r-d}}{(m)_r^2}
    \le
    \frac{1}{m(m-1)}.
\]
Hence
\begin{equation}\label{eq:Id-high-order}
    |\hat{V}_d|
    \le
    \frac{C_r \sigma^{2d} C_{\mu_{2r-2d}} M^r \mcaN_2^{(r)}(\lambda)}{n m(m-1)},
    \qquad
    d\ge 2.
\end{equation}

Summing \cref{eq:Id-high-order} over \(d=2,\dots,r\) and using that \(r\) is fixed,
\[
    \sum_{d=2}^r |\hat{V}_d|
    \le
    \frac{C_r M^r \mcaN_2^{(r)}(\lambda)}{n m(m-1)}
    \sum_{d=2}^r \sigma^{2d} C_{\mu_{2r-2d}}.
\]
Finally, \(\mcaN_2^{(r)}(\lambda)\asymp \lambda^{-1/\beta}(\ln \frac{1}{\lambda})^{r-1}\); see \Cref{lem:appendix-effective-dimension}.
Under the assumptions
\[
    v_r \to 0,
    \qquad
    (\ln \tfrac{1}{\lambda})^{r-1}=o(m),
\]
we therefore obtain
\begin{equation}\label{eq:high-order-negligible}
    \sum_{d=2}^r |\hat{V}_d|
    =
    O_\bbP \left(
    \frac{\lambda^{-1/\beta}(\ln \frac{1}{\lambda})^{r-1}}{n m(m-1)}
    \right)
    =
    o_\bbP \left(
    \frac{\lambda^{-1/\beta}}{mn}
    \right).
\end{equation}
So the only terms that can contribute at the leading order are \(\hat{V}_0\) and \(\hat{V}_1\).

\subsection{Estimating the Noise-free Term \texorpdfstring{\(\hat{V}_0\)}{V0}}
\label{sec:variance-v0}

The following inequality was also used by \citet{li2024GeneralizationErrorCurves}.
\begin{lemma}\label{lem:appendix-operator-concentration}
    Let \(T_{r,\lambda} \coloneqq T_r+\lambda I\).  If \(\lambda\to0\) satisfies
    \(v_r=\frac{\lambda^{-1/\beta}(\ln \frac{1}{\lambda})^{r}}{mn}\to 0\), then
    \[
        \|T_{r,\lambda}^{-1/2}(\hat{T}_r-T_r)T_{r,\lambda}^{-1/2}\|
        =
        O_\bbP(\sqrt{v_r}),
    \]
    and, with probability tending to one,
    \[
        \|T_{r,\lambda}^{-1/2}(\hat{T}_r-T_r)T_{r,\lambda}^{-1/2}\|
        \le \frac{1}{2},
        \qquad
        \|T_{r,\lambda}^{1/2} \hat{T}_{r,\lambda}^{-1} T_{r,\lambda}^{1/2}\|
        \le 2.
    \]
\end{lemma}

We now return to the coefficient of \(\sigma^0\) from \Cref{eq:variance-polynomial}.
By definition,
\[
    \hat{V}_0
    =
    \frac{1}{n^2(m)_r^2}
    \sum_{i=1}^n
    \sum_{\mathbf{j},\mathbf{k}\in I_m^r}
    \Sigma_r(t_{i\mathbf{j}};t_{i\mathbf{k}})
    \left\langle
    \hat{T}_{r,\lambda}^{-1} K(t_{i\mathbf{j}}),
    \hat{T}_{r,\lambda}^{-1} K(t_{i\mathbf{k}})
    \right\rangle_{L^2}.
\]
It is convenient to package the covariance kernel into an empirical operator.
Define
\[
    \hat{G}_0
    \coloneqq
    \frac{1}{n(m)_r^2}
    \sum_{i=1}^n
    \sum_{\mathbf{j},\mathbf{k}\in I_m^r}
    \Sigma_r(t_{i\mathbf{j}};t_{i\mathbf{k}})
    K(t_{i\mathbf{k}})\otimes K(t_{i\mathbf{j}}),
    \qquad
    G_0 \coloneqq \E[\hat{G}_0],
\]
where \(u\otimes v\) denotes the rank-one operator \(f\mapsto \ang{f,v}_{\caH^{\otimes r}}u\).

\begin{lemma}\label{lem:I0-G0-psd}
    The operators \(\hat{G}_0\) and \(G_0\) are self-adjoint and positive semidefinite.
\end{lemma}

\begin{proof}
    Since \(\Sigma_r(\bm{t};\bm{t}')=\Sigma_r(\bm{t}';\bm{t})\), the definition of \(\hat{G}_0\) is symmetric in \(\mathbf{j}\) and \(\mathbf{k}\), so \(\hat{G}_0\) is self-adjoint.
    For any \(f\in \caH^{\otimes r}\),
    \[
        \ang{\hat{G}_0 f,f}_{\caH^{\otimes r}}
        =
        \frac{1}{n}
        \sum_{i=1}^n
        \frac{1}{(m)_r^2}
        \sum_{\mathbf{j},\mathbf{k}\in I_m^r}
        \Sigma_r(t_{i\mathbf{j}};t_{i\mathbf{k}})
        f(t_{i\mathbf{j}})f(t_{i\mathbf{k}}).
    \]
    If we set
    \[
        Z_i(f)
        \coloneqq
        \frac{1}{(m)_r}
        \sum_{\mathbf{j}\in I_m^r}
        f(t_{i\mathbf{j}})
        \bigl(
        X_i(t_{ij_1})\cdots X_i(t_{ij_r})
        -
        \mu_r(t_{i\mathbf{j}})
        \bigr),
    \]
    then the right-hand side is exactly
    \[
        \frac{1}{n}\sum_{i=1}^n \E[Z_i(f)^2\mid \mca{T}]
        \ge 0.
    \]
    Hence \(\hat{G}_0\) is positive semidefinite.
    The same properties pass to \(G_0=\E[\hat{G}_0]\).
\end{proof}

With this notation, \(\hat{V}_0\) has the operator form
\begin{equation}\label{eq:I0-operator-form}
    \hat{V}_0
    =
    \frac{1}{n}
    \Tr\bigl(
    T_r^{1/2}
    \hat{T}_{r,\lambda}^{-1}
    \hat{G}_0
    \hat{T}_{r,\lambda}^{-1}
    T_r^{1/2}
    \bigr).
\end{equation}
To separate the empirical covariance part from the empirical sampling operator, we introduce
\begin{equation}\label{eq:I0-tilde-pop}
    \tilde{V}_0
    \coloneqq
    \frac{1}{n}
    \Tr\bigl(
    T_r^{1/2}
    T_{r,\lambda}^{-1}
    \hat{G}_0
    T_{r,\lambda}^{-1}
    T_r^{1/2}
    \bigr),
    \qquad
    V_0
    \coloneqq
    \frac{1}{n}
    \Tr\bigl(
    T_r^{1/2}
    T_{r,\lambda}^{-1}
    G_0
    T_{r,\lambda}^{-1}
    T_r^{1/2}
    \bigr).
\end{equation}
Because \(\hat{G}_0\) and \(G_0\) are positive semidefinite, these traces can also be written as Hilbert-Schmidt norms:
\[
    \hat{V}_0
    =
    \frac{1}{n}
    \|T_r^{1/2} \hat{T}_{r,\lambda}^{-1} \hat{G}_0^{1/2}\|_{HS}^2,
    \qquad
    \tilde{V}_0
    =
    \frac{1}{n}
    \|T_r^{1/2} T_{r,\lambda}^{-1} \hat{G}_0^{1/2}\|_{HS}^2,
    \qquad
    V_0
    =
    \frac{1}{n}
    \|T_r^{1/2} T_{r,\lambda}^{-1} G_0^{1/2}\|_{HS}^2.
\]

The comparison naturally splits into two pieces:
\begin{equation}\label{eq:I0-basic-split}
    \hat{V}_0-V_0
    =
    (\hat{V}_0-\tilde{V}_0)
    +
    (\tilde{V}_0-V_0).
\end{equation}
The second term already has the form
\begin{equation}\label{eq:I0-second-gap}
    \tilde{V}_0-V_0
    =
    \frac{1}{n}
    \Tr\bigl(
    T_r^{1/2}
    T_{r,\lambda}^{-1}
    (\hat{G}_0-G_0)
    T_{r,\lambda}^{-1}
    T_r^{1/2}
    \bigr).
\end{equation}
For the first term, the resolvent identity gives
\[
    \hat{T}_{r,\lambda}^{-1}-T_{r,\lambda}^{-1}
    =
    \hat{T}_{r,\lambda}^{-1}(T_r-\hat{T}_r)T_{r,\lambda}^{-1}.
\]
Hence
\begin{align}
    |\hat{V}_0-\tilde{V}_0|
    \le\  &
    \frac{1}{n}
    \|T_r^{1/2}(\hat{T}_{r,\lambda}^{-1}-T_{r,\lambda}^{-1})\hat{G}_0^{1/2}\|_{HS}
    \notag         \\
          & \times
    \Bigl(
    \|T_r^{1/2} \hat{T}_{r,\lambda}^{-1} \hat{G}_0^{1/2}\|_{HS}
    +
    \|T_r^{1/2} T_{r,\lambda}^{-1} \hat{G}_0^{1/2}\|_{HS}
    \Bigr). \label{eq:I0-first-gap-start}
\end{align}
Set
\[
    \Delta_r
    \coloneqq
    T_{r,\lambda}^{-1/2}(\hat{T}_r-T_r)T_{r,\lambda}^{-1/2},
    \qquad
    H_0
    \coloneqq
    T_{r,\lambda}^{-1/2} \hat{G}_0^{1/2}.
\]
Then
\[
    T_r^{1/2}(\hat{T}_{r,\lambda}^{-1}-T_{r,\lambda}^{-1})\hat{G}_0^{1/2}
    =
    T_r^{1/2} T_{r,\lambda}^{-1/2}
    \Delta_r
    T_{r,\lambda}^{1/2} \hat{T}_{r,\lambda}^{-1} T_{r,\lambda}^{1/2}
    H_0.
\]
Therefore, by \(T_r^{1/2} T_{r,\lambda}^{-1/2} \le I\),
\[
    \|T_r^{1/2}(\hat{T}_{r,\lambda}^{-1}-T_{r,\lambda}^{-1})\hat{G}_0^{1/2}\|_{HS}
    \le
    \|\Delta_r\|
    \|T_{r,\lambda}^{1/2} \hat{T}_{r,\lambda}^{-1} T_{r,\lambda}^{1/2}\|
    \|H_0\|_{HS}.
\]
By \Cref{lem:appendix-operator-concentration}, when $v_r \to0$, with probability tending to one we are on the event
\[
    \|\Delta_r\|\le \frac{1}{2},
    \qquad
    \|T_{r,\lambda}^{1/2} \hat{T}_{r,\lambda}^{-1} T_{r,\lambda}^{1/2}\|\le 2.
\]
Hence
\begin{equation}\label{eq:I0-difference-HS}
    \|T_r^{1/2}(\hat{T}_{r,\lambda}^{-1}-T_{r,\lambda}^{-1})\hat{G}_0^{1/2}\|_{HS}
    \le
    2\|\Delta_r\| \|H_0\|_{HS}.
\end{equation}
Also,
\begin{align}
    \|T_r^{1/2} T_{r,\lambda}^{-1} \hat{G}_0^{1/2}\|_{HS}
     & =
    \|T_r^{1/2} T_{r,\lambda}^{-1/2} H_0\|_{HS}
    \notag \\
     & \le
    \|H_0\|_{HS}, \label{eq:I0-base-HS}
\end{align}
and therefore
\begin{align}
    \|T_r^{1/2} \hat{T}_{r,\lambda}^{-1} \hat{G}_0^{1/2}\|_{HS}
     & \le
    \|T_r^{1/2}(\hat{T}_{r,\lambda}^{-1}-T_{r,\lambda}^{-1})\hat{G}_0^{1/2}\|_{HS}
    +
    \|T_r^{1/2} T_{r,\lambda}^{-1} \hat{G}_0^{1/2}\|_{HS}
    \notag \\
     & \le
    (2\|\Delta_r\|+1)\|H_0\|_{HS}. \label{eq:I0-hat-HS}
\end{align}
Substituting \Cref{eq:I0-difference-HS,eq:I0-base-HS,eq:I0-hat-HS} into \Cref{eq:I0-first-gap-start}, we get
\begin{align*}
    |\hat{V}_0-\tilde{V}_0|
     & \le
    \frac{1}{n}
    \bigl(
    2\|\Delta_r\| \|H_0\|_{HS}
    \bigr)
    \bigl(
    (2\|\Delta_r\|+1)\|H_0\|_{HS}
    +
    \|H_0\|_{HS}
    \bigr) \\
     & =
    \frac{1}{n}
    4\|\Delta_r\|(\|\Delta_r\|+1)\|H_0\|_{HS}^2.
\end{align*}
Since \(\|\Delta_r\|\le 1/2\) on the same event, \(4\|\Delta_r\|(\|\Delta_r\|+1)\le 6\|\Delta_r\|\). Using \(\|H_0\|_{HS}^2=\Tr(T_{r,\lambda}^{-1/2} \hat{G}_0 T_{r,\lambda}^{-1/2})\), we obtain
\begin{equation}\label{eq:I0-first-gap-trace}
    |\hat{V}_0-\tilde{V}_0|
    \le
    \frac{C}{n}
    \|\Delta_r\|
    \Tr\bigl(
    T_{r,\lambda}^{-1/2} \hat{G}_0 T_{r,\lambda}^{-1/2}
    \bigr).
\end{equation}
Finally,
\begin{equation}\label{eq:I0-trace-split}
    \Tr\bigl(
    T_{r,\lambda}^{-1/2} \hat{G}_0 T_{r,\lambda}^{-1/2}
    \bigr)
    \le
    \Tr\bigl(
    T_{r,\lambda}^{-1/2} G_0 T_{r,\lambda}^{-1/2}
    \bigr)
    +
    \left|
    \Tr\bigl(
    T_{r,\lambda}^{-1/2}
    (\hat{G}_0-G_0)
    T_{r,\lambda}^{-1/2}
    \bigr)
    \right|.
\end{equation}

\begin{proposition}\label{prop:I0-main}
    Assume \(\Sigma_r \neq 0\), suppose $\lambda\to0$ satisfies
    \[
        \left(\ln \frac{1}{\lambda}\right)^{r-1}=o(m),
        \qquad
        v_r \to 0.
    \]
    Then
    \begin{gather*}
        V_0
        =
        \frac{1}{n}V_{\mathrm{curve}}(\lambda;\Sigma_r)
        +
        \frac{1}{mn}V_{\mathrm{point}}(\lambda;\Sigma_r)
        +
        o \left(
        \frac{\lambda^{-1/\beta}}{mn}
        \right),    \\
        \frac{1}{n}V_{\mathrm{curve}}(\lambda;\Sigma_r)
        \asymp
        \frac{1}{n},
        \quad
        \frac{1}{mn}V_{\mathrm{point}}(\lambda;\Sigma_r)
        =
        O \left(
        \frac{\lambda^{-1/\beta}}{mn}
        \right),     \\
        \hat{V}_0
        =
        V_0
        +
        o_\bbP \left(
        \frac{1}{n}
        +
        \frac{\lambda^{-1/\beta}}{mn}
        \right).
    \end{gather*}
\end{proposition}

\begin{proof}
    By \Cref{prop:appendix-I0-pop-phi},
    \[
        V_0
        =
        \frac{1}{n}V_{\mathrm{curve}}(\lambda;\Sigma_r)
        +
        \frac{1}{mn}V_{\mathrm{point}}(\lambda;\Sigma_r)
        +
        o \left(
        \frac{\lambda^{-1/\beta}}{mn}
        \right).
    \]
    Since \(\Sigma_r \neq 0\) and \(\lambda\to0\), the expectation bounds in
    \Cref{lem:appendix-I0-expectations}, together with the definition of
    \(V_{\mathrm{curve}}\), give
    \[
        \frac{1}{n}V_{\mathrm{curve}}(\lambda;\Sigma_r)
        \asymp
        \frac{1}{n}.
    \]
    The corresponding one-overlap bound and the definition of
    \(V_{\mathrm{point}}\) give
    \[
        \frac{1}{mn}V_{\mathrm{point}}(\lambda;\Sigma_r)
        =
        O \left(
        \frac{\lambda^{-1/\beta}}{mn}
        \right).
    \]

    Since \(\lambda\to0\), we have \(\ln \frac{1}{\lambda}\to\infty\). Therefore
    \begin{equation}\label{eq:I0-vr-log-over-mn}
        \frac{\lambda^{-1/\beta} \xk{\ln \frac{1}{\lambda}}^{r-1}}{mn}
        =
        \frac{v_r}{\ln \frac{1}{\lambda}}
        \to 0.
    \end{equation}
    Also,
    \begin{equation}\label{eq:I0-log-over-sqrtmn}
        \frac{\xk{\ln \frac{1}{\lambda}}^{r-1}}{\sqrt{mn}}
        =
        \lambda^{1/(2\beta)}
        \xk{\ln \frac{1}{\lambda}}^{r/2-1}
        \sqrt{v_r}
        \to 0,
    \end{equation}
    because \(v_r \to0\) and \(\lambda^c(\ln \frac{1}{\lambda})^a\to0\) as \(\lambda\to0\) for every \(c>0\) and every fixed \(a\in\R\).

    We next estimate the two gaps in \Cref{eq:I0-basic-split}.

    For \(\tilde{V}_0-V_0\), \Cref{prop:appendix-I0-fluctuation-phi}
    and \Cref{eq:I0-second-gap} imply
    \begin{align*}
        |\tilde{V}_0-V_0|
        =\  &
        \frac{1}{n}
        O_\bbP
        \Biggl(
        \frac{\lambda^{-1/(2\beta)}(\ln \frac{1}{\lambda})^{(r-1)/2}}{m^{1/2} n^{1/2}}
        +
        \lambda^{-1/\beta}
        \left(\ln \frac{1}{\lambda}\right)^{r-1}
        \left(
        \frac{1}{mn}
        +
        \frac{1}{m^{3/2} n^{1/2}}
        \right)
        \Biggr).
    \end{align*}
    The first term is
    \[
        \frac{1}{n}
        \sqrt{
            \frac{\lambda^{-1/\beta}(\ln \frac{1}{\lambda})^{r-1}}{mn}
        }
        =
        o \left(\frac{1}{n}\right)
    \]
    by \Cref{eq:I0-vr-log-over-mn}. The second term is
    \[
        \frac{1}{n}
        \frac{\lambda^{-1/\beta}(\ln \frac{1}{\lambda})^{r-1}}{mn}
        =
        o \left(\frac{1}{n}\right)
    \]
    again by \Cref{eq:I0-vr-log-over-mn}. For the third term,
    \[
        \frac{1}{n}
        \frac{\lambda^{-1/\beta}(\ln \frac{1}{\lambda})^{r-1}}{m^{3/2} n^{1/2}}
        =
        \frac{\lambda^{-1/\beta}}{mn}
        \frac{(\ln \frac{1}{\lambda})^{r-1}}{\sqrt{mn}}
        =
        o \left(
        \frac{\lambda^{-1/\beta}}{mn}
        \right)
    \]
    by \Cref{eq:I0-log-over-sqrtmn}. Hence
    \[
        \tilde{V}_0-V_0
        =
        o_\bbP \left(
        \frac{1}{n}
        +
        \frac{\lambda^{-1/\beta}}{mn}
        \right).
    \]

    We now estimate the first gap in \Cref{eq:I0-basic-split}, namely \(\hat{V}_0-\tilde{V}_0\).
    By \Cref{lem:appendix-operator-concentration},
    \[
        \|\Delta_r\|=O_\bbP(\sqrt{v_r}).
    \]
    Next, \Cref{prop:appendix-I0-pop-resolvent} gives
    \[
        \Tr\bigl(
        T_{r,\lambda}^{-1/2} G_0 T_{r,\lambda}^{-1/2}
        \bigr)
        =
        O \left(
        1+\frac{\lambda^{-1/\beta}}{m}
        \right),
    \]
    while \Cref{prop:appendix-I0-fluctuation-resolvent} gives
    \begin{align*}
            & \left|
        \Tr\bigl(
        T_{r,\lambda}^{-1/2}
        (\hat{G}_0-G_0)
        T_{r,\lambda}^{-1/2}
        \bigr)
        \right|      \\
        =\  &
        O_\bbP \left(
        \frac{\lambda^{-1/(2\beta)}(\ln \frac{1}{\lambda})^{(r-1)/2}}{m^{1/2} n^{1/2}}
        +
        \lambda^{-1/\beta}
        \left(\ln \frac{1}{\lambda}\right)^{r-1}
        \left(
        \frac{1}{mn}
        +
        \frac{1}{m^{3/2} n^{1/2}}
        \right)
        \right).
    \end{align*}
    The first term above is
    \[
        \sqrt{
            \frac{\lambda^{-1/\beta}(\ln \frac{1}{\lambda})^{r-1}}{mn}
        }
        =
        o(1)
    \]
    by \Cref{eq:I0-vr-log-over-mn}. The second term is
    \[
        \frac{\lambda^{-1/\beta}(\ln \frac{1}{\lambda})^{r-1}}{mn}
        =
        o(1)
    \]
    by \Cref{eq:I0-vr-log-over-mn}, and the last one is
    \[
        \frac{\lambda^{-1/\beta}}{m}
        \frac{(\ln \frac{1}{\lambda})^{r-1}}{\sqrt{mn}}
        =
        o \left(
        \frac{\lambda^{-1/\beta}}{m}
        \right)
    \]
    by \Cref{eq:I0-log-over-sqrtmn}. Hence
    \[
        \Tr\bigl(
        T_{r,\lambda}^{-1/2} \hat{G}_0 T_{r,\lambda}^{-1/2}
        \bigr)
        =
        O_\bbP \left(
        1+\frac{\lambda^{-1/\beta}}{m}
        \right).
    \]
    Substituting this into the high-probability inequality
    \Cref{eq:I0-first-gap-trace}, we obtain
    \[
        |\hat{V}_0-\tilde{V}_0|
        =
        O_\bbP \left(
        \sqrt{v_r}
        \left(
            \frac{1}{n}
            +
            \frac{\lambda^{-1/\beta}}{mn}
            \right)
        \right)
        =
        o_\bbP \left(
        \frac{1}{n}
        +
        \frac{\lambda^{-1/\beta}}{mn}
        \right).
    \]
    Combining the estimates for \(\tilde{V}_0-V_0\) and \(\hat{V}_0-\tilde{V}_0\) with \Cref{eq:I0-basic-split} proves the claim for \(\hat{V}_0\).
\end{proof}

\subsection{Estimating the First Noise Term \texorpdfstring{\(\hat{V}_1\)}{V1}}
\label{sec:variance-v1}

Throughout this subsection, we work in the noisy case \(\sigma^2>0\).
By definition,
\begin{equation}
    \label{eq:I1-definition}
    \hat{V}_1
    =
    \frac{\sigma^2}{n^2(m)_r^2}
    \sum_{i=1}^n
    \sum_{\mathbf{j},\mathbf{k}\in I_m^r}
    \left(
    \sum_{q\in J(\mathbf{j})\cap J(\mathbf{k})}
    \mu_{2r-2} \bigl(\bm{t}^{\{q\}}_{i,\mathbf{j},\mathbf{k}}\bigr)
    \right)
    \left\langle
    \hat{T}_{r,\lambda}^{-1} K(t_{i\mathbf{j}}),
    \hat{T}_{r,\lambda}^{-1} K(t_{i\mathbf{k}})
    \right\rangle_{L^2}.
\end{equation}
For \(\xi\in \mca{I}\) and \(\tilde{\bm{t}}=(t_1,\dots,t_{r-1})\in \mca{I}^{r-1}\), define the insertion sum
\begin{equation}\label{eq:I1-insertion-sum}
    S_r(\xi;\tilde{\bm{t}})
    \coloneqq
    \sum_{a=1}^r
    K(t_1,\dots,t_{a-1},\xi,t_a,\dots,t_{r-1})
    \in
    \caH^{\otimes r}.
\end{equation}
Also, for \(1\le l\le m\), let
\begin{equation}\label{eq:I1-index-set}
    I_{m,l}^{r-1}
    \coloneqq
    \{
    \mathbf{j}=(j_1,\dots,j_{r-1})\in I_m^{r-1}:\ l\notin J(\mathbf{j})
    \}.
\end{equation}
Then \(|I_{m,l}^{r-1}|=|I_{m-1}^{r-1}|=(m-1)_{r-1}\). For notational convenience, when
\[
    \tilde{\bm{t}}=(t_1,\dots,t_{r-1}),
    \qquad
    \tilde{\bm{t}}'=(t_1',\dots,t_{r-1}'),
\]
we write
\[
    \mu_{2r-2}(\tilde{\bm{t}};\tilde{\bm{t}}')
    \coloneqq
    \mu_{2r-2}(t_1,\dots,t_{r-1},t_1',\dots,t_{r-1}').
\]
In \Cref{eq:I1-definition}, the common location can be chosen as \(l\), while the remaining coordinates form two ordered \((r-1)\)-tuples in \(I_{m,l}^{r-1}\).
The shared point may appear in any position of each \(r\)-block, and \Cref{eq:I1-insertion-sum} sums over these possible insertions.
Therefore
\begin{equation}\label{eq:I1-raw-sum}
    \hat{V}_1
    =
    \frac{\sigma^2}{n^2 m^2(m-1)_{r-1}^2}
    \sum_{i=1}^n
    \sum_{l=1}^m
    \sum_{\mathbf{j},\mathbf{k}\in I_{m,l}^{r-1}}
    \hat{g}_{1,\lambda}
    \bigl(
    t_{il};
    t_{i\mathbf{j}};
    t_{i\mathbf{k}}
    \bigr),
\end{equation}
where
\begin{equation}\label{eq:I1-ghat-kernel}
    \hat{g}_{1,\lambda}(\xi;\tilde{\bm{t}};\tilde{\bm{t}}')
    \coloneqq
    \mu_{2r-2}(\tilde{\bm{t}};\tilde{\bm{t}}')
    \ang{
        \hat{T}_{r,\lambda}^{-1} S_r(\xi;\tilde{\bm{t}}),
        \hat{T}_{r,\lambda}^{-1} S_r(\xi;\tilde{\bm{t}}')
    }_{L^2}.
\end{equation}

As in the noise-free case, it is convenient to package the kernel into an empirical operator.
Define
\begin{equation}\label{eq:I1-g1-operator}
    \hat{G}_1
    \coloneqq
    \frac{1}{n m (m-1)_{r-1}^2}
    \sum_{i=1}^n
    \sum_{l=1}^m
    \sum_{\mathbf{j},\mathbf{k}\in I_{m,l}^{r-1}}
    \mu_{2r-2}(t_{i\mathbf{j}};t_{i\mathbf{k}})
    S_r(t_{il};t_{i\mathbf{k}})
    \otimes
    S_r(t_{il};t_{i\mathbf{j}}),
    \qquad
    G_1 \coloneqq \E[\hat{G}_1].
\end{equation}

\begin{lemma}\label{lem:I1-G1-psd}
    The operators \(\hat{G}_1\) and \(G_1\) are self-adjoint and positive semidefinite.
\end{lemma}

\begin{proof}
    The definition of \(\hat{G}_1\) is symmetric in \(\mathbf{j}\) and \(\mathbf{k}\), because \(\mu_{2r-2}(\tilde{\bm{t}};\tilde{\bm{t}}')=\mu_{2r-2}(\tilde{\bm{t}}';\tilde{\bm{t}})\).
    Hence \(\hat{G}_1\) is self-adjoint.
    For any \(f\in \caH^{\otimes r}\),
    \[
        \ang{\hat{G}_1 f,f}_{\caH^{\otimes r}}
        =
        \frac{1}{n m (m-1)_{r-1}^2}
        \sum_{i=1}^n
        \sum_{l=1}^m
        \sum_{\mathbf{j},\mathbf{k}\in I_{m,l}^{r-1}}
        \mu_{2r-2}(t_{i\mathbf{j}};t_{i\mathbf{k}})
        a_{i,l,\mathbf{j}}(f)a_{i,l,\mathbf{k}}(f),
    \]
    where
    \[
        a_{i,l,\mathbf{j}}(f)
        \coloneqq
        \ang{
        f,
        S_r(t_{il};t_{i\mathbf{j}})
        }_{\caH^{\otimes r}}.
    \]
    For each fixed \(i\) and \(l\), the matrix
    \[
        \bigl(
        \mu_{2r-2}(t_{i\mathbf{j}};t_{i\mathbf{k}})
        \bigr)_{\mathbf{j},\mathbf{k}\in I_{m,l}^{r-1}}
    \]
    is positive semidefinite.
    Indeed, if
    \[
        Z_{i,\mathbf{j}}
        \coloneqq
        X_i(t_{ij_1})\cdots X_i(t_{ij_{r-1}}),
        \qquad
        \mathbf{j}\in I_{m,l}^{r-1},
    \]
    then
    \[
        \mu_{2r-2}(t_{i\mathbf{j}};t_{i\mathbf{k}})
        =
        \E[Z_{i,\mathbf{j}}Z_{i,\mathbf{k}}\mid \mca{T}].
    \]
    Hence the corresponding quadratic form is nonnegative.
    Therefore \(\ang{\hat{G}_1 f,f}_{\caH^{\otimes r}}\ge0\), and \(\hat{G}_1\) is positive semidefinite.
    The same properties pass to \(G_1=\E[\hat{G}_1]\).
\end{proof}

With this notation,
\begin{equation}\label{eq:I1-operator-form}
    \hat{V}_1
    =
    \frac{\sigma^2}{mn}
    \Tr\bigl(
    T_r^{1/2}
    \hat{T}_{r,\lambda}^{-1}
    \hat{G}_1
    \hat{T}_{r,\lambda}^{-1}
    T_r^{1/2}
    \bigr).
\end{equation}
We also define the mixed and population versions
\begin{equation}\label{eq:I1-tilde-pop}
    \tilde{V}_1
    \coloneqq
    \frac{\sigma^2}{mn}
    \Tr\bigl(
    T_r^{1/2}
    T_{r,\lambda}^{-1}
    \hat{G}_1
    T_{r,\lambda}^{-1}
    T_r^{1/2}
    \bigr),
    \qquad
    V_1
    \coloneqq
    \frac{\sigma^2}{mn}
    \Tr\bigl(
    T_r^{1/2}
    T_{r,\lambda}^{-1}
    G_1
    T_{r,\lambda}^{-1}
    T_r^{1/2}
    \bigr).
\end{equation}
Because \(\hat{G}_1\) and \(G_1\) are positive semidefinite,
\begin{equation*}
    \hat{V}_1
    =
    \frac{\sigma^2}{mn}
    \|T_r^{1/2} \hat{T}_{r,\lambda}^{-1} \hat{G}_1^{1/2}\|_{HS}^2,
    \quad
    \tilde{V}_1
    =
    \frac{\sigma^2}{mn}
    \|T_r^{1/2} T_{r,\lambda}^{-1} \hat{G}_1^{1/2}\|_{HS}^2,
    \quad
    V_1
    =
    \frac{\sigma^2}{mn}
    \|T_r^{1/2} T_{r,\lambda}^{-1} G_1^{1/2}\|_{HS}^2.
\end{equation*}

As before,
\begin{equation}\label{eq:I1-basic-split}
    \hat{V}_1-V_1
    =
    (\hat{V}_1-\tilde{V}_1)
    +
    (\tilde{V}_1-V_1),
\end{equation}
with
\begin{equation}\label{eq:I1-second-gap}
    \tilde{V}_1-V_1
    =
    \frac{\sigma^2}{mn}
    \Tr\bigl(
    T_r^{1/2}
    T_{r,\lambda}^{-1}
    (\hat{G}_1-G_1)
    T_{r,\lambda}^{-1}
    T_r^{1/2}
    \bigr).
\end{equation}
The Hilbert-Schmidt comparison used earlier for \(\hat{V}_0-\tilde{V}_0\) applies verbatim here, after replacing \(\hat{G}_0\) by \(\hat{G}_1\).
Recalling
\[
    \Delta_r
    \coloneqq
    T_{r,\lambda}^{-1/2}(\hat{T}_r-T_r)T_{r,\lambda}^{-1/2},
\]
we work on the event
\[
    \|\Delta_r\|\le \frac{1}{2},
    \qquad
    \|T_{r,\lambda}^{1/2} \hat{T}_{r,\lambda}^{-1} T_{r,\lambda}^{1/2}\|\le 2,
\]
which has probability tending to one when \(v_r \to0\) by \Cref{lem:appendix-operator-concentration}.
On this event,
\begin{equation}\label{eq:I1-first-gap-trace}
    |\hat{V}_1-\tilde{V}_1|
    \le
    \frac{C\sigma^2}{mn}
    \|\Delta_r\|
    \Tr\bigl(
    T_{r,\lambda}^{-1/2} \hat{G}_1 T_{r,\lambda}^{-1/2}
    \bigr).
\end{equation}
Finally,
\begin{equation}\label{eq:I1-trace-split}
    \Tr\bigl(
    T_{r,\lambda}^{-1/2} \hat{G}_1 T_{r,\lambda}^{-1/2}
    \bigr)
    \le
    \Tr\bigl(
    T_{r,\lambda}^{-1/2} G_1 T_{r,\lambda}^{-1/2}
    \bigr)
    +
    \left|
    \Tr\bigl(
    T_{r,\lambda}^{-1/2}
    (\hat{G}_1-G_1)
    T_{r,\lambda}^{-1/2}
    \bigr)
    \right|.
\end{equation}

\begin{proposition}\label{prop:I1-main}
    Assume \(\mu_{2r-2} \neq 0\) and \(\sigma^2>0\), and suppose \(\lambda\to0\) satisfies
    \[
        \left(\ln \frac{1}{\lambda}\right)^{r-1}=o(m),
        \qquad
        v_r
        =
        \frac{\lambda^{-1/\beta}(\ln \frac{1}{\lambda})^r}{mn}
        \to 0.
    \]
    Then
    \begin{gather*}
        V_1
        =
        \frac{\sigma^2}{mn}V_{\mathrm{noise}}(\lambda;\mu_{2r-2})
        +
        o \left(
        \frac{\lambda^{-1/\beta}}{mn}
        \right),
        \qquad
        \frac{\sigma^2}{mn}V_{\mathrm{noise}}(\lambda;\mu_{2r-2})
        \asymp
        \frac{\lambda^{-1/\beta}}{mn}, \\
        \hat{V}_1
        =
        V_1
        +
        o_\bbP \left(
        \frac{\lambda^{-1/\beta}}{mn}
        \right).
    \end{gather*}
\end{proposition}

\begin{proof}
    By \Cref{prop:appendix-I1-pop-phi,lem:appendix-I1-expectations}
    and since \(\mu_{2r-2} \neq0\),
    \[
        V_1
        =
        \frac{\sigma^2}{mn}V_{\mathrm{noise}}(\lambda;\mu_{2r-2})
        +
        O \left(
        \frac{\sigma^2}{mn}
        \right)
        +
        O \left(
        \frac{\sigma^2 \lambda^{-1/\beta}
            \left(\ln \frac{1}{\lambda}\right)^{r-1}}{m^2 n}
        \right).
    \]
    Moreover,
    \(V_{\mathrm{noise}}(\lambda;\mu_{2r-2})\asymp\lambda^{-1/\beta}\).
    Since \(\sigma^2>0\) is fixed, \(\lambda^{-1/\beta} \to\infty\), and
    \(\left(\ln \frac{1}{\lambda}\right)^{r-1}=o(m)\), the two remainder terms are
    \(o(\lambda^{-1/\beta}/(mn))\), so
    \[
        V_1
        =
        \frac{\sigma^2}{mn}V_{\mathrm{noise}}(\lambda;\mu_{2r-2})
        +
        o \left(
        \frac{\lambda^{-1/\beta}}{mn}
        \right),
        \qquad
        V_1
        \asymp
        \frac{\lambda^{-1/\beta}}{mn}.
    \]

    Next, \Cref{prop:appendix-I1-fluctuation-phi} yields
    \[
        \tilde{V}_1-V_1
        =
        O_\bbP \left(
        \frac{\sigma^2 \lambda^{-1/\beta}
            \left(\ln \frac{1}{\lambda}\right)^{r-1}}{mn\sqrt{mn}}
        \right).
    \]
    By \Cref{eq:I0-log-over-sqrtmn},
    \[
        \frac{\left(\ln \frac{1}{\lambda}\right)^{r-1}}{\sqrt{mn}}
        \to 0,
    \]
    so
    \begin{equation}\label{eq:I1-second-gap-small}
        \tilde{V}_1-V_1
        =
        o_\bbP \left(
        \frac{\lambda^{-1/\beta}}{mn}
        \right).
    \end{equation}

    For the first gap, \Cref{lem:appendix-operator-concentration} gives
    \[
        \|\Delta_r\|
        =
        O_\bbP(\sqrt{v_r}).
    \]
    Also, \Cref{prop:appendix-I1-pop-resolvent} gives
    \[
        \Tr\bigl(
        T_{r,\lambda}^{-1/2} G_1 T_{r,\lambda}^{-1/2}
        \bigr)
        =
        O \left(
        \lambda^{-1/\beta}
        \right),
    \]
    while \Cref{prop:appendix-I1-fluctuation-resolvent} gives
    \[
        \Tr\bigl(
        T_{r,\lambda}^{-1/2}
        (\hat{G}_1-G_1)
        T_{r,\lambda}^{-1/2}
        \bigr)
        =
        O_\bbP \left(
        \frac{\lambda^{-1/\beta}
            \left(\ln \frac{1}{\lambda}\right)^{r-1}}{\sqrt{mn}}
        \right)
        =
        o_\bbP(\lambda^{-1/\beta}),
    \]
    again by \Cref{eq:I0-log-over-sqrtmn}.
    Hence \Cref{eq:I1-trace-split} implies
    \[
        \Tr\bigl(
        T_{r,\lambda}^{-1/2} \hat{G}_1 T_{r,\lambda}^{-1/2}
        \bigr)
        =
        O_\bbP \left(
        \lambda^{-1/\beta}
        \right).
    \]
    Substituting this into the high-probability inequality
    \Cref{eq:I1-first-gap-trace},
    \[
        \hat{V}_1-\tilde{V}_1
        =
        O_\bbP \left(
        \sigma^2 \sqrt{v_r}
        \frac{\lambda^{-1/\beta}}{mn}
        \right)
        =
        o_\bbP \left(
        \frac{\lambda^{-1/\beta}}{mn}
        \right).
    \]
    Combining this with \Cref{eq:I1-second-gap-small} and \Cref{eq:I1-basic-split} proves the claim.
\end{proof}

\subsection{Summary of the Variance Term}
\label{sec:variance-summary}

At this point the variance analysis is complete, so we collect the previous three sections into one statement.

\begin{theorem}\label{thm:variance-summary}
    Assume \(\Sigma_r \neq 0\), \(\mu_{2r-2} \neq 0\), and \(\sigma^2>0\).
    Suppose \(\lambda\to0\) satisfies
    \[
        \left(\ln \frac{1}{\lambda}\right)^{r-1}=o(m),
        \qquad
        v_r
        =
        \frac{\lambda^{-1/\beta}(\ln \frac{1}{\lambda})^r}{mn}
        \to 0.
    \]
    Then
    \[
        \Var(\lambda)
        =
        \bigl(
        1+o_\bbP(1)
        \bigr)
        (V_0+V_1).
    \]
    Equivalently,
    \begin{equation*}
        \Var(\lambda)
        =
        \bigl(
        1+o_\bbP(1)
        \bigr)
        \Bigl(
        \frac{1}{n}V_{\mathrm{curve}}(\lambda;\Sigma_r)
        +
        \frac{1}{mn}V_{\mathrm{point}}(\lambda;\Sigma_r)
        +
        \frac{\sigma^2}{mn}V_{\mathrm{noise}}(\lambda;\mu_{2r-2})
        \Bigr).
    \end{equation*}
\end{theorem}

\begin{proof}
    By \Cref{eq:variance-polynomial,eq:high-order-negligible,prop:I0-main,prop:I1-main},
    \[
        \Var(\lambda)
        =
        V_0
        +
        V_1
        +
        o_\bbP \left(
        \frac{1}{n}
        +
        \frac{\lambda^{-1/\beta}}{mn}
        \right).
    \]

    The deterministic equivalents in \Cref{prop:I0-main,prop:I1-main} give
    \begin{align*}
        V_0
         & =
        \frac{1}{n}V_{\mathrm{curve}}(\lambda;\Sigma_r)
        +
        \frac{1}{mn}V_{\mathrm{point}}(\lambda;\Sigma_r)
        +
        o \left(
        \frac{\lambda^{-1/\beta}}{mn}
        \right), \\
        V_1
         & =
        \frac{\sigma^2}{mn}V_{\mathrm{noise}}(\lambda;\mu_{2r-2})
        +
        o \left(
        \frac{\lambda^{-1/\beta}}{mn}
        \right),
    \end{align*}
    and hence
    \begin{align*}
        \Var(\lambda)
         & =
        \frac{1}{n}V_{\mathrm{curve}}(\lambda;\Sigma_r)
        +
        \frac{1}{mn}V_{\mathrm{point}}(\lambda;\Sigma_r)
        \\
         & \quad+
        \frac{\sigma^2}{mn}V_{\mathrm{noise}}(\lambda;\mu_{2r-2})
        +
        o_\bbP \left(
        \frac{1}{n}
        +
        \frac{\lambda^{-1/\beta}}{mn}
        \right).
    \end{align*}
    The same propositions imply that both \(V_0+V_1\) and the deterministic
    equivalent in the preceding display are
    \(\asymp \frac{1}{n}+\frac{\lambda^{-1/\beta}}{mn}\).
    Thus the two additive remainders above are relative \(o_\bbP(1)\), which gives
    both claimed multiplicative forms.
\end{proof}

\begin{proposition}[Noise-free variance expansion]\label{cor:variance-noiseless}
    Assume \(\Sigma_r \neq 0\), \(\sigma^2=0\), and suppose \Cref{ass:prelim-ae-lower-regularity} holds.
    If \(\lambda\to0\) satisfies
    \[
        \left(\ln \frac{1}{\lambda}\right)^{r-1}=o(m),
        \qquad
        v_r
        =
        \frac{\lambda^{-1/\beta}(\ln \frac{1}{\lambda})^r}{mn}
        \to 0,
    \]
    then
    \[
        \Var(\lambda)
        =
        \hat{V}_0
        =
        \bigl(
        1+o_\bbP(1)
        \bigr)
        \left(
        \frac{1}{n}V_{\mathrm{curve}}(\lambda;\Sigma_r)
        +
        \frac{1}{mn}V_{\mathrm{point}}(\lambda;\Sigma_r)
        \right).
    \]
\end{proposition}

\begin{proof}
    When \(\sigma^2=0\), all terms \(\hat{V}_d\) with \(d\ge1\) vanish in
    \Cref{eq:variance-polynomial}, so \(\Var(\lambda)=\hat{V}_0\).  By
    \Cref{prop:I0-main},
    \[
        \hat{V}_0
        =
        \frac{1}{n}V_{\mathrm{curve}}(\lambda;\Sigma_r)
        +
        \frac{1}{mn}V_{\mathrm{point}}(\lambda;\Sigma_r)
        +
        o_\bbP \left(
        \frac{1}{n}
        +
        \frac{\lambda^{-1/\beta}}{mn}
        \right).
    \]
    Moreover, \Cref{lem:appendix-I0-expectations} gives
    \(V_{\mathrm{curve}}(\lambda;\Sigma_r)\asymp1\), and
    \Cref{lem:appendix-I0-vpoint-lower} gives
    \(V_{\mathrm{point}}(\lambda;\Sigma_r)\asymp\lambda^{-1/\beta}\).
    Therefore the additive remainder is of smaller order than the displayed
    deterministic leading term.
\end{proof}

\subsection{Estimating the Bias Term}
\label{sec:bias-term}

We now turn to the second term in \Cref{eq:bias-variance}, namely
\[
    \|\lambda\hat{T}_{r,\lambda}^{-1} \mu_r\|_{L^2(\mca I^r)}^2.
\]
Recall the \(L^2(\mca I^r)\)-expansion of \(\mu_r\):
\[
    \mu_r
    =
    \sum_{\mathbf{i}=1}^{\infty}
    c_{\mathbf{i}} e_{\mathbf{i}},
    \qquad
    e_{\mathbf{i}}
    =
    e_{i_1} \otimes\cdots\otimes e_{i_r},
    \qquad
    \Lambda_{\mathbf{i}}
    =
    \prod_{a=1}^r \lambda_{i_a}.
\]
The deterministic quantity that governs the bias is
\begin{equation}\label{eq:bias-population-scale}
    B(\lambda;\mu_r)
    \coloneqq
    \|\lambda T_{r,\lambda}^{-1} \mu_r\|_{L^2(\mca I^r)}.
\end{equation}

\begin{lemma}\label{lem:bias-source-upper}
    If \(\mu_r \in([H]^s)^{\otimes r}\) for some \(s>0\), then
    \[
        B(\lambda;\mu_r)
        =
        O \bigl(
        \lambda^{\min(s,2)/2}
        \bigr),
    \]
    or equivalently,
    \[
        B(\lambda;\mu_r)^2
        =
        O \bigl(
        \lambda^{\min(s,2)}
        \bigr).
    \]
\end{lemma}

\begin{proof}
    Write
    \[
        \mu_r
        =
        \sum_{\mathbf{i}=1}^{\infty}
        a_{\mathbf{i}}
        \Lambda_{\mathbf{i}}^{s/2}
        e_{\mathbf{i}},
        \qquad
        \sum_{\mathbf{i}=1}^{\infty} a_{\mathbf{i}}^2
        =
        \|\mu_r\|_{([H]^s)^{\otimes r}}^2.
    \]
    Then
    \begin{align*}
        B(\lambda;\mu_r)^2
         & =
        \sum_{\mathbf{i}=1}^{\infty}
        a_{\mathbf{i}}^2
        \Lambda_{\mathbf{i}}^s
        \left(
        \frac{\lambda}{\Lambda_{\mathbf{i}}+\lambda}
        \right)^2 \\
         & \le
        \lambda^2
        \sup_{x\in(0,\lambda_1^r]}
        \frac{x^s}{(x+\lambda)^2}
        \sum_{\mathbf{i}=1}^{\infty} a_{\mathbf{i}}^2.
    \end{align*}
    If \(s\le2\), the supremum is \(O(\lambda^{s-2})\).  If \(s>2\), it is
    \(O(1)\).  This gives the claim.
\end{proof}

\begin{proposition}\label{prop:bias-population-size}
    Suppose \Cref{ass:source} holds, and set \(\tilde{s} \coloneqq \min(s,2)\).
    Then:
    \begin{enumerate}
        \item
              for every \(s'<s\),
              \[
                  B(\lambda;\mu_r)
                  =
                  O \bigl(
                  \lambda^{\min(s',2)/2}
                  \bigr);
              \]
        \item
              \[
                  B(\lambda;\mu_r)
                  =
                  \Omega(\lambda^{\tilde{s}/2}).
              \]
    \end{enumerate}
\end{proposition}

\begin{proof}
    The first claim follows from \Cref{lem:bias-source-upper}, applied with
    \(s'\) in place of \(s\).

    For the lower bound, first assume \(s<2\).
    Then
    \begin{align*}
        B(\lambda;\mu_r)^2
         & =
        \sum_{\mathbf{i}=1}^{\infty}
        c_{\mathbf{i}}^2
        \left(
        \frac{\lambda}{\Lambda_{\mathbf{i}}+\lambda}
        \right)^2 \\
         & \ge
        \sum_{\Lambda_{\mathbf{i}}<\lambda}
        c_{\mathbf{i}}^2
        \left(
        \frac{\lambda}{\Lambda_{\mathbf{i}}+\lambda}
        \right)^2 \\
         & \ge
        \frac{1}{4}
        \sum_{\Lambda_{\mathbf{i}}<\lambda} c_{\mathbf{i}}^2
        =
        \Omega(\lambda^s).
    \end{align*}
    Therefore \(B(\lambda;\mu_r)=\Omega(\lambda^{s/2})\).

    If \(s\ge 2\), pick any multi-index \(\mathbf{i}_\ast\) such that
    \(c_{\mathbf{i}_\ast}\neq 0\).
    Then
    \[
        B(\lambda;\mu_r)^2
        \ge
        c_{\mathbf{i}_\ast}^2
        \left(
        \frac{\lambda}{\Lambda_{\mathbf{i}_\ast}+\lambda}
        \right)^2
        \asymp
        \lambda^2.
    \]
    Hence \(B(\lambda;\mu_r)=\Omega(\lambda)\).
    Since \(\tilde{s}=\min(s,2)\), this proves the second claim.
\end{proof}

\begin{example}[A borderline coefficient sequence]\label{ex:bias-borderline-coefficients}
    This example gives a concrete borderline form of the squared bias
    \(B(\lambda;\mu_r)^2\).
    Suppose that, for some \(s>0\), the coefficients in the
    \(L^2(\mca I^r)\)-expansion
    \[
        \mu_r
        =
        \sum_{\mathbf{i}=1}^{\infty}
        b_{\mathbf{i}}e_{\mathbf{i}}
    \]
    satisfy
    \[
        |b_{\mathbf{i}}|
        \asymp
        \Lambda_{\mathbf{i}}^{(s+1/\beta)/2}.
    \]
    Then, as \(\lambda\to0\),
    \[
        B(\lambda;\mu_r)^2
        \asymp
        \begin{cases}
            \lambda^s \left(\ln\frac{1}{\lambda}\right)^{r-1},
             & s<2, \\
            \lambda^2 \left(\ln\frac{1}{\lambda}\right)^r,
             & s=2, \\
            \lambda^2,
             & s>2.
        \end{cases}
    \]
    The proof is deferred to \Cref{app:omitted-proofs}.
\end{example}

\begin{theorem}\label{thm:bias-main}
    Suppose \Cref{ass:source} holds.
    Suppose also that \(\lambda\to0\) satisfies
    \[
        \lambda
        =
        \Omega\bigl(
        (mn)^{-\theta}
        \bigr)
    \]
    for some fixed \(\theta<\beta\).
    Then
    \[
        \|\lambda\hat{T}_{r,\lambda}^{-1} \mu_r\|_{L^2(\mca I^r)}^2
        =
        \bigl(
        1+o_\bbP(1)
        \bigr)
        B(\lambda;\mu_r)^2.
    \]
\end{theorem}

\begin{proof}
    Set
    \[
        \Bias(\lambda)
        \coloneqq
        \|\lambda\hat{T}_{r,\lambda}^{-1} \mu_r\|_{L^2(\mca I^r)},
        \qquad
        B(\lambda;\mu_r)
        \coloneqq
        \|\lambda T_{r,\lambda}^{-1} \mu_r\|_{L^2(\mca I^r)}.
    \]
    By the triangle inequality,
    \[
        |\Bias(\lambda)-B(\lambda;\mu_r)|
        \le
        \|(\lambda\hat{T}_{r,\lambda}^{-1}-\lambda T_{r,\lambda}^{-1})\mu_r\|_{L^2(\mca I^r)}.
    \]

    The resolvent identity gives
    \[
        \lambda\hat{T}_{r,\lambda}^{-1}-\lambda T_{r,\lambda}^{-1}
        =
        \lambda\hat{T}_{r,\lambda}^{-1}(T_r-\hat{T}_r)T_{r,\lambda}^{-1}.
    \]
    Therefore
    \begin{align}
        \|(\lambda\hat{T}_{r,\lambda}^{-1}-\lambda T_{r,\lambda}^{-1})\mu_r\|_{L^2}
        =\    &
        \lambda
        \|T_r^{1/2} \hat{T}_{r,\lambda}^{-1}(T_r-\hat{T}_r)T_{r,\lambda}^{-1} \mu_r\|_{\caH^{\otimes r}}\notag \\
        \le\  &
        \lambda
        \|T_r^{1/2} T_{r,\lambda}^{-1/2}\|
        \|T_{r,\lambda}^{1/2} \hat{T}_{r,\lambda}^{-1} T_{r,\lambda}^{1/2}\|
        \|T_{r,\lambda}^{-1/2}(\hat{T}_r-T_r)T_{r,\lambda}^{-1} \mu_r\|_{\caH^{\otimes r}}.
        \label{eq:bias-resolvent-difference}
    \end{align}
    Since \(\lambda=\Omega((mn)^{-\theta})\), we have
    \[
        v_r
        =
        \frac{\lambda^{-1/\beta}(\ln \frac{1}{\lambda})^r}{mn}
        =
        O \left(
        (mn)^{\theta/\beta-1}
        \left(\ln(mn)\right)^r
        \right)
        =
        o(1),
    \]
    because \(\theta<\beta\).
    On the high-probability event from \Cref{lem:appendix-operator-concentration},
    \[
        \|T_r^{1/2} T_{r,\lambda}^{-1/2}\|\le 1,
        \qquad
        \|T_{r,\lambda}^{1/2} \hat{T}_{r,\lambda}^{-1} T_{r,\lambda}^{1/2}\|
        \le 2,
    \]
    so \cref{eq:bias-resolvent-difference} becomes
    \begin{equation}\label{eq:bias-difference-main}
        \|(\lambda\hat{T}_{r,\lambda}^{-1}-\lambda T_{r,\lambda}^{-1})\mu_r\|_{L^2}
        \le
        2\lambda
        \|T_{r,\lambda}^{-1/2}(\hat{T}_r-T_r)T_{r,\lambda}^{-1} \mu_r\|_{\caH^{\otimes r}}.
    \end{equation}

    Define the scalar function
    \[
        q_\lambda(\bm{t})
        \coloneqq
        (T_{r,\lambda}^{-1} \mu_r)(\bm{t}),
        \qquad
        \bm{t}=(t_1,\ldots,t_r)\in \mca I^r,
    \]
    and the \(\caH^{\otimes r}\)-valued kernel
    \[
        \xi_\lambda(\bm{t})
        \coloneqq
        q_\lambda(\bm{t})
        T_{r,\lambda}^{-1/2} K(\bm{t}).
    \]
    Then applying the reproducing property of \(K\) gives
    \[
        T_{r,\lambda}^{-1/2} \hat{T}_r T_{r,\lambda}^{-1} \mu_r
        =
        \frac{1}{n}\sum_{i=1}^n \frac{1}{(m)_r}
        \sum_{\mathbf{j}\in I_m^r}
        \xi_\lambda(t_{i\mathbf{j}}),
    \]
    and
    \[
        \E[\xi_\lambda(\bm{t})]
        =
        T_{r,\lambda}^{-1/2} T_r T_{r,\lambda}^{-1} \mu_r.
    \]
    Hence
    \begin{equation}\label{eq:bias-xi-centered}
        T_{r,\lambda}^{-1/2}(\hat{T}_r-T_r)T_{r,\lambda}^{-1} \mu_r
        =
        \frac{1}{n}\sum_{i=1}^n \frac{1}{(m)_r}
        \sum_{\mathbf{j}\in I_m^r}
        \xi_\lambda(t_{i\mathbf{j}})
        -
        \E[\xi_\lambda(\bm{t})].
    \end{equation}

    We next estimate the variance and the uniform bound of \(\xi_\lambda\).
    First,
    \begin{align}
        \E[q_\lambda(\bm{t})^2]
         & =
        \|q_\lambda\|_{L^2(\mca I^r)}^2 \notag \\
         & =
        \|T_{r,\lambda}^{-1} \mu_r\|_{L^2(\mca I^r)}^2
        =
        \lambda^{-2} B(\lambda;\mu_r)^2.
        \label{eq:bias-q-L2}
    \end{align}
    Also, by \Cref{lem:appendix-phi-T},
    \[
        \|T_{r,\lambda}^{-1/2} K(\bm{t})\|_{\caH^{\otimes r}}^2
        \le
        M^r \mcaN_1^{(r)}(\lambda),
        \qquad
        \bm{t}\in \mca I^r.
    \]
    Therefore
    \begin{equation}\label{eq:bias-xi-variance}
        \E\|\xi_\lambda(\bm{t})\|_{\caH^{\otimes r}}^2
        \le
        M^r \mcaN_1^{(r)}(\lambda)\lambda^{-2} B(\lambda;\mu_r)^2.
    \end{equation}

    To bound \(q_\lambda\) uniformly, choose any
    \[
        s'\in(1/\beta,\min(s,2)),
        \qquad
        \alpha\in(1/\beta,s').
    \]
    Since \(\mu_r \in([H]^{s'})^{\otimes r}\), we may write
    \[
        \mu_r
        =
        \sum_{\mathbf{i}=1}^{\infty}
        a_{\mathbf{i}}
        \Lambda_{\mathbf{i}}^{s'/2}
        e_{\mathbf{i}},
        \qquad
        \sum_{\mathbf{i}=1}^{\infty} a_{\mathbf{i}}^2
        =
        \|\mu_r\|_{([H]^{s'})^{\otimes r}}^2.
    \]
    Hence
    \[
        q_\lambda(\bm{t})
        =
        \sum_{\mathbf{i}=1}^{\infty}
        a_{\mathbf{i}}
        \frac{\Lambda_{\mathbf{i}}^{s'/2}}{\Lambda_{\mathbf{i}}+\lambda}
        e_{\mathbf{i}}(\bm{t}).
    \]
    By Cauchy-Schwarz,
    \begin{align*}
        |q_\lambda(\bm{t})|^2
        \le\  &
        \|\mu_r\|_{([H]^{s'})^{\otimes r}}^2
        \sum_{\mathbf{i}=1}^{\infty}
        \frac{\Lambda_{\mathbf{i}}^{s'}}{(\Lambda_{\mathbf{i}}+\lambda)^2}
        e_{\mathbf{i}}(\bm{t})^2 \\
        =\    &
        \|\mu_r\|_{([H]^{s'})^{\otimes r}}^2
        \sum_{\mathbf{i}=1}^{\infty}
        \frac{\Lambda_{\mathbf{i}}^{s'-\alpha}}{(\Lambda_{\mathbf{i}}+\lambda)^2}
        \Lambda_{\mathbf{i}}^\alpha
        e_{\mathbf{i}}(\bm{t})^2.
    \end{align*}
    Since \(s'<2\), we have
    \[
        \sup_{x>0}
        \frac{x^{s'-\alpha}}{(x+\lambda)^2}
        \lesssim
        \lambda^{s'-\alpha-2}.
    \]
    Moreover,
    \[
        \sum_{\mathbf{i}=1}^{\infty}
        \Lambda_{\mathbf{i}}^\alpha e_{\mathbf{i}}(\bm{t})^2
        =
        \prod_{\ell=1}^r
        \left(
        \sum_{j=1}^{\infty}
        \lambda_j^\alpha e_j(t_\ell)^2
        \right)
        \le
        M_\alpha^{2r},
    \]
    because \(\alpha>1/\beta\).
    Therefore
    \begin{equation}\label{eq:bias-q-sup}
        \|q_\lambda\|_{\infty}
        =
        O \left(
        \lambda^{-(2-s'+\alpha)/2}
        \right).
    \end{equation}
    Combining this with \Cref{lem:appendix-phi-T}, we obtain
    \begin{equation}\label{eq:bias-xi-sup}
        \|\xi_\lambda\|_{\infty}
        =
        O \left(
        \lambda^{-(2-s'+\alpha)/2}
        \sqrt{
            \mcaN_1^{(r)}(\lambda)
        }
        \right).
    \end{equation}

    Now apply \Cref{lem:appendix-U-bernstein-vector} to the kernel \(\xi_\lambda\).
    Using \Cref{eq:bias-xi-centered,eq:bias-xi-variance,eq:bias-xi-sup} and the fact that
    \(\lfloor m/r\rfloor\asymp m\) because \(r\) is fixed, we get
    \begin{equation}\label{eq:bias-centered-operator-bound}
        \|T_{r,\lambda}^{-1/2}(\hat{T}_r-T_r)T_{r,\lambda}^{-1} \mu_r\|_{\caH^{\otimes r}}
        =
        O_\bbP \left(
        \lambda^{-1} B(\lambda;\mu_r)
        \sqrt{
            \frac{
                \mcaN_1^{(r)}(\lambda)
            }{
                mn
            }
        }
        +
        \lambda^{-(2-s'+\alpha)/2}
        \frac{
            \sqrt{
                \mcaN_1^{(r)}(\lambda)
            }
        }{
            mn
        }
        \right).
    \end{equation}
    Combining \Cref{eq:bias-centered-operator-bound} with the
    high-probability inequality \Cref{eq:bias-difference-main},
    \begin{equation}\label{eq:bias-difference-pre-final}
        \|(\lambda\hat{T}_{r,\lambda}^{-1}-\lambda T_{r,\lambda}^{-1})\mu_r\|_{L^2}
        =
        O_\bbP \left(
        B(\lambda;\mu_r)
        \sqrt{
            \frac{
                \mcaN_1^{(r)}(\lambda)
            }{
                mn
            }
        }
        +
        \lambda^{1-(2-s'+\alpha)/2}
        \frac{
            \sqrt{
                \mcaN_1^{(r)}(\lambda)
            }
        }{
            mn
        }
        \right).
    \end{equation}

    We now compare the two terms with \(B(\lambda;\mu_r)\).
    By \Cref{lem:appendix-effective-dimension},
    \[
        \mcaN_1^{(r)}(\lambda)
        \asymp
        \lambda^{-1/\beta}
        \left(
        \ln \frac{1}{\lambda}
        \right)^{r-1}.
    \]
    Hence
    \[
        \frac{\mcaN_1^{(r)}(\lambda)}{mn}
        \lesssim
        \lambda^{1/\theta-1/\beta}
        \left(
        \ln \frac{1}{\lambda}
        \right)^{r-1}
        \to 0,
    \]
    so the first term on the right-hand side of \Cref{eq:bias-difference-pre-final} is
    \(o_\bbP(B(\lambda;\mu_r))\).

    Let \(\tilde{s}=\min(s,2)\).
    By \Cref{prop:bias-population-size},
    \[
        B(\lambda;\mu_r)
        =
        \Omega(\lambda^{\tilde{s}/2}).
    \]
    Therefore the second term in \Cref{eq:bias-difference-pre-final}, divided by \(B(\lambda;\mu_r)\), is
    \[
        O_\bbP \left(
        \frac{
                \lambda^{1-(2-s'+\alpha)/2}
                \sqrt{
                    \mcaN_1^{(r)}(\lambda)
                }
            }{
                mn B(\lambda;\mu_r)
            }
        \right)
        =
        O_\bbP \left(
        \frac{
            \lambda^{(s'-\tilde{s}-\alpha-1/\beta)/2}
            \left(
            \ln \frac{1}{\lambda}
            \right)^{(r-1)/2}
        }{
            mn
        }
        \right).
    \]
    Since \(\lambda=\Omega((mn)^{-\theta})\), we have
    \[
        \frac{1}{mn}
        =
        O(\lambda^{1/\theta}).
    \]
    Hence the previous display is bounded by
    \[
        O_\bbP \left(
        \lambda^\kappa
        \left(
            \ln \frac{1}{\lambda}
            \right)^{(r-1)/2}
        \right),
        \qquad
        \kappa
        \coloneqq
        \frac{1}{\theta}
        +
        \frac{
            s'-\tilde{s}-\alpha-1/\beta
        }{
            2
        }.
    \]
    Because \(\theta<\beta\), we have \(1/\theta>1/\beta\).
    We may therefore choose \(s'<\min(s,2)\) sufficiently close to \(\tilde{s}\) and \(\alpha>1/\beta\) sufficiently close to \(1/\beta\) so that \(\kappa>0\).
    For this choice,
    \[
        \lambda^\kappa
        \left(
        \ln \frac{1}{\lambda}
        \right)^{(r-1)/2}
        \to 0.
    \]
    Hence the second term in \Cref{eq:bias-difference-pre-final} is also \(o_\bbP(B(\lambda;\mu_r))\).

    We have proved
    \[
        \|(\lambda\hat{T}_{r,\lambda}^{-1}-\lambda T_{r,\lambda}^{-1})\mu_r\|_{L^2}
        =
        o_\bbP(B(\lambda;\mu_r)).
    \]
    Combining this with the initial triangle inequality yields
    \[
        \Bias(\lambda)
        =
        B(\lambda;\mu_r)
        +
        o_\bbP(B(\lambda;\mu_r)).
    \]
    Since both quantities are nonnegative, this immediately implies
    \[
        \Bias(\lambda)^2
        =
        \bigl(
        1+o_\bbP(1)
        \bigr)
        B(\lambda;\mu_r)^2.
    \]
\end{proof}
 \section{Conclusion}
\label{sec:conclusion}

We studied KRR for estimating the \(r\)-th moment function of a random function from noisy, discretely observed curves.
Using a tensor-product RKHS, the construction covers mean, covariance, and higher-order moment estimation.
For every admissible regularization parameter, our main result gives a \(1+o_{\bbP}(1)\) deterministic equivalent of the \(L^2\) generalization error.
It describes how the error varies with \(\lambda\), the number of curves \(n\), the number of observations per curve \(m\), and the noise level.
It also identifies four terms representing different sources of error: the bias term and the curve-level, point-level, and noise variance terms.
As applications, we derived the optimal regularization parameter and showed that KRR attains the minimax rate when \(s\le2\), whereas the saturation effect makes it suboptimal in the sparse regime when \(s>2\).
These results provide a complete picture of the generalization error of KRR for functional data.

On the technical side, we formulate the estimation of the \(r\)-th moment function as KRR on the tensor-product RKHS \(\caH^{\otimes r}\), with within-curve products at distinct sampling locations serving as regression responses.
This construction suggests a general route to other multivariate functional estimation problems.
Moreover, we provide a set of Bernstein-type concentration inequalities for \(U\)-statistics arising from functional data.
These inequalities are useful under the dependent product structure of functional observations, and they may be of independent interest.

Several future directions are worth exploring.
First, the sampling design can affect the minimax rate in functional data analysis \citep{cai2011OptimalEstimationMean}.
Extending the present analysis to the common design would require a substantive modification, because observations from different curves then share the same sampling locations and hence obey a different sampling geometry.
Second, spectral algorithms form a well-studied class of regularization methods that includes KRR~\citep{bauer2007RegularizationAlgorithmsLearning}, and their exact generalization error curves have been established in nonparametric regression~\citep{li2024GeneralizationErrorCurves}.
For functional data, building upon the optimal rates for spectral algorithms~\citep{gupta2026MinimaxOptimalEstimation},
it would be natural to establish analogous deterministic equivalents for these methods.
 
\section*{Acknowledgments}

The authors thank colleagues and seminar participants for helpful comments on earlier versions of this manuscript.
The authors received no specific funding for this work and declare no competing interests.

\appendix
\section{Concentration Inequalities}
\label{app:concentration}

This appendix collects the concentration tools used implicitly in the trace estimates.
The scalar \(U\)-statistic inequalities are obtained from moment-generating-function bounds together with Hoeffding's blocking representation \citep{hoeffding1963ProbabilityInequalitiesSums}.
The operator-valued inequality combines the same blocking representation with the matrix Laplace-transform method of \citet{tropp2015IntroductionMatrixConcentration} and its Hilbert-space extension with effective rank due to \citet{minsker2017ExtensionsBernsteinsInequality}.
The vector-valued inequality is then derived from the operator-valued inequality by a standard self-adjoint dilation.

\subsection{Classical Inequalities}

\begin{lemma}[Hoeffding's inequality]\label{lem:appendix-hoeffding}
    Let \(\xi,\xi_1,\dots,\xi_n\) be i.i.d.\ real-valued random variables such that
    \[
        \E\exp \bigl(
            \lambda(\xi-\E\xi)
        \bigr)
        \le
        \exp \left(
            \frac{\lambda^2 B^2}{2}
        \right),
        \qquad
        \lambda\in\R.
    \]
    Then for any \(\delta\in(0,1)\), with probability at least \(1-\delta\),
    \[
        \left|
            \frac{1}{n}\sum_{i=1}^n \xi_i-\E\xi
        \right|
        \le
        \sqrt{
            \frac{2B^2 \ln(2/\delta)}{n}
        }.
    \]
    In particular, if
    \[
        \mr{ess sup} \xi-\mr{ess inf} \xi\le 2B
    \]
    almost surely, then the same conclusion holds.
\end{lemma}

\begin{lemma}[Bernstein's inequality]\label{lem:appendix-bernstein}
    Let \(\xi,\xi_1,\dots,\xi_n\) be i.i.d.\ real-valued random variables.
    Suppose at least one of the following conditions holds:
    \begin{enumerate}
        \item
        \[
            |\xi-\E\xi|\le L
            \quad\text{a.s.},
            \qquad
            \Var(\xi)\le \sigma^2.
        \]
        \item
        \[
            \E|\xi-\E\xi|^p
            \le
            \frac{1}{2}p!\sigma^2 L^{p-2},
            \qquad
            p=2,3,\dots.
        \]
        \item
        \[
            \E\exp \bigl(
                \lambda(\xi-\E\xi)
            \bigr)
            \le
            \exp \left(
                \frac{\lambda^2 \sigma^2}{2(1-|\lambda|L)}
            \right),
            \qquad
            |\lambda|<\frac{1}{L}.
        \]
    \end{enumerate}
    Then for any \(\delta\in(0,1)\), with probability at least \(1-\delta\),
    \[
        \left|
            \frac{1}{n}\sum_{i=1}^n \xi_i-\E\xi
        \right|
        \le
        \sqrt{
            \frac{2\sigma^2 \ln(2/\delta)}{n}
        }
        +
        \frac{L\ln(2/\delta)}{n}.
    \]
    Moreover, condition \(1\Rightarrow2\Rightarrow3\).
\end{lemma}
The two inequalities above are classical.
In what follows we only need the moment-generating-function conditions.

\subsection{Scalar \texorpdfstring{\(U\)}{U}-statistics}

For a measurable kernel \(h:\mca{X}^k \to\R\) and an i.i.d.\ array
\[
    \{X_{ij}:1\le i\le n,\ 1\le j\le m\},
\]
write
\[
    U_{m,k}^{(i)}(h)
    \coloneqq
    \frac{1}{(m)_k}
    \sum_{\mathbf{j}\in I_m^k}
    h(X_{i,j_1},\dots,X_{i,j_k}),
    \qquad
    U_{m,k}(h)
    \coloneqq
    \frac{1}{n}\sum_{i=1}^n U_{m,k}^{(i)}(h).
\]

The following inequality is due to
\citet{hoeffding1963ProbabilityInequalitiesSums}, through Hoeffding's blocking
representation.
\begin{lemma}[Bernstein's inequality for scalar \(U\)-statistics]\label{lem:appendix-U-bernstein-scalar}
    Let \(h:\mca{X}^k \to\R\) be measurable, and let
    \[
        \theta \coloneqq \E h(T_1,\dots,T_k),
    \]
    where \(T_1,\dots,T_k\) are i.i.d.\ copies of \(X_{11}\).
    Assume
    \[
        \Var\bigl(
            h(T_1,\dots,T_k)
        \bigr)
        \le \sigma^2,
        \qquad
        \|h-\theta\|_\infty\le L.
    \]
    Then for any \(\delta\in(0,1)\), with probability at least \(1-\delta\),
    \[
        \left|
            U_{m,k}(h)-\theta
        \right|
        \le
        \sqrt{
            \frac{2\sigma^2 \ln(2/\delta)}
            {n\lfloor m/k\rfloor}
        }
        +
        \frac{L\ln(2/\delta)}
        {n\lfloor m/k\rfloor}.
    \]
\end{lemma}

\begin{proof}
    Define the symmetrized kernel
    \[
        h^{\mr{sym}}(x_1,\dots,x_k)
        \coloneqq
        \frac{1}{k!}
        \sum_{\tau\in\mathfrak{S}_k}
        h(x_{\tau(1)},\dots,x_{\tau(k)}).
    \]
    Then
    \[
        U_{m,k}^{(i)}(h)
        =
        \frac{1}{\binom{m}{k}}
        \sum_{1\le j_1<\cdots<j_k \le m}
        h^{\mr{sym}}(X_{ij_1},\dots,X_{ij_k}),
    \]
    and
    \[
        \E h^{\mr{sym}}(T_1,\dots,T_k)=\theta.
    \]
    Moreover,
    \[
        \Var\bigl(
            h^{\mr{sym}}(T_1,\dots,T_k)
        \bigr)
        \le \sigma^2,
        \qquad
        \|h^{\mr{sym}}-\theta\|_\infty\le L,
    \]
    so we may assume from the outset that \(h\) is symmetric.

    Let
    \[
        \bar{h}(x_1,\dots,x_k) \coloneqq h(x_1,\dots,x_k)-\theta,
    \]
    and write
    \[
        q \coloneqq \lfloor m/k\rfloor.
    \]
    For every \(\tau\in\mathfrak{S}_m\), define
    \[
        Y_\tau
        \coloneqq
        \frac{1}{nq}
        \sum_{i=1}^n
        \sum_{\ell=1}^{q}
        \bar{h}\bigl(
            X_{i,\tau(k(\ell-1)+1)},
            \dots,
            X_{i,\tau(k\ell)}
        \bigr).
    \]
    Hoeffding's blocking representation gives
    \[
        U_{m,k}(h)-\theta
        =
        \frac{1}{m!}
        \sum_{\tau\in\mathfrak{S}_m} Y_\tau.
    \]
    For fixed \(\tau\), the variable \(Y_\tau\) is the average of \(nq\) i.i.d.\ centered random variables, each distributed as \(\bar{h}(T_1,\dots,T_k)\).
    Hence condition \(1\) of \Cref{lem:appendix-bernstein} implies the moment-generating-function bound
    \[
        \E\exp(\lambda Y_\tau)
        \le
        \exp \left(
            \frac{\lambda^2 \sigma^2}
            {2nq\left(1-\frac{|\lambda|L}{nq}\right)}
        \right),
        \qquad
        |\lambda|<\frac{nq}{L}.
    \]
    Since the \(Y_\tau\) are identically distributed, Jensen's inequality gives
    \begin{align*}
        \E\exp \bigl(
            \lambda(U_{m,k}(h)-\theta)
        \bigr)
        &=
        \E\exp \left(
            \frac{\lambda}{m!}
            \sum_{\tau\in\mathfrak{S}_m}Y_\tau
        \right)\\
        &\le
        \frac{1}{m!}
        \sum_{\tau\in\mathfrak{S}_m}
        \E\exp(\lambda Y_\tau)\\
        &\le
        \exp \left(
            \frac{\lambda^2 \sigma^2}{2nq\left(1-\frac{|\lambda|L}{nq}\right)}
        \right),
        \qquad
        |\lambda|<\frac{nq}{L}.
    \end{align*}
    Thus \(U_{m,k}(h)-\theta\) itself satisfies condition \(3\) of \Cref{lem:appendix-bernstein} with effective sample size \(nq\).
    Applying that lemma completes the proof.
\end{proof}

\begin{lemma}[Hoeffding's inequality for scalar \(U\)-statistics]\label{lem:appendix-U-hoeffding-scalar}
    Let \(h:\mca{X}^k \to\R\) be measurable, and let
    \[
        \theta \coloneqq \E h(T_1,\dots,T_k),
    \]
    where \(T_1,\dots,T_k\) are i.i.d.\ copies of \(X_{11}\).
    Assume
    \[
        \mr{ess sup} h-\mr{ess inf} h\le 2L.
    \]
    Then for any \(\delta\in(0,1)\), with probability at least \(1-\delta\),
    \[
        \left|
            U_{m,k}(h)-\theta
        \right|
        \le
        L\sqrt{
            \frac{2\ln(2/\delta)}
            {n\lfloor m/k\rfloor}
        }.
    \]
\end{lemma}

\begin{proof}
    Since Hoeffding's inequality also follows from a moment-generating-function condition, we can apply the same proof as in \Cref{lem:appendix-U-bernstein-scalar}.
    One only replaces the Bernstein moment-generating-function bound by \Cref{lem:appendix-hoeffding}.
\end{proof}

\subsection{Operator-valued \texorpdfstring{\(U\)}{U}-statistics}

The next lemma is an operator-valued \(U\)-statistic analogue of Minsker's Bernstein inequality.
The concentration mechanism is not new: in the matrix case it is the matrix Laplace-transform argument of \citet{tropp2015IntroductionMatrixConcentration}, refined to effective rank and extended to separable Hilbert spaces by \citet{minsker2017ExtensionsBernsteinsInequality}.
The additional step here is to insert Hoeffding's blocking representation into that operator Bernstein argument, so that the effective sample size becomes \(n\lfloor m/k\rfloor\).

\begin{lemma}[Operator Bernstein for \(U\)-statistics]\label{lem:appendix-U-bernstein-operator}
    Let \(P\) be a probability measure on \(\mca{X}\).
    Let \(H\) be a separable Hilbert space, and let
    \[
        Z:\mca{X}^k \to \mca{S}(H)
    \]
    be a measurable kernel taking values in the self-adjoint Hilbert-Schmidt operators on \(H\).
    Define
    \[
        D \coloneqq \E Z(X_1,\dots,X_k),
    \]
    and
    \[
        \hat{D}
        \coloneqq
        \frac{1}{n}\sum_{i=1}^n
        \frac{1}{(m)_k}
        \sum_{\mathbf{j}\in I_m^k}
        Z(X_{i,j_1},\dots,X_{i,j_k}),
    \]
    where \(X_{ij} \overset{\mr{i.i.d.}}{\sim}P\) and \(m\ge k\).
    Suppose there exist \(L>0\) and a positive trace-class operator \(S\) such that
    \[
        \|Z-D\|\le L
        \quad\text{a.s.},
        \qquad
        \E\bigl[(Z(X_1,\dots,X_k)-D)^2\bigr]\preceq S.
    \]
    If \(S=0\), then \(\hat{D}=D\) almost surely.
    Otherwise, for every \(\delta\in(0,1)\),
    \[
        \bbP \left(
            \|\hat{D}-D\|
            >
            \sqrt{
                \frac{2u_\delta\|S\|}
                {n\lfloor m/k\rfloor}
            }
            +
            \frac{2u_\delta L}
            {3n\lfloor m/k\rfloor}
        \right)
        \le
        \delta,
    \]
    where
    \[
        u_\delta
        \coloneqq
        \ln\frac{4d}{\delta},
        \qquad
        d \coloneqq \frac{\Tr S}{\|S\|}.
    \]
    Here \(d\) is the effective rank of \(S\).
\end{lemma}

\begin{proof}
    If \(S=0\), then
    \[
        \E\|(Z-D)u\|_H^2
        =
        \ang{\E[(Z-D)^2]u,u}
        =
        0,
        \qquad u\in H,
    \]
    and hence \(Z=D\) almost surely, so there is nothing to prove.

    Define the symmetrized kernel
    \[
        \tilde{Z}(x_1,\dots,x_k)
        \coloneqq
        \frac{1}{k!}
        \sum_{\tau\in\mathfrak{S}_k}
        Z(x_{\tau(1)},\dots,x_{\tau(k)}).
    \]
    Then
    \[
        \hat{D}
        =
        \frac{1}{n}\sum_{i=1}^n
        \frac{1}{\binom{m}{k}}
        \sum_{1\le j_1<\cdots<j_k \le m}
        \tilde{Z}(X_{ij_1},\dots,X_{ij_k}).
    \]
    Also \(\E\tilde{Z}=D\), \(\|\tilde{Z}-D\|\le L\) almost surely, and for every \(u\in H\),
    \begin{align*}
        \ang{
            \E(\tilde{Z}-D)^2u,u
        }
        &=
        \E\|(\tilde{Z}-D)u\|_H^2\\
        &=
        \E\left\|
            \frac{1}{k!}
            \sum_{\tau\in\mathfrak{S}_k}
            \bigl(
                Z_\tau-D
            \bigr)u
        \right\|_H^2\\
        &\le
        \frac{1}{k!}
        \sum_{\tau\in\mathfrak{S}_k}
        \E\|(Z_\tau-D)u\|_H^2\\
        &=
        \E\|(Z-D)u\|_H^2\\
        &\le
        \ang{Su,u},
    \end{align*}
    where \(Z_\tau(x_1,\dots,x_k) \coloneqq Z(x_{\tau(1)},\dots,x_{\tau(k)})\).
    Thus we may assume from the outset that \(Z\) is symmetric.

    Write
    \[
        q \coloneqq \lfloor m/k\rfloor,
        \qquad
        \tilde{Z}_0 \coloneqq Z-D.
    \]
    For \(\tau\in\mathfrak{S}_m\), set
    \[
        Y_\tau
        \coloneqq
        \frac{1}{nq}
        \sum_{i=1}^n
        \sum_{\ell=1}^{q}
        \tilde{Z}_0 \bigl(
            X_{i,\tau(k(\ell-1)+1)},
            \dots,
            X_{i,\tau(k\ell)}
        \bigr).
    \]
    Then Hoeffding's blocking representation \citep{hoeffding1963ProbabilityInequalitiesSums} gives
    \[
        \hat{D}-D
        =
        \frac{1}{m!}
        \sum_{\tau\in\mathfrak{S}_m} Y_\tau.
    \]
    For each fixed \(\tau\), the operator \(Y_\tau\) is the average of \(nq\) i.i.d.\ centered self-adjoint Hilbert-Schmidt operators.
    Moreover,
    \[
        \|Y_\tau\|
        \le L
        \quad\text{a.s.},
        \qquad
        \sum_{i=1}^n \sum_{\ell=1}^q
        \E\bigl[
            \tilde{Z}_0(X_{i,\tau(k(\ell-1)+1)},\dots,X_{i,\tau(k\ell)})^2
        \bigr]
        \preceq
        nq S.
    \]

    We first work in the finite-dimensional matrix case \(H=\mathbb{C}^N\).
    Let
    \[
        \phi(x) \coloneqq e^x-x-1.
    \]
    For each fixed permutation \(\tau\), apply the matrix Laplace-transform estimate used in the proof of Minsker's Bernstein inequality with effective rank \citep[Theorem~3.1]{minsker2017ExtensionsBernsteinsInequality}; this estimate itself is based on Tropp's matrix Laplace-transform method \citep[Chapter~3]{tropp2015IntroductionMatrixConcentration}.
    It gives, for every \(\theta>0\),
    \[
        \E\Tr\phi \bigl(
            \theta nq Y_\tau
        \bigr)
        \le
        \Tr \left(
            \exp \left(
                \frac{\phi(\theta L)}{L^2} nq S
            \right)-I
        \right).
    \]
    Now diagonalize
    \[
        A
        \coloneqq
        \frac{\phi(\theta L)}{L^2} nq S.
    \]
    Since \(A\) is positive and each eigenvalue of \(A\) is at most \(\|A\|\), we have
    \[
        \Tr(e^A-I)
        =
        \sum_j (e^{\lambda_j(A)}-1)
        \le
        \frac{e^{\|A\|}}{\|A\|}
        \sum_j \lambda_j(A)
        =
        \frac{\Tr A}{\|A\|}e^{\|A\|}.
    \]
    Hence
    \begin{equation}\label{eq:appendix-operator-U-mgf-tau}
        \E\Tr\phi \bigl(
            \theta nq Y_\tau
        \bigr)
        \le
        \frac{\Tr S}{\|S\|}
        \exp \left(
            \frac{\phi(\theta L)}{L^2} nq \|S\|
        \right).
    \end{equation}

    Since \(\phi\) is convex, the map \(A\mapsto \Tr\phi(A)\) is convex on self-adjoint matrices; see, for example, the convex trace-function fact used in \citet{minsker2017ExtensionsBernsteinsInequality}.
    Therefore Jensen's inequality, now applied to Hoeffding's average over permutations, and \cref{eq:appendix-operator-U-mgf-tau} give
    \begin{align}
        \E\Tr\phi \bigl(
            \theta nq (\hat{D}-D)
        \bigr)
        &=
        \E\Tr\phi \left(
            \frac{\theta nq}{m!}
            \sum_{\tau\in\mathfrak{S}_m}Y_\tau
        \right)\notag\\
        &\le
        \frac{1}{m!}
        \sum_{\tau\in\mathfrak{S}_m}
        \E\Tr\phi \bigl(
            \theta nq Y_\tau
        \bigr)\notag\\
        &\le
        \frac{\Tr S}{\|S\|}
        \exp \left(
            \frac{\phi(\theta L)}{L^2} nq \|S\|
        \right).
        \label{eq:appendix-operator-U-mgf-hat}
    \end{align}

    The remaining finite-dimensional tail calculation is the standard Bernstein optimization following Minsker's proof; we include the details to keep track of the \(U\)-statistic effective sample size \(nq\).
    We now estimate the upper tail.
    Since \(\phi\) is increasing and nonnegative on \([0,\infty)\), Markov's inequality implies
    \begin{align*}
        & \bbP \left(
            \lambda_{\max}(\hat{D}-D)>t
        \right)\\
        =\ &
        \bbP \left(
            \lambda_{\max} \bigl(
                \theta nq(\hat{D}-D)
            \bigr)
            >
            \theta nq t
        \right)\\
        \le\ &
        \bbP \left(
            \Tr\phi \bigl(
                \theta nq(\hat{D}-D)
            \bigr)
            >
            \phi(\theta nq t)
        \right)\\
        \le\ &
        \frac{
            \E\Tr\phi \bigl(
                \theta nq(\hat{D}-D)
            \bigr)
        }{
            \phi(\theta nq t)
        }\\
        \le\ &
        \frac{\Tr S}{\|S\|}
        \exp \left(
            nq\left[
                \frac{\phi(\theta L)}{L^2}\|S\|
                -
                \theta t
            \right]
        \right)
        \frac{e^{\theta nq t}}{\phi(\theta nq t)},
    \end{align*}
    where in the last step we used \cref{eq:appendix-operator-U-mgf-hat}.

    Next, whenever \(\theta L<3\),
    \[
        \frac{\phi(\theta L)}{L^2}
        =
        \frac{1}{L^2}\sum_{j=2}^{\infty} \frac{(\theta L)^j}{j!}
        \le
        \frac{\theta^2}{2}
        \sum_{j=2}^{\infty}
        \left(
            \frac{\theta L}{3}
        \right)^{j-2}
        =
        \frac{\theta^2}{2(1-\theta L/3)}.
    \]
    Also, for every \(u>0\),
    \[
        \frac{e^u}{\phi(u)}
        =
        1+
        \frac{1+u}{e^u-u-1}
        \le
        1+\frac{6}{u^2}.
    \]
    Therefore, for \(\theta<3/L\),
    \begin{equation}\label{eq:appendix-operator-U-upper-tail-pre}
        \bbP \left(
            \lambda_{\max}(\hat{D}-D)>t
        \right)
        \le
        \frac{\Tr S}{\|S\|}
        \exp \left(
            nq\left[
                \frac{\theta^2\|S\|}{2(1-\theta L/3)}
                -
                \theta t
            \right]
        \right)
        \left(
            1+\frac{6}{(\theta nq t)^2}
        \right).
    \end{equation}

    Choose
    \[
        \theta \coloneqq \frac{t}{\|S\|+Lt/3}.
    \]
    Then \(\theta L<3\), and \cref{eq:appendix-operator-U-upper-tail-pre} becomes
    \begin{equation}\label{eq:appendix-operator-U-upper-tail-theta}
        \bbP \left(
            \lambda_{\max}(\hat{D}-D)>t
        \right)
        \le
        \frac{\Tr S}{\|S\|}
        \exp \left(
            -
            nq
            \frac{t^2/2}{\|S\|+Lt/3}
        \right)
        \left(
            1+\frac{6}{(\theta nq t)^2}
        \right).
    \end{equation}
    If
    \begin{equation}\label{eq:appendix-operator-U-large-t}
        \frac{t^2}{\|S\|+Lt/3}
        \ge
        \frac{\sqrt 6}{nq},
    \end{equation}
    then \(\theta nq t\ge \sqrt 6\), hence
    \[
        1+\frac{6}{(\theta nq t)^2}\le 2.
    \]
    Under \cref{eq:appendix-operator-U-large-t}, \cref{eq:appendix-operator-U-upper-tail-theta} reduces to
    \begin{equation}\label{eq:appendix-operator-U-upper-tail-final}
        \bbP \left(
            \lambda_{\max}(\hat{D}-D)>t
        \right)
        \le
        2\frac{\Tr S}{\|S\|}
        \exp \left(
            -
            nq
            \frac{t^2/2}{\|S\|+Lt/3}
        \right).
    \end{equation}

    Apply the same argument to \(-Z\).
    This gives the identical bound for
    \[
        \bbP \left(
            \lambda_{\max}(D-\hat{D})>t
        \right)
        =
        \bbP \left(
            \lambda_{\min}(\hat{D}-D)<-t
        \right).
    \]
    Hence, still in the matrix case,
    \begin{equation}\label{eq:appendix-operator-U-two-sided-matrix}
        \bbP \left(
            \|\hat{D}-D\|>t
        \right)
        \le
        4\frac{\Tr S}{\|S\|}
        \exp \left(
            -
            nq
            \frac{t^2/2}{\|S\|+Lt/3}
        \right)
    \end{equation}
    whenever \cref{eq:appendix-operator-U-large-t} holds.

    We now pass to general separable \(H\), following Minsker's finite-rank approximation argument for self-adjoint Hilbert-Schmidt operators.
    Let \(L_1 \subset L_2 \subset\cdots\) be finite-dimensional subspaces with dense union, and let \(P_j\) be the orthogonal projection onto \(L_j\).
    Since \(\hat{D}-D\) is Hilbert-Schmidt, we have
    \[
        \|(\hat{D}-D)-P_j(\hat{D}-D)P_j\|\to 0.
    \]
    Therefore, by Fatou's lemma,
    \[
        \bbP \left(
            \|\hat{D}-D\|>t
        \right)
        \le
        \liminf_{j\to\infty}
        \bbP \left(
            \|P_j(\hat{D}-D)P_j\|>t
        \right).
    \]
    Moreover,
    \[
        \E\bigl[
            (P_j(Z-D)P_j)^2
        \bigr]
        \preceq
        P_j S P_j,
    \]
    and
    \[
        \|P_j(Z-D)P_j\|\le L
        \quad\text{a.s.}
    \]
    Since \(S\) is trace-class, \(P_j S P_j \to S\) in trace norm, hence
    \[
        \Tr(P_j S P_j)\to \Tr S,
        \qquad
        \|P_j S P_j\|\to \|S\|.
    \]
    Applying the finite-dimensional estimate \cref{eq:appendix-operator-U-two-sided-matrix} to the compressed operators and then letting \(j\to\infty\), we obtain the same tail bound in the Hilbert-space case:
    \begin{equation}\label{eq:appendix-operator-U-two-sided}
        \bbP \left(
            \|\hat{D}-D\|>t
        \right)
        \le
        4\frac{\Tr S}{\|S\|}
        \exp \left(
            -
            nq
            \frac{t^2/2}{\|S\|+Lt/3}
        \right),
    \end{equation}
    again under \cref{eq:appendix-operator-U-large-t}.

    Finally, set
    \[
        u_\delta \coloneqq \ln\frac{4\Tr S}{\delta\|S\|}.
    \]
    If
    \[
        \frac{nq t^2}{2(\|S\|+Lt/3)}\ge u_\delta,
    \]
    then \cref{eq:appendix-operator-U-two-sided} gives \(\bbP(\|\hat{D}-D\|>t)\le\delta\).
    This inequality is equivalent to
    \[
        t^2-\frac{2u_\delta L}{3nq}t-\frac{2u_\delta\|S\|}{nq}\ge 0,
    \]
    so it is enough to require
    \[
        t
        \ge
        \frac{u_\delta L}{3nq}
        +
        \sqrt{
            \frac{2u_\delta\|S\|}{nq}
            +
            \frac{u_\delta^2 L^2}{9n^2 q^2}
        }.
    \]
    Since
    \[
        2u_\delta
        =
        2\ln\frac{4\Tr S}{\delta\|S\|}
        \ge
        2\ln 4
        \ge
        \sqrt 6,
    \]
    this choice of \(t\) also implies \cref{eq:appendix-operator-U-large-t}.
    Finally, use \(\sqrt{a+b}\le \sqrt a+\sqrt b\) to obtain the simpler sufficient condition
    \[
        t
        \ge
        \sqrt{
            \frac{2u_\delta\|S\|}{nq}
        }
        +
        \frac{2u_\delta L}{3nq}.
    \]
    This proves the claim.
\end{proof}

\begin{remark}
    \citet[Supplement, Theorem~E.3]{sriperumbudur2022ApproximateKernelPCA}
    also state a Bernstein-type inequality for a second-order operator-valued
    \(U\)-statistic.  We do not use that result here.  In their
    moment-generating-function argument, the equality immediately after the
    display marked \((*)\) replaces
    \[
        \sum_{i\ne j}
        (\delta_{X_i}-\delta_{X_i'})
        (\delta_{X_j}-\delta_{X_j'})Z
    \]
    by an expression involving two independent Rademacher sequences
    \((\varepsilon_i^{(1)})_i\) and \((\varepsilon_i^{(2)})_i\).
    Even if each difference \(\delta_{X_i}-\delta_{X_i'}\) were symmetric
    about zero, such a symmetrization would only introduce one Rademacher sign
    for that difference, not two independent signs.  Moreover, the preceding
    expectation over \((X_i')_{i=1}^n\) is taken conditionally on the fixed
    sample \((X_i)_{i=1}^n\); under this conditional law,
    \(\delta_{X_i}-\delta_{X_i'}\) is not symmetric about zero.  Therefore the
    Rademacherization step between \((*)\) and \((\dagger)\) is not justified.
    \Cref{lem:appendix-U-bernstein-operator} avoids this gap by inserting
    Hoeffding's blocking representation directly into the Tropp--Minsker
    Laplace-transform argument.
\end{remark}

\subsection{Vector-valued \texorpdfstring{\(U\)}{U}-statistics}

\begin{lemma}[Bernstein's inequality for vector-valued \(U\)-statistics]\label{lem:appendix-U-bernstein-vector}
    Let \((\mca X,P)\) be a measurable space, let \(H\) be a separable Hilbert space, and let
    \[
        \xi:\mca X^k \to H
    \]
    be measurable.
    Define
    \[
        \theta \coloneqq \E\xi(X_1,\dots,X_k),
    \]
    and
    \[
        \hat{\theta}
        \coloneqq
        \frac{1}{n}\sum_{i=1}^n
        \frac{1}{(m)_k}
        \sum_{\mathbf{j}\in I_m^k}
        \xi(X_{i,j_1},\dots,X_{i,j_k}),
    \]
    where \(X_{ij} \overset{\mr{i.i.d.}}{\sim}P\) and \(m\ge k\).
    Assume
    \[
        \|\xi-\theta\|_H\le L
        \quad\text{a.s.},
        \qquad
        \E\|\xi-\theta\|_H^2 \le \sigma^2.
    \]
    Then for every \(\delta\in(0,1)\), with probability at least \(1-\delta\),
    \[
        \|\hat{\theta}-\theta\|_H
        \le
        \sqrt{
            \frac{
                2\sigma^2 \ln(8/\delta)
            }{
                n\lfloor m/k\rfloor
            }
        }
        +
        \frac{
            2L\ln(8/\delta)
        }{
            3n\lfloor m/k\rfloor
        }.
    \]
\end{lemma}

\begin{proof}
    If \(\sigma=0\), then \(\xi=\theta\) almost surely, so the conclusion is trivial.
    Assume henceforth that \(\sigma>0\).

    Let \(H\oplus\R\) be the orthogonal direct sum, and for \(u\in H\) define
    \[
        \mca J(u)
        \coloneqq
        \begin{pmatrix}
            0 & u\\
            u^* & 0
        \end{pmatrix},
    \]
    where \(u:\R\to H\) is the rank-one operator \(a\mapsto au\), and
    \(u^*:H\to\R\) is its adjoint \(v\mapsto \ang{u,v}_H\).
    Then \(\mca J(u)\) is a self-adjoint Hilbert-Schmidt operator on \(H\oplus\R\), and the map
    \[
        \mca J:H\to \mca S(H\oplus\R)
    \]
    is linear.

    A direct computation shows that
    \[
        \mca J(u)^2
        =
        \begin{pmatrix}
            u\otimes u & 0\\
            0 & \|u\|_H^2
        \end{pmatrix}.
    \]
    Hence
    \[
        \|\mca J(u)\|=\|u\|_H,
    \]
    so \(\mca J\) is an isometry from \(H\) into \(\mca S(H\oplus\R)\).

    Apply \Cref{lem:appendix-U-bernstein-operator} to the operator-valued kernel
    \[
        Z \coloneqq \mca J\circ \xi.
    \]
    Its expectation is
    \[
        D
        =
        \E Z(X_1,\dots,X_k)
        =
        \mca J(\theta),
    \]
    by linearity of \(\mca J\).
    Also,
    \[
        \|Z-D\|
        =
        \|\mca J(\xi-\theta)\|
        =
        \|\xi-\theta\|_H
        \le L
        \quad\text{a.s.}
    \]

    Set
    \[
        S \coloneqq \E\bigl[(Z-D)^2\bigr].
    \]
    Since
    \[
        (Z-D)^2
        =
        \begin{pmatrix}
            (\xi-\theta)\otimes(\xi-\theta) & 0\\
            0 & \|\xi-\theta\|_H^2
        \end{pmatrix},
    \]
    we get
    \[
        S
        =
        \begin{pmatrix}
            \E[(\xi-\theta)\otimes(\xi-\theta)] & 0\\
            0 & \E\|\xi-\theta\|_H^2
        \end{pmatrix}.
    \]
    Therefore
    \[
        \|S\|
        =
        \E\|\xi-\theta\|_H^2
        \le \sigma^2,
    \]
    because the lower-right block already has norm \(\E\|\xi-\theta\|_H^2\), while the upper-left block has operator norm at most the same quantity.
    Moreover,
    \[
        \Tr S
        =
        \Tr\E\bigl[(Z-D)^2\bigr]
        =
        \E\Tr\bigl((Z-D)^2\bigr)
        =
        2\E\|\xi-\theta\|_H^2,
    \]
    because
    \[
        \Tr\bigl((\xi-\theta)\otimes(\xi-\theta)\bigr)
        =
        \|\xi-\theta\|_H^2,
    \]
    so
    \[
        \frac{\Tr S}{\|S\|}=2.
    \]
    Thus the effective rank in \Cref{lem:appendix-U-bernstein-operator} is \(d=2\).

    Finally, by linearity of \(\mca J\),
    \[
        \frac{1}{n}\sum_{i=1}^n
        \frac{1}{(m)_k}
        \sum_{\mathbf{j}\in I_m^k}
        Z(X_{i,j_1},\dots,X_{i,j_k})
        -
        D
        =
        \mca J(\hat{\theta}-\theta).
    \]
    Applying \Cref{lem:appendix-U-bernstein-operator} and using the isometry
    \(\|\mca J(u)\|=\|u\|_H\) gives
    \[
        \|\hat{\theta}-\theta\|_H
        \le
        \sqrt{
            \frac{
                2\|S\|\ln(8/\delta)
            }{
                n\lfloor m/k\rfloor
            }
        }
        +
        \frac{
            2L\ln(8/\delta)
        }{
            3n\lfloor m/k\rfloor
        }
        \le
        \sqrt{
            \frac{
                2\sigma^2 \ln(8/\delta)
            }{
                n\lfloor m/k\rfloor
            }
        }
        +
        \frac{
            2L\ln(8/\delta)
        }{
            3n\lfloor m/k\rfloor
        }.
    \]
    This proves the claim.
\end{proof}

\begin{remark}
    Exponential inequalities for Hilbert-valued \(U\)-statistics are available in the literature \citep{giraudo2025ExponentialInequalityHilbertvalued}.
    The particular form needed here follows directly from the operator-valued inequality above.
\end{remark}
 \section{Regular RKHS}
\label{app:regular-rkhs}

This appendix collects the regular-RKHS consequences and tensor eigenvalue estimates used in the main proofs.

\begin{lemma}\label{lem:appendix-supnorm}
    Suppose Assumption~\ref{ass:prelim-regular-rkhs} holds, and let
    \(f:[0,\lambda_1]\to[0,\infty)\) be non-decreasing with
    \(\sum_{j=1}^\infty f(\lambda_j)<\infty\).
    Then
    \[
        \sup_{x\in \mca{I}}
        \sum_{j=1}^{\infty} f(\lambda_j)e_j(x)^2
        \le
        M\sum_{j=1}^{\infty} f(\lambda_j).
    \]
    In particular, if \(s>1/\beta\), then
    \[
        M_s^2
        \coloneqq
        \sup_{x\in \mca{I}}\sum_{j=1}^\infty \lambda_j^s e_j(x)^2
        <\infty.
    \]
    Consequently, for every \(q\ge 1\) and every \(h\in ([H]^s)^{\otimes q}\),
    \[
        \|h\|_{\infty}
        \le
        M_s^q \|h\|_{([H]^s)^{\otimes q}}.
    \]
\end{lemma}

\begin{proof}
    Let \((\theta_m)_{m\ge1}\), \(d_m\), and \(k_m\) be as in
    Assumption~\ref{ass:prelim-regular-rkhs}.  Fix \(x\in \mca I\), and set
    \[
        b_m \coloneqq k_m(x,x),
        \qquad
        c_m \coloneqq M d_m,
        \qquad
        m\ge1.
    \]
    By Assumption~\ref{ass:prelim-regular-rkhs},
    \[
        \sum_{\ell=1}^m b_\ell
        =
        \sum_{\ell=1}^m k_\ell(x,x)
        \le
        M\sum_{\ell=1}^m d_\ell
        =
        \sum_{\ell=1}^m c_\ell,
        \qquad
        m\ge1.
    \]
    Since \((\theta_m)\) is decreasing and \(f\) is non-decreasing,
    \((f(\theta_m))_{m\ge1}\) is non-increasing.  Abel summation implies that
    for every \(N\ge1\),
    \[
        \sum_{m=1}^N f(\theta_m)k_m(x,x)
        =
        \sum_{m=1}^N f(\theta_m)b_m
        \le
        \sum_{m=1}^N f(\theta_m)c_m
        =
        M\sum_{m=1}^N f(\theta_m)d_m.
    \]
    Letting \(N\to\infty\) gives
    \[
        \sum_{j=1}^{\infty} f(\lambda_j)e_j(x)^2
        =
        \sum_{m=1}^{\infty} f(\theta_m)k_m(x,x)
        \le
        M\sum_{m=1}^{\infty} f(\theta_m)d_m
        =
        M\sum_{j=1}^{\infty} f(\lambda_j).
    \]
    Since \(x\) was arbitrary, the first claim follows.

    Apply this with \(f(x)=x^s\).
    Because \(\lambda_j \asymp j^{-\beta}\) and \(s>1/\beta\), we have
    \(\sum_{j=1}^\infty \lambda_j^s<\infty\), hence
    \[
        M_s^2
        \le
        M\sum_{j=1}^{\infty} \lambda_j^s
        <
        \infty.
    \]

    Finally, let \(h\in([H]^s)^{\otimes q}\) and write
    \[
        h(t_1,\dots,t_q)
        =
        \sum_{i_1,\dots,i_q=1}^{\infty}
        a_{i_1,\dots,i_q}
        \prod_{\ell=1}^q
        \lambda_{i_\ell}^{s/2} e_{i_\ell}(t_\ell),
        \qquad
        \sum_{i_1,\dots,i_q=1}^{\infty} a_{i_1,\dots,i_q}^2
        =
        \|h\|_{([H]^s)^{\otimes q}}^2.
    \]
    By Cauchy-Schwarz,
    \begin{align*}
        |h(t_1,\dots,t_q)|^2
         & \le
        \left(
        \sum_{i_1,\dots,i_q=1}^{\infty} a_{i_1,\dots,i_q}^2
        \right)
        \left(
        \sum_{i_1,\dots,i_q=1}^{\infty}
        \prod_{\ell=1}^q
        \lambda_{i_\ell}^{s} e_{i_\ell}(t_\ell)^2
        \right) \\
         & =
        \|h\|_{([H]^s)^{\otimes q}}^2
        \prod_{\ell=1}^q
        \left(
        \sum_{j=1}^{\infty} \lambda_j^s e_j(t_\ell)^2
        \right) \\
         & \le
        M_s^{2q}\|h\|_{([H]^s)^{\otimes q}}^2.
    \end{align*}
    Taking square roots proves the final inequality.
\end{proof}

The following standard estimate is Exercise~2.1.18(c) in
\citet{montgomery2006MultiplicativeNumberTheory}.

\begin{lemma}\label{lem:appendix-tensor-eigenvalue-counting}
    For \(R\ge1\), define
    \[
        A_r(R)
        \coloneqq
        \#\{(i_1,\dots,i_r)\in \bbN^r:\ i_1 \cdots i_r \le R\}.
    \]
    Then
    \[
        A_r(R)\asymp R(\ln(eR))^{r-1}.
    \]
\end{lemma}

\begin{lemma}\label{lem:appendix-log-power-integrals}
    Let \(y\ge 1\) and \(a,b\in\mathbb{R}\).
    If \(b>-1\), then, as \(y\to\infty\),
    \[
        \int_1^y x^{-a}(\ln x)^b\dd x
        \asymp
        \begin{cases}
            y^{1-a}(\ln y)^b, & a<1, \\
            (\ln y)^{b+1},    & a=1, \\
            1,                & a>1.
        \end{cases}
    \]
    Moreover,
    \[
        \int_y^\infty x^{-a}(\ln x)^b\dd x
        \asymp
        \begin{cases}
            y^{1-a}(\ln y)^b, & a>1,\ b\in\mathbb{R}, \\
            (\ln y)^{b+1},    & a=1,\ b<-1,
        \end{cases}
    \]
    in the parameter regimes where the tail integral is finite.  If
    \(a<1\), or if \(a=1\) and \(b\ge -1\), the tail integral is infinite.
\end{lemma}

\begin{proof}
    The change of variables \(t=\ln x\) gives
    \[
        \int_1^y x^{-a}(\ln x)^b\dd x
        =
        \int_0^{\ln{y}} e^{-(a-1)t} t^b \dd t,
        \qquad
        \int_y^\infty x^{-a}(\ln x)^b\dd x
        =
        \int_{\ln y}^\infty e^{-(a-1)t} t^b \dd t .
    \]
    The condition \(b>-1\) is exactly the local integrability condition at
    \(t=0\) for the first integral.

    We use the standard incomplete-gamma notation and large-variable
    asymptotics from \citet{NIST:DLMF}: the definitions of \(\gamma(s,z)\) and
    \(\Gamma(s,z)\) are given in
    \citet[\href{https://dlmf.nist.gov/8.2.E1}{(8.2.1)}--\href{https://dlmf.nist.gov/8.2.E2}{(8.2.2)}]{NIST:DLMF},
    and the fixed-parameter large-\(z\) expansion of \(\Gamma(s,z)\) is given
    in \citet[\href{https://dlmf.nist.gov/8.11.E2}{(8.11.2)}]{NIST:DLMF}.

    First consider the integral over \([1,y]\).  If \(a>1\), \(a-1>0\),
    \begin{align*}
        \int_0^{\ln y} e^{-(a-1)t} t^b \dd t
        = &
        (a-1)^{-b-1}\gamma(b+1,(a-1)\ln y) \\
        = &
        (a-1)^{-b-1}\Gamma(b+1)
        -
        (a-1)^{-b-1}\Gamma(b+1,(a-1)\ln y)
        >0.
    \end{align*}
    Since \(\Gamma(b+1,(a-1)\ln y)\to 0\), this integral converges to a
    positive finite constant, and hence is \(\asymp1\).

    If \(a=1\), the identity
    \[
        \int_0^{\ln y} t^b \dd t
        =
        \frac{(\ln y)^{b+1}}{b+1}
    \]
    gives the middle case.

    If \(a<1\), \(1-a>0\).  The contribution comes from the upper endpoint.
    Put \(L \coloneqq \ln y\).  For the upper bound, split the integral over
    \([0,L/2]\) and \([L/2,L]\).  The first part is bounded by
    \(C e^{(1-a)L/2} L^{b+1}\), which is
    \(O(e^{(1-a)L} L^b)\); on \([L/2,L]\), \(t^b \le C_b L^b\), and hence
    \[
        \int_{L/2}^{L} e^{(1-a)t} t^b \dd t
        \le
        C e^{(1-a)L} L^b
        =
        C y^{1-a}(\ln y)^b.
    \]
    For the lower bound, if \(L\ge2\), then
    \(t^b\ge c_bL^b\) for every \(t\in[L-1,L]\), where \(c_b>0\) depends
    only on \(b\).  Hence
    \[
        \int_{L-1}^{L} e^{(1-a)t} t^b \dd t
        \ge
        c_b e^{(1-a)(L-1)}L^b
        =
        C y^{1-a}(\ln y)^b.
    \]
    This proves the case \(a<1\).

    We now turn to the tail integral.  If \(a>1\), \(a-1>0\),
    \[
        \int_{\ln y}^\infty e^{-(a-1)t} t^b \dd t
        =
        (a-1)^{-b-1}\Gamma(b+1,(a-1)\ln y).
    \]
    The DLMF large-\(z\) asymptotic gives
    \[
        \Gamma(b+1,(a-1)\ln y)
        \sim
        e^{-(a-1)\ln y}((a-1)\ln y)^b,
    \]
    and therefore the tail is
    \(\asymp y^{1-a}(\ln y)^b\).  If \(a=1\), then
    \[
        \int_{\ln y}^\infty t^b \dd t
        =
        \frac{(\ln y)^{b+1}}{-b-1},
        \qquad
        b<-1.
    \]
    The same display also shows divergence when \(b\ge-1\); when \(a<1\),
    the exponential factor \(e^{(1-a)t}\) makes the tail diverge.
\end{proof}

\begin{lemma}[Tensor effective dimension]\label{lem:appendix-effective-dimension}
    Let
    \[
        \Lambda_{\mathbf{i}}
        \coloneqq
        \prod_{a=1}^r \lambda_{i_a},
        \qquad
        \mathbf{i}=(i_1,\dots,i_r)\in\bbN^r,
    \]
    and let \((\lambda_j^{(r)})_{j\ge1}\) be the non-increasing
    rearrangement of
    \(\{\Lambda_{\mathbf{i}}:\mathbf{i}\in\bbN^r\}\), counted with
    multiplicity.  Then, as \(u\downarrow0\),
    \[
        N_r(u)
        \coloneqq
        \#\{j\ge1:\lambda_j^{(r)} \ge u\}
        \asymp
        u^{-1/\beta}
        \left(\ln\frac{1}{u}\right)^{r-1}.
    \]
    Equivalently, for every \(j\ge1\),
    \[
        \lambda_j^{(r)}
        \asymp
        j^{-\beta}(\ln(ej))^{\beta(r-1)}.
    \]
    Moreover, for every fixed \(p\ge 1\), define
    \[
        \mcaN_p^{(r)}(\lambda)
        \coloneqq
        \sum_{i_1,\dots,i_r=1}^{\infty}
        \bigl(
        \frac{\Lambda_{\mathbf{i}}}{\Lambda_{\mathbf{i}}+\lambda}
        \bigr)^p.
    \]
    Then, as \(\lambda\downarrow0\),
    \[
        \mcaN_p^{(r)}(\lambda)
        \asymp
        \lambda^{-1/\beta}
        \left(\ln \frac{1}{\lambda}\right)^{r-1}.
    \]
    Moreover, if \(s>1/\beta\), then for every \(1\le a\le r\), as
    \(\lambda\downarrow0\),
    \[
        \sum_{i_1,\dots,i_r=1}^{\infty}
        \left(
        \prod_{\ell\neq a} \lambda_{i_\ell}^s
        \right)
        \bigl(
        \frac{\Lambda_{\mathbf{i}}}{\Lambda_{\mathbf{i}}+\lambda}
        \bigr)^p
        \asymp
        \lambda^{-1/\beta}.
    \]
\end{lemma}

\begin{proof}
    Since \(\lambda_j \asymp j^{-\beta}\), there are constants
    \(0<c\le C<\infty\) such that
    \[
        c^r(i_1 \cdots i_r)^{-\beta}
        \le
        \Lambda_{\mathbf{i}}
        \le
        C^r(i_1 \cdots i_r)^{-\beta}.
    \]
    Therefore, for all sufficiently small \(u>0\),
    \[
        A_r \left((c^r/u)^{1/\beta}\right)
        \le
        N_r(u)
        \le
        A_r \left((C^r/u)^{1/\beta}\right).
    \]
    The counting estimate in
    \Cref{lem:appendix-tensor-eigenvalue-counting} gives
    \[
        N_r(u)
        \asymp
        u^{-1/\beta}
        \left(\ln\frac{1}{u}\right)^{r-1},
        \qquad u\downarrow0.
    \]

    We next invert this counting estimate.  Put
    \[
        b_j
        \coloneqq
        j^{-\beta}(\ln(ej))^{\beta(r-1)}.
    \]
    For each fixed \(A>0\), as \(j\to\infty\),
    \begin{align*}
        N_r(Ab_j)
         & \asymp
        (Ab_j)^{-1/\beta}
        \left(\ln\frac{1}{Ab_j}\right)^{r-1} \\
         & =
        A^{-1/\beta} j(\ln(ej))^{-(r-1)}
        \left(\ln\frac{1}{Ab_j}\right)^{r-1}
        \asymp
        A^{-1/\beta} j,
    \end{align*}
    because \(\ln(1/(Ab_j))\asymp\ln(ej)\).  Choosing \(A\) large gives
    \(N_r(Ab_j)<j\) for all sufficiently large \(j\), and hence
    \(\lambda_j^{(r)} \le Ab_j\).  Choosing \(A\) small gives
    \(N_r(Ab_j)\ge j\) for all sufficiently large \(j\), and hence
    \(\lambda_j^{(r)} \ge Ab_j\).  Adjusting the constants to cover the
    finitely many remaining indices proves
    \[
        \lambda_j^{(r)}
        \asymp
        j^{-\beta}(\ln(ej))^{\beta(r-1)}.
    \]

    We now prove the effective-dimension estimate.  By the definition of
    \((\lambda_j^{(r)})_{j\ge1}\),
    \[
        \mcaN_p^{(r)}(\lambda)
        =
        \sum_{j=1}^{\infty}
        \left(
        \frac{\lambda_j^{(r)}}{\lambda_j^{(r)}+\lambda}
        \right)^p.
    \]
    Let \(N \coloneqq N_r(\lambda)\).  Since the sequence
    \((\lambda_j^{(r)})_{j\ge1}\) is non-increasing, the split at the level
    \(\lambda\) gives
    \begin{align*}
        \mcaN_p^{(r)}(\lambda)
         & =
        \sum_{j\le N}
        \left(
        \frac{\lambda_j^{(r)}}{\lambda_j^{(r)}+\lambda}
        \right)^p
        +
        \sum_{j>N}
        \left(
        \frac{\lambda_j^{(r)}}{\lambda_j^{(r)}+\lambda}
        \right)^p \\
         & \asymp
        N
        +
        \lambda^{-p}
        \sum_{j>N}
        \bigl(\lambda_j^{(r)} \bigr)^p .
    \end{align*}
    Indeed, if \(j\le N\), then
    \(\lambda_j^{(r)}/(\lambda_j^{(r)}+\lambda)\in[1/2,1]\); if \(j>N\),
    then
    \[
        \frac{\lambda_j^{(r)}}{2\lambda}
        \le
        \frac{\lambda_j^{(r)}}{\lambda_j^{(r)}+\lambda}
        \le
        \frac{\lambda_j^{(r)}}{\lambda}.
    \]
    The first summand already gives the lower bound
    \(N_r(\lambda)\asymp\lambda^{-1/\beta}
    (\ln(1/\lambda))^{r-1}\).  It remains to control the second summand.

    The eigenvalue bound proved above gives
    \(\lambda_j^{(r)} \asymp j^{-\beta}(\ln(ej))^{\beta(r-1)}\).  Hence
    \[
        \sum_{j>N}
        \bigl(\lambda_j^{(r)} \bigr)^p
        \asymp
        \sum_{j>N}
        j^{-\beta p}(\ln(ej))^{\beta p(r-1)}.
    \]
    The function \(x\mapsto x^{-\beta p}(\ln(ex))^{\beta p(r-1)}\) is eventually
    decreasing.  Hence, for small enough \(\lambda\), integral comparison gives
    \[
        \sum_{j>N}
        j^{-\beta p}(\ln(ej))^{\beta p(r-1)}
        \asymp
        \int_N^\infty x^{-\beta p}(\ln(ex))^{\beta p(r-1)}\dd x .
    \]
    Since \(\ln(ex)\asymp\ln x\) on \([N,\infty)\) for \(N\) large, the
    tail estimate in \Cref{lem:appendix-log-power-integrals}, with
    \(a=\beta p>1\) and \(b=\beta p(r-1)\), gives
    \[
        \int_N^\infty x^{-\beta p}(\ln(ex))^{\beta p(r-1)}\dd x
        \asymp
        N^{1-\beta p}(\ln(eN))^{\beta p(r-1)}.
    \]
    Since
    \[
        N
        \asymp
        \lambda^{-1/\beta}
        \left(\ln\frac{1}{\lambda}\right)^{r-1},
        \qquad
        \ln(eN)\asymp \ln\frac{1}{\lambda},
    \]
    it follows that
    \[
        \lambda^{-p}
        \sum_{j>N}
        \bigl(\lambda_j^{(r)} \bigr)^p
        \asymp
        \lambda^{-1/\beta}
        \left(\ln\frac{1}{\lambda}\right)^{r-1}.
    \]
    Together with the split above, this proves
    \[
        \mcaN_p^{(r)}(\lambda)
        \asymp
        \lambda^{-1/\beta}
        \left(
        \ln \frac{1}{\lambda}
        \right)^{r-1}.
    \]

    For the second claim, by symmetry it is enough to treat \(a=r\).
    Write
    \[
        \alpha_{i_1,\dots,i_{r-1}}
        \coloneqq
        \prod_{\ell=1}^{r-1} \lambda_{i_\ell}.
    \]
    Then
    \begin{align*}
            & \sum_{i_1,\dots,i_r=1}^{\infty}
        \left(
        \prod_{\ell=1}^{r-1} \lambda_{i_\ell}^s
        \right)
        \bigl(
        \frac{\Lambda_{\mathbf{i}}}{\Lambda_{\mathbf{i}}+\lambda}
        \bigr)^p                              \\
        =\  &
        \sum_{i_1,\dots,i_{r-1}=1}^{\infty}
        \left(
        \prod_{\ell=1}^{r-1} \lambda_{i_\ell}^s
        \right)
        \sum_{i_r=1}^{\infty}
        \left(
        \frac{\alpha_{i_1,\dots,i_{r-1}}\lambda_{i_r}}{\alpha_{i_1,\dots,i_{r-1}}\lambda_{i_r}+\lambda}
        \right)^p.
    \end{align*}
    The lower bound follows by keeping only the term \(i_1=\cdots=i_{r-1}=1\):
    \[
        \lambda_1^{s(r-1)}
        \sum_{i_r=1}^{\infty}
        \left(
        \frac{\lambda_{1}^{r-1} \lambda_{i_r}}
        {\lambda_{1}^{r-1} \lambda_{i_r}+\lambda}
        \right)^p
        \asymp
        \lambda^{-1/\beta},
    \]
    where we used the effective-dimension estimate just proved in the case
    \(r=1\), with regularization parameter \(\lambda/\lambda_1^{r-1}\).

    For the upper bound, note that for every \(\rho>0\),
    \begin{align*}
        \sum_{j=1}^{\infty}
        \left(
        \frac{\lambda_j}{\lambda_j+\rho}
        \right)^p
         & \le
        C\sum_{j=1}^{\infty}
        \frac{1}{(1+\rho j^\beta)^p}                     \\
         & \le
        C\int_0^\infty \frac{1}{(1+\rho x^\beta)^p}\dd x \\
         & =
        C\rho^{-1/\beta}
        \int_0^\infty \frac{1}{(1+y^\beta)^p}\dd y
        \le
        C_p \rho^{-1/\beta}.
    \end{align*}
    This estimate is valid for every \(\rho>0\); in particular, it does not require \(\rho\to0\).
    Applying it with \(\rho=\lambda/\alpha_{i_1,\dots,i_{r-1}}\),
    \begin{align*}
              & \sum_{i_1,\dots,i_r=1}^{\infty}
        \left(
        \prod_{\ell=1}^{r-1} \lambda_{i_\ell}^s
        \right)
        \bigl(
        \frac{\Lambda_{\mathbf{i}}}{\Lambda_{\mathbf{i}}+\lambda}
        \bigr)^p                                \\
        \le\  &
        C_p
        \sum_{i_1,\dots,i_{r-1}=1}^{\infty}
        \left(
        \prod_{\ell=1}^{r-1} \lambda_{i_\ell}^s
        \right)
        \left(
        \frac{\lambda}{\alpha_{i_1,\dots,i_{r-1}}}
        \right)^{-1/\beta}                      \\
        =\    &
        C_p \lambda^{-1/\beta}
        \sum_{i_1,\dots,i_{r-1}=1}^{\infty}
        \prod_{\ell=1}^{r-1} \lambda_{i_\ell}^{s+1/\beta}
        \asymp
        \lambda^{-1/\beta},
    \end{align*}
    because \(s>1/\beta\) implies \(\sum_{j=1}^\infty \lambda_j^{s+1/\beta}<\infty\).
\end{proof}

\begin{lemma}\label{lem:appendix-phi-T}
    For the population operator,
    \[
        \sup_{\bm{t}\in \mca{I}^r}
        \|T_{r,\lambda}^{-1} K(\bm{t})\|_{L^2}^2
        \le
        M^r \mcaN_2^{(r)}(\lambda),
    \]
    and
    \[
        \sup_{\bm{t}\in \mca{I}^r}
        \|T_{r,\lambda}^{-1/2} K(\bm{t})\|_{\caH^{\otimes r}}^2
        \le
        M^r \mcaN_1^{(r)}(\lambda).
    \]
\end{lemma}

\begin{proof}
    Since \(x\mapsto x/(x+\lambda)\) is increasing on \((0,\infty)\), both
    \[
        u\mapsto
        \left(
        \frac{Au}{Au+\lambda}
        \right)^2
        \qquad\text{and}\qquad
        u\mapsto
        \frac{Au}{Au+\lambda},
        \qquad A>0,
    \]
    are non-decreasing functions of \(u\).

    Using the spectral decomposition of \(T_r\),
    \[
        \|T_{r,\lambda}^{-1} K(\bm{t})\|_{L^2}^2
        =
        \sum_{\mathbf{i}=1}^\infty
        \bigl(
        \frac{\Lambda_{\mathbf{i}}}{\Lambda_{\mathbf{i}}+\lambda}
        \bigr)^2
        \prod_{\ell=1}^r e_{i_\ell}(t_\ell)^2,
    \]
    and
    \[
        \|T_{r,\lambda}^{-1/2} K(\bm{t})\|_{\caH^{\otimes r}}^2
        =
        \sum_{\mathbf{i}=1}^\infty
        \frac{\Lambda_{\mathbf{i}}}{\Lambda_{\mathbf{i}}+\lambda}
        \prod_{\ell=1}^r e_{i_\ell}(t_\ell)^2.
    \]
    We estimate the first display; the second one is identical.
    The functions displayed at the start of the proof are summable along
    \((\lambda_j)_{j\ge1}\), since they are bounded by a constant multiple of
    \(u\) and \(\sum_j \lambda_j<\infty\).
    For fixed \(i_2,\dots,i_r\), apply
    \Cref{lem:appendix-supnorm} with
    \[
        f(u)
        =
        \left(
        \frac{u\prod_{\ell=2}^r \lambda_{i_\ell}}
        {u\prod_{\ell=2}^r \lambda_{i_\ell}+\lambda}
        \right)^2.
    \]
    Here and below, an empty product is interpreted as \(1\).
    Repeating the same application coordinate by coordinate yields
    \[
        \sup_{\bm{t}\in \mca{I}^r}
        \sum_{\mathbf{i}=1}^\infty
        \bigl(
        \frac{\Lambda_{\mathbf{i}}}{\Lambda_{\mathbf{i}}+\lambda}
        \bigr)^2
        \prod_{\ell=1}^r e_{i_\ell}(t_\ell)^2
        \le
        M^r
        \sum_{\mathbf{i}=1}^\infty
        \bigl(
        \frac{\Lambda_{\mathbf{i}}}{\Lambda_{\mathbf{i}}+\lambda}
        \bigr)^2
        =
        M^r \mcaN_2^{(r)}(\lambda).
    \]
    The same iteration with
    \(\frac{\Lambda_{\mathbf{i}}}{\Lambda_{\mathbf{i}}+\lambda}\)
    gives the second inequality.
\end{proof}
 \section{Trace Bounds for the Noise-free Term}
\label{app:i0}

This appendix collects the four trace estimates used in \Cref{prop:I0-main}.
We separate the estimates into three parts:
the overlap decomposition of the double sums, the bounds for the two auxiliary kernels,
and finally the four trace estimates themselves.
Throughout this appendix, assume $m\ge 2r$.

\subsection{Overlap Decomposition}

For a measurable kernel \(h:\mca{I}^r \times \mca{I}^r \to \R\), define the separate symmetrization
\[
    h^{\mr{sym}}(\bm{t};\bm{t}')
    \coloneqq
    \frac{1}{(r!)^2}
    \sum_{\pi,\rho\in \mathfrak{S}_r}
    h\bigl(
    t_{\pi(1)},\dots,t_{\pi(r)};
    t_{\rho(1)}',\dots,t_{\rho(r)}'
    \bigr).
\]
For \(0\le q\le r\), define the \((2r-q)\)-variable kernel
\[
    h_q(t_1,\dots,t_{2r-q})
    \coloneqq
    h^{\mr{sym}}\bigl(
    (t_1,\dots,t_r);
    (t_1,\dots,t_q,t_{r+1},\dots,t_{2r-q})
    \bigr).
\]
For a \(q\)-variable function \(u\), write
\[
    U_{m,q}^{(i)}(u)
    \coloneqq
    \frac{1}{(m)_q}\sum_{\bm{\ell}\in I_m^q} u(t_{i\bm{\ell}}),
    \qquad
    U_{m,q}(u)
    \coloneqq
    \frac{1}{n}\sum_{i=1}^n U_{m,q}^{(i)}(u).
\]
We also abbreviate
\[
    c_{r,q,m}
    \coloneqq
    \binom{r}{q}^2 q! \frac{(m)_{2r-q}}{(m)_r^2},
    \qquad
    0\le q\le r.
\]

\begin{lemma}\label{lem:appendix-I0-overlap}
    For every measurable \(h:\mca{I}^r \times \mca{I}^r \to \R\) and every \(1\le i\le n\),
    \[
        \frac{1}{(m)_r^2}
        \sum_{\mathbf{j},\mathbf{k}\in I_m^r}
        h(t_{i\mathbf{j}};t_{i\mathbf{k}})
        =
        \sum_{q=0}^r
        c_{r,q,m}
        U_{m,2r-q}^{(i)}(h_q).
    \]
\end{lemma}

\begin{proof}
    We classify the ordered pairs \((\mathbf{j},\mathbf{k})\in I_m^r \times I_m^r\) by the overlap size
    \[
        q=q(\mathbf{j},\mathbf{k})=|J(\mathbf{j})\cap J(\mathbf{k})|.
    \]
    For fixed \(q\), there are \(\binom{r}{q}^2 q!\) ways to choose the overlapping positions and match them.
    Once this is done, the remaining \(2r-q\) distinct sample indices can be chosen in \((m)_{2r-q}\) ordered ways.
    Averaging over the separate permutations of the first and second \(r\)-blocks produces the symmetrized kernel \(h_q\), and the claimed identity follows after dividing by \((m)_r^2\).
\end{proof}

\subsection{Two Auxiliary Kernels}

For \(\bm{t},\bm{t}'\in\mca{I}^r\), define
\begin{align}
    g_{0,\lambda}(\bm{t};\bm{t}')
     &\coloneqq
    \Sigma_r(\bm{t};\bm{t}')
    \ang{
        T_{r,\lambda}^{-1} K(\bm{t}),
        T_{r,\lambda}^{-1} K(\bm{t}')
    }_{L^2(\mca{I}^r)},
    \label{eq:I0-gpop} \\
    f_{0,\lambda}(\bm{t};\bm{t}')
     &\coloneqq
    \Sigma_r(\bm{t};\bm{t}')
    \ang{
        T_{r,\lambda}^{-1/2} K(\bm{t}),
        T_{r,\lambda}^{-1/2} K(\bm{t}')
    }_{\caH^{\otimes r}}.
    \label{eq:I0-fpop}
\end{align}
By construction,
\begin{align}
    \Tr\bigl(
    T_r^{1/2} T_{r,\lambda}^{-1} \hat{G}_0 T_{r,\lambda}^{-1} T_r^{1/2}
    \bigr)
     & =
    \sum_{q=0}^r
    c_{r,q,m}
    U_{m,2r-q} \bigl((g_{0,\lambda})_q\bigr),
    \label{eq:appendix-g0-trace-expansion} \\
    \Tr\bigl(
    T_{r,\lambda}^{-1/2} \hat{G}_0 T_{r,\lambda}^{-1/2}
    \bigr)
     & =
    \sum_{q=0}^r
    c_{r,q,m}
    U_{m,2r-q} \bigl((f_{0,\lambda})_q\bigr). \label{eq:appendix-f0-trace-expansion}
\end{align}

\begin{lemma}\label{lem:appendix-I0-kernel-bounds}
    Let \(C_{\Sigma_r}\) be the constant supplied by
    Assumption~\ref{ass:source}\textup{(iv)}, so that
    \(\abs{\Sigma_r}\le C_{\Sigma_r}\).
    Then
    \begin{align*}
        \|g_{0,\lambda}\|_\infty
         & \le
        C_{\Sigma_r} M^r \mcaN_2^{(r)}(\lambda),
         &
        \E[g_{0,\lambda}^2]
         & \le
        C_{\Sigma_r}^2 \mcaN_4^{(r)}(\lambda), \\
        \|f_{0,\lambda}\|_\infty
         & \le
        C_{\Sigma_r} M^r \mcaN_1^{(r)}(\lambda),
         &
        \E[f_{0,\lambda}^2]
         & \le
        C_{\Sigma_r}^2 \mcaN_2^{(r)}(\lambda).
    \end{align*}
\end{lemma}

\begin{proof}
    The sup-norm bound for \(g_{0,\lambda}\) follows from Cauchy-Schwarz and \Cref{lem:appendix-phi-T}:
    \[
        |g_{0,\lambda}(\bm{t};\bm{t}')|
        \le
        C_{\Sigma_r}
        \|T_{r,\lambda}^{-1} K(\bm{t})\|_{L^2}
        \|T_{r,\lambda}^{-1} K(\bm{t}')\|_{L^2}
        \le
        C_{\Sigma_r} M^r \mcaN_2^{(r)}(\lambda).
    \]
    For the second moment, write
    \[
        \ang{
            T_{r,\lambda}^{-1} K(\bm{t}),
            T_{r,\lambda}^{-1} K(\bm{t}')
        }_{L^2}
        =
        \sum_{\mathbf{i}=1}^\infty
        \bigl(
        \frac{\Lambda_{\mathbf{i}}}{\Lambda_{\mathbf{i}}+\lambda}
        \bigr)^2
        \prod_{\ell=1}^r e_{i_\ell}(t_\ell)e_{i_\ell}(t_\ell'),
    \]
    hence orthonormality gives
    \[
        \E[g_{0,\lambda}^2]
        \le
        C_{\Sigma_r}^2
        \sum_{\mathbf{i}=1}^\infty
        \bigl(
        \frac{\Lambda_{\mathbf{i}}}{\Lambda_{\mathbf{i}}+\lambda}
        \bigr)^4
        =
        C_{\Sigma_r}^2 \mcaN_4^{(r)}(\lambda).
    \]
    The proof for \(f_{0,\lambda}\) is the same, using
    \[
        \ang{
            T_{r,\lambda}^{-1/2} K(\bm{t}),
            T_{r,\lambda}^{-1/2} K(\bm{t}')
        }_{\caH^{\otimes r}}
        =
        \sum_{\mathbf{i}=1}^\infty
        \frac{\Lambda_{\mathbf{i}}}{\Lambda_{\mathbf{i}}+\lambda}
        \prod_{\ell=1}^r e_{i_\ell}(t_\ell)e_{i_\ell}(t_\ell').
    \]
    Therefore
    \[
        \|f_{0,\lambda}\|_\infty
        \le
        C_{\Sigma_r} M^r \mcaN_1^{(r)}(\lambda),
        \qquad
        \E[f_{0,\lambda}^2]
        \le
        C_{\Sigma_r}^2 \mcaN_2^{(r)}(\lambda).
    \]
\end{proof}

\begin{lemma}\label{lem:appendix-I0-expectations}
    Recall from \Cref{lem:appendix-I0-overlap} that, for a kernel \(h\), the notation \(h_1\) means the one-overlap kernel obtained by forcing exactly one common design point between the two \(r\)-blocks.
    In particular, \((g_{0,\lambda})_1\) and \((f_{0,\lambda})_1\) are the one-overlap versions of \(g_{0,\lambda}\) and \(f_{0,\lambda}\).
    We have
    \[
        V_{\mathrm{curve}}
        =
        \E[g_{0,\lambda}]
        =
        O(1),
        \qquad
        \E[f_{0,\lambda}]
        =
        O(1),
    \]
    and
    \begin{gather*}
        r^2 \E[(g_{0,\lambda})_1]
        =
        V_{\mathrm{point}}(\lambda;\Sigma_r)
        +
        O(1),
        \qquad
        V_{\mathrm{point}}(\lambda;\Sigma_r)
        =
        O(\lambda^{-1/\beta}),
        \\
        \E[(f_{0,\lambda})_1]
        =
        O(\lambda^{-1/\beta}).
    \end{gather*}
    If \(\Sigma_r \neq 0\), then \(\E[g_{0,\lambda}]\) is bounded away from zero for all sufficiently small \(\lambda\).
\end{lemma}

\begin{proof}
    Expand \(\Sigma_r\) in the tensor basis:
    \[
        \Sigma_r(\bm{t};\bm{t}')
        =
        \sum_{\mathbf{j},\mathbf{k}=1}^\infty
        a_{\mathbf{j};\mathbf{k}}
        \prod_{\ell=1}^r \lambda_{j_\ell}^{s_1/2} e_{j_\ell}(t_\ell)
        \prod_{\ell=1}^r \lambda_{k_\ell}^{s_1/2} e_{k_\ell}(t_\ell').
    \]
    Then
    \[
        \E[g_{0,\lambda}]
        =
        \sum_{\mathbf{i}=1}^\infty
        a_{\mathbf{i};\mathbf{i}}
        \Lambda_{\mathbf{i}}^{s_1}
        \bigl(
        \frac{\Lambda_{\mathbf{i}}}{\Lambda_{\mathbf{i}}+\lambda}
        \bigr)^2
        \le
        \sum_{\mathbf{i}=1}^\infty
        a_{\mathbf{i};\mathbf{i}}\Lambda_{\mathbf{i}}^{s_1}
        <
        \infty
    \]
    because \(s_1>\frac{1}{\beta}\). The same computation with one power of \(\frac{\Lambda_{\mathbf{i}}}{\Lambda_{\mathbf{i}}+\lambda}\) gives \(\E[f_{0,\lambda}]=O(1)\).

    If \(\Sigma_r \neq 0\), then since \(\Sigma_r\) is positive semidefinite there exists some multi-index \(\mathbf{i}^\ast\) with \(a_{\mathbf{i}^\ast;\mathbf{i}^\ast}>0\).
    Therefore
    \[
        \E[g_{0,\lambda}]
        \ge
        a_{\mathbf{i}^\ast;\mathbf{i}^\ast}
        \Lambda_{\mathbf{i}^\ast}^{s_1}
        \left(
        \frac{\Lambda_{\mathbf{i}^\ast}}{\Lambda_{\mathbf{i}^\ast}+\lambda}
        \right)^2,
    \]
    which is bounded below by a positive constant when \(\lambda\) is small enough.

    It remains to estimate the one-overlap expectations.
    It follows from symmetry that
    \[
        \E[(g_{0,\lambda})_1]
        =
        \frac{1}{r}
        \E\bigl[
            g_{0,\lambda}(\xi,t_2,\dots,t_r;\xi,t_2',\dots,t_r')
            \bigr]
        +
        \frac{r-1}{r}
        \E\bigl[
            g_{0,\lambda}(\xi,t_2,\dots,t_r;t_1',\xi,t_3',\dots,t_r')
            \bigr],
    \]
    where the second term is absent when \(r=1\).
    Denote the first and second expectations by \(A_\lambda\) and
    \(B_\lambda\), respectively, with \(B_\lambda \coloneqq 0\) when \(r=1\).
    By the definition of
    \(V_{\mathrm{point}}(\lambda;\Sigma_r)\),
    \[
        V_{\mathrm{point}}(\lambda;\Sigma_r)
        =
        r A_\lambda.
    \]

    We begin with \(A_\lambda\).
    Expanding \(\Sigma_r\) and the kernel inner product, then using orthonormality in the variables \(t_2,\dots,t_r,t_2',\dots,t_r'\), we get
    \begin{align*}
        A_\lambda
        =\  & \E\bigl[
            g_{0,\lambda}(\xi,t_2,\dots,t_r;\xi,t_2',\dots,t_r')
        \bigr]         \\
        =\  &
        \int
        \sum_{\mathbf{i}=1}^\infty
        \sum_{j_1,k_1=1}^\infty
        a_{j_1,i_2,\dots,i_r;k_1,i_2,\dots,i_r}
        \\
        & \qquad\times
        \lambda_{j_1}^{s_1/2} \lambda_{k_1}^{s_1/2}
        e_{j_1}(\xi)e_{k_1}(\xi)
        \left(
        \prod_{\ell=2}^r \lambda_{i_\ell}^{s_1}
        \right)
        \bigl(
        \frac{\Lambda_{\mathbf{i}}}{\Lambda_{\mathbf{i}}+\lambda}
        \bigr)^2
        e_{i_1}(\xi)^2
         d\rho(\xi).
    \end{align*}
    Therefore
    \begin{align*}
        |A_\lambda|
        \le\  &
        \sup_{i_2,\dots,i_r,\xi}
        \left|
        \sum_{j_1,k_1=1}^\infty
        a_{j_1,i_2,\dots,i_r;k_1,i_2,\dots,i_r}
        \lambda_{j_1}^{s_1/2} \lambda_{k_1}^{s_1/2}
        e_{j_1}(\xi)e_{k_1}(\xi)
        \right|              \\
              & \qquad\times
        \int
        \sum_{\mathbf{i}=1}^\infty
        \left(
        \prod_{\ell=2}^r \lambda_{i_\ell}^{s_1}
        \right)
        \bigl(
        \frac{\Lambda_{\mathbf{i}}}{\Lambda_{\mathbf{i}}+\lambda}
        \bigr)^2
        e_{i_1}(\xi)^2
         d\rho(\xi).
    \end{align*}
    For each fixed \(i_2,\dots,i_r\), the coefficient block
    \[
        \bigl(
        a_{j_1,i_2,\dots,i_r;k_1,i_2,\dots,i_r}
        \bigr)_{j_1,k_1 \ge 1}
    \]
    is a partial \(\ell^2\)-sum of the coefficient tensor of \(\Sigma_r\), hence
    \[
        \sum_{j_1,k_1=1}^\infty
        a_{j_1,i_2,\dots,i_r;k_1,i_2,\dots,i_r}^2
        \le
        \|\Sigma_r\|_{([H]^{s_1})^{\otimes 2r}}^2.
    \]
    By Cauchy-Schwarz and \Cref{lem:appendix-supnorm},
    \begin{align*}
              & \left|
        \sum_{j_1,k_1=1}^\infty
        a_{j_1,i_2,\dots,i_r;k_1,i_2,\dots,i_r}
        \lambda_{j_1}^{s_1/2} \lambda_{k_1}^{s_1/2}
        e_{j_1}(\xi)e_{k_1}(\xi)
        \right|        \\
        \le\  &
        \left(
        \sum_{j_1,k_1=1}^\infty
        a_{j_1,i_2,\dots,i_r;k_1,i_2,\dots,i_r}^2
        \right)^{1/2}
        \left(
        \sum_{j_1=1}^\infty
        \lambda_{j_1}^{s_1} e_{j_1}(\xi)^2
        \right)^{1/2}
        \left(
        \sum_{k_1=1}^\infty
        \lambda_{k_1}^{s_1} e_{k_1}(\xi)^2
        \right)^{1/2}  \\
        \le\  &
        \|\Sigma_r\|_{([H]^{s_1})^{\otimes 2r}} M_{s_1}^2.
    \end{align*}
    Since \(\int e_{i_1}(\xi)^2=1\), we conclude that
    \begin{align*}
        |A_\lambda|
        \le\  &
        \|\Sigma_r\|_{([H]^{s_1})^{\otimes 2r}} M_{s_1}^2
        \sum_{\mathbf{i}=1}^\infty
        \left(
        \prod_{\ell=2}^r \lambda_{i_\ell}^{s_1}
        \right)
        \bigl(
        \frac{\Lambda_{\mathbf{i}}}{\Lambda_{\mathbf{i}}+\lambda}
        \bigr)^2
        =
        O(\lambda^{-1/\beta}),
    \end{align*}
    by the second part of \Cref{lem:appendix-effective-dimension} with \(a=1\) and \(p=2\).
    Consequently,
    \[
        V_{\mathrm{point}}(\lambda;\Sigma_r)
        =
        r A_\lambda
        =
        O(\lambda^{-1/\beta}).
    \]

    We now turn to \(B_\lambda\).  There is nothing to prove when
    \(r=1\), so assume \(r\ge2\).
    Expanding as above and integrating out the non-overlapping variables gives
    \begin{align*}
        B_\lambda
        =\  & \E\bigl[
            g_{0,\lambda}(\xi,t_2,\dots,t_r;t_1',\xi,t_3',\dots,t_r')
        \bigr]                   \\
        =\  &
        \sum_{\mathbf{i}=1}^\infty
        \sum_{j_1,k_2=1}^\infty
        a_{j_1,i_2,\dots,i_r;i_1,k_2,i_3,\dots,i_r}
        \lambda_{j_1}^{s_1/2} \lambda_{k_2}^{s_1/2}
        \lambda_{i_1}^{s_1/2} \lambda_{i_2}^{s_1/2}
        \left(
        \prod_{\ell=3}^r \lambda_{i_\ell}^{s_1}
        \right)
        \bigl(
        \frac{\Lambda_{\mathbf{i}}}{\Lambda_{\mathbf{i}}+\lambda}
        \bigr)^2
        \\
            & \qquad\times
        \int
        e_{j_1}(\xi)e_{k_2}(\xi)e_{i_1}(\xi)e_{i_2}(\xi)
         d\rho(\xi).
    \end{align*}
    We prove that \(B_\lambda=O(1)\).  Since
    \(\Lambda_{\mathbf{i}}/(\Lambda_{\mathbf{i}}+\lambda)\le1\), Cauchy-Schwarz
    gives
    \begin{align*}
        |B_\lambda|
        \le\  &
        \left(
        \sum_{\mathbf{i}=1}^\infty
        \sum_{j_1,k_2=1}^\infty
        a_{j_1,i_2,\dots,i_r;i_1,k_2,i_3,\dots,i_r}^2
        \right)^{1/2}                         \\
        &\times
        \Biggl(
        \sum_{\mathbf{i}=1}^\infty
        \sum_{j_1,k_2=1}^\infty
        \lambda_{j_1}^{s_1} \lambda_{k_2}^{s_1}
        \lambda_{i_1}^{s_1} \lambda_{i_2}^{s_1}
        \prod_{\ell=3}^r \lambda_{i_\ell}^{2s_1}
        \left(
        \int
        e_{j_1}(\xi)e_{k_2}(\xi)e_{i_1}(\xi)e_{i_2}(\xi)
         d\rho(\xi)
        \right)^2
        \Biggr)^{1/2}.
    \end{align*}
    The first factor is bounded by
    \[
        \|\Sigma_r\|_{([H]^{s_1})^{\otimes 2r}},
    \]
    because it is a partial \(\ell^2\)-sum of the coefficient tensor of
    \(\Sigma_r\).  In the second factor, the variables \(i_3,\dots,i_r\)
    appear only through
    \(\prod_{\ell=3}^r \lambda_{i_\ell}^{2s_1}\), and therefore their sum is finite
    since \(s_1>1/\beta\).  It remains to bound the nonnegative series
    \[
        S_4
        \coloneqq
        \sum_{i_1,i_2,j_1,k_2=1}^\infty
        \lambda_{i_1}^{s_1} \lambda_{i_2}^{s_1}
        \lambda_{j_1}^{s_1} \lambda_{k_2}^{s_1}
        \left(
        \int
        e_{i_1}(\xi)e_{i_2}(\xi)e_{j_1}(\xi)e_{k_2}(\xi)
         d\rho(\xi)
        \right)^2.
    \]
    Let
    \[
        L_{s_1}(x,y)
        \coloneqq
        \sum_{j=1}^\infty
        \lambda_j^{s_1} e_j(x)e_j(y).
    \]
    Expanding the fourth power for finite truncations of \(L_{s_1}\) gives,
    formally,
    \[
        S_4
        =
        \int_{\mca I} \int_{\mca I}
        L_{s_1}(x,y)^4 d\rho(x)d\rho(y).
    \]
    By Cauchy-Schwarz and \Cref{lem:appendix-supnorm},
    \[
        |L_{s_1}(x,y)|
        \le
        \left(
        \sum_{j=1}^\infty \lambda_j^{s_1} e_j(x)^2
        \right)^{1/2}
        \left(
        \sum_{j=1}^\infty \lambda_j^{s_1} e_j(y)^2
        \right)^{1/2}
        \le
        M_{s_1}^2.
    \]
    The same bound applies to the finite truncations, so dominated convergence
    justifies the preceding identity and gives
    \(S_4 \le M_{s_1}^8<\infty\).  Hence \(B_\lambda=O(1)\).
    Consequently,
    \[
        r^2 \E[(g_{0,\lambda})_1]
        =
        rA_\lambda+r(r-1)B_\lambda
        =
        V_{\mathrm{point}}(\lambda;\Sigma_r)
        +
        O(1).
    \]

    The same decomposition works for \((f_{0,\lambda})_1\).
    For the first pattern, one only replaces
    \(\bigl(\frac{\Lambda_{\mathbf{i}}}{\Lambda_{\mathbf{i}}+\lambda}\bigr)^2\)
    by
    \(\frac{\Lambda_{\mathbf{i}}}{\Lambda_{\mathbf{i}}+\lambda}\), and then
    applies the second part of \Cref{lem:appendix-effective-dimension} with
    \(p=1\).  For the off-diagonal pattern, the preceding \(O(1)\) argument is
    unchanged because it only uses
    \(\Lambda_{\mathbf{i}}/(\Lambda_{\mathbf{i}}+\lambda)\le1\).
    Thus \(\E[(f_{0,\lambda})_1]=O(\lambda^{-1/\beta})\).
\end{proof}

\begin{lemma}\label{lem:appendix-I0-vpoint-lower}
    Suppose Assumption~\ref{ass:prelim-ae-lower-regularity} holds.  If
    \(\Sigma_r \neq0\), then
    \[
        V_{\mathrm{point}}(\lambda;\Sigma_r)
        \asymp
        \lambda^{-1/\beta}.
    \]
\end{lemma}

\begin{proof}
    The upper bound is part of \Cref{lem:appendix-I0-expectations}, so it remains
    to prove the lower bound.  We use the notation from the proof of
    \Cref{lem:appendix-I0-expectations}; in particular,
    \(V_{\mathrm{point}}(\lambda;\Sigma_r)=rA_\lambda\).  Rearranging the
    expression for \(A_\lambda\) obtained there gives
    \begin{align*}
        A_\lambda
        =
        \sum_{i_2,\dots,i_r=1}^\infty
        \left(
        \prod_{\ell=2}^r \lambda_{i_\ell}^{s_1}
        \right)
        \int
        &\left(
        \sum_{j_1,k_1=1}^\infty
        a_{j_1,i_2,\dots,i_r;k_1,i_2,\dots,i_r}
        \lambda_{j_1}^{s_1/2} \lambda_{k_1}^{s_1/2}
        e_{j_1}(\xi)e_{k_1}(\xi)
        \right)
        \\
        &\times
        \left(
        \sum_{m=1}^\infty
        \left(
        \frac{\theta_m \prod_{\ell=2}^r \lambda_{i_\ell}}
             {\theta_m \prod_{\ell=2}^r \lambda_{i_\ell}+\lambda}
        \right)^2
        k_m(\xi,\xi)
        \right)
         d\rho(\xi),
    \end{align*}
    where, when \(r=1\), the product over \(\ell=2,\dots,r\) and the
    corresponding summation are omitted.

    For any fixed indices \(i_2,\dots,i_r\), the matrix
    \[
        \bigl(
        a_{j_1,i_2,\dots,i_r;k_1,i_2,\dots,i_r}
        \bigr)_{j_1,k_1 \ge1}
    \]
    is a principal submatrix of the positive semidefinite coefficient matrix of
    \(\Sigma_r\).  Therefore, for every \(\xi\),
    \[
        \sum_{j_1,k_1=1}^\infty
        a_{j_1,i_2,\dots,i_r;k_1,i_2,\dots,i_r}
        \lambda_{j_1}^{s_1/2} \lambda_{k_1}^{s_1/2}
        e_{j_1}(\xi)e_{k_1}(\xi)
        \ge0 .
    \]
    Since \(\Sigma_r \neq0\), the positive semidefinite coefficient matrix has a
    positive diagonal entry.  Choose
    \(\bm{i}=(i_1,\dots,i_r)\) such that
    \(a_{\bm{i};\bm{i}}>0\), and fix these indices \(i_2,\dots,i_r\) in what
    follows.

    Moreover,
    \[
        \begin{aligned}
        \int
        \sum_{j_1,k_1=1}^\infty
        a_{j_1,i_2,\dots,i_r;k_1,i_2,\dots,i_r}
        \lambda_{j_1}^{s_1/2} \lambda_{k_1}^{s_1/2}
        e_{j_1}(\xi)e_{k_1}(\xi)
         d\rho(\xi)
        &=
        \sum_{j_1=1}^\infty
        a_{j_1,i_2,\dots,i_r;j_1,i_2,\dots,i_r}
        \lambda_{j_1}^{s_1}
        \\
        &\ge
        a_{i_1,\dots,i_r;i_1,\dots,i_r}
        \lambda_{i_1}^{s_1}
        >
        0 .
    \end{aligned}
\]
    Since the integrand in the last display is nonnegative, choose
    \(\epsilon>0\) small enough that the same integral remains positive after
    restriction to any set of measure at least \(1-\epsilon\).  To apply
    Assumption~\ref{ass:prelim-ae-lower-regularity} directly, set
    \[
        \ell_N(\xi)
        =
        \frac{\sum_{m=1}^N k_m(\xi,\xi)}{\sum_{m=1}^N d_m},
        \qquad
        b_n(\xi)
        =
        \inf_{N\ge n}\ell_N(\xi).
    \]
    Then \(b_n(\xi)\uparrow\liminf_{N\to\infty}\ell_N(\xi)>0\) for
    \(\rho\)-almost every \(\xi\).  Consequently, the increasing sets
    \[
        E_n
        =
        \{\xi\in\mca I:b_n(\xi)\ge n^{-1}\}
    \]
    have union of full measure.  Choose \(N_\epsilon\) such that
    \(\rho(E_{N_\epsilon})\ge1-\epsilon\), and set
    \(E_\epsilon=E_{N_\epsilon}\) and
    \(c_\epsilon=N_\epsilon^{-1}\).  It follows that
    \[
        \sum_{m=1}^{N} k_m(\xi,\xi)
        \ge
        c_\epsilon
        \sum_{m=1}^{N} d_m,
        \qquad
        \xi\in E_\epsilon,\ N\ge N_\epsilon.
    \]

    The sequence
    \[
        \left(
        \frac{\theta_m \prod_{\ell=2}^r \lambda_{i_\ell}}
             {\theta_m \prod_{\ell=2}^r \lambda_{i_\ell}+\lambda}
        \right)^2,
        \qquad m\ge1,
    \]
    is non-increasing in \(m\).  Abel summation gives, uniformly for
    \(\xi\in E_\epsilon\),
    \[
        \sum_{m=1}^{\infty}
        \left(
        \frac{\theta_m \prod_{\ell=2}^r \lambda_{i_\ell}}
             {\theta_m \prod_{\ell=2}^r \lambda_{i_\ell}+\lambda}
        \right)^2
        k_m(\xi,\xi)
        \ge
        c_\epsilon
        \sum_{m=1}^{\infty}
        \left(
        \frac{\theta_m \prod_{\ell=2}^r \lambda_{i_\ell}}
             {\theta_m \prod_{\ell=2}^r \lambda_{i_\ell}+\lambda}
        \right)^2
        d_m
        -
        O_\epsilon(1).
    \]
    The \(O_\epsilon(1)\) term only removes the finitely many indices
    \(m<N_\epsilon\).
    Since
    \[
        \sum_{m=1}^{\infty}
        \left(
        \frac{\theta_m \prod_{\ell=2}^r \lambda_{i_\ell}}
             {\theta_m \prod_{\ell=2}^r \lambda_{i_\ell}+\lambda}
        \right)^2
        d_m
        =
        \sum_{j=1}^{\infty}
        \left(
        \frac{\lambda_j}
             {\lambda_j+\lambda/\prod_{\ell=2}^r \lambda_{i_\ell}}
        \right)^2
        \asymp
        \lambda^{-1/\beta},
    \]
    by \Cref{lem:appendix-effective-dimension} with \(r=1\) and \(p=2\), we
    have
    \[
        \sum_{m=1}^{\infty}
        \left(
        \frac{\theta_m \prod_{\ell=2}^r \lambda_{i_\ell}}
             {\theta_m \prod_{\ell=2}^r \lambda_{i_\ell}+\lambda}
        \right)^2
        k_m(\xi,\xi)
        \gtrsim
        \lambda^{-1/\beta},
        \qquad
        \xi\in E_\epsilon,
    \]
    for all sufficiently small \(\lambda\).  Keeping only the nonnegative
    summand indexed by the fixed \(i_2,\dots,i_r\) in the expression for
    \(A_\lambda\), and restricting the integral to \(E_\epsilon\), gives
    \[
        A_\lambda
        \gtrsim
        \lambda^{-1/\beta}.
    \]
    Hence
    \(V_{\mathrm{point}}(\lambda;\Sigma_r)=rA_\lambda
    \gtrsim \lambda^{-1/\beta}\), as desired.
\end{proof}

\subsection{The Four Trace Estimates}

\begin{proposition}\label{prop:appendix-I0-pop-phi}
    Assume \(\lambda\to0\) satisfies \((\ln \frac{1}{\lambda})^{r-1}=o(m)\).
    Then
    \[
        \Tr\bigl(
        T_r^{1/2} T_{r,\lambda}^{-1} G_0 T_{r,\lambda}^{-1} T_r^{1/2}
        \bigr)
        =
        V_{\mathrm{curve}}(\lambda;\Sigma_r)
        +
        \frac{1}{m}V_{\mathrm{point}}(\lambda;\Sigma_r)
        +
        o \left(
        \frac{\lambda^{-1/\beta}}{m}
        \right).
    \]
    In particular,
    \[
        \Tr\bigl(
        T_r^{1/2} T_{r,\lambda}^{-1} G_0 T_{r,\lambda}^{-1} T_r^{1/2}
        \bigr)
        =
        O \left(
        1+\frac{\lambda^{-1/\beta}}{m}
        \right).
    \]
\end{proposition}

\begin{proof}
    Take expectations in \Cref{eq:appendix-g0-trace-expansion}:
    \[
        \Tr\bigl(
        T_r^{1/2} T_{r,\lambda}^{-1} G_0 T_{r,\lambda}^{-1} T_r^{1/2}
        \bigr)
        =
        \sum_{q=0}^r
        c_{r,q,m}
        \E[(g_{0,\lambda})_q].
    \]
    Since
    \[
        \frac{(m)_{2r}}{(m)_r^2}
        =
        1+O(m^{-1}),
        \qquad
        \E[g_{0,\lambda}]
        =
        V_{\mathrm{curve}}(\lambda;\Sigma_r)
        =
        O(1),
    \]
    the \(q=0\) term is
    \[
        V_{\mathrm{curve}}(\lambda;\Sigma_r)
        +
        O(m^{-1})
        =
        V_{\mathrm{curve}}(\lambda;\Sigma_r)
        +
        o \left(
        \frac{\lambda^{-1/\beta}}{m}
        \right).
    \]
    Also,
    \[
        c_{r,1,m}
        =
        \frac{r^2}{m}
        +
        O(m^{-2}).
    \]
    By \Cref{lem:appendix-I0-expectations},
    \[
        r^2 \E[(g_{0,\lambda})_1]
        =
        V_{\mathrm{point}}(\lambda;\Sigma_r)+O(1),
        \qquad
        \E[(g_{0,\lambda})_1]
        =
        O(\lambda^{-1/\beta}).
    \]
    Hence the \(q=1\) term is
    \[
        c_{r,1,m} \E[(g_{0,\lambda})_1]
        =
        \frac{1}{m}V_{\mathrm{point}}(\lambda;\Sigma_r)
        +
        o \left(
        \frac{\lambda^{-1/\beta}}{m}
        \right).
    \]
    For \(q\ge 2\), \(|\E[(g_{0,\lambda})_q]|\le \|g_{0,\lambda}\|_\infty\), so
    \[
        \sum_{q=2}^r c_{r,q,m} \E[(g_{0,\lambda})_q]
        =
        O \left(
        \frac{\mcaN_2^{(r)}(\lambda)}{m^2}
        \right)
        =
        O \left(
        \frac{\lambda^{-1/\beta}(\ln \frac{1}{\lambda})^{r-1}}{m^2}
        \right)
        =
        o \left(
        \frac{\lambda^{-1/\beta}}{m}
        \right),
    \]
    by \Cref{lem:appendix-I0-kernel-bounds,lem:appendix-effective-dimension}.
    This proves the expansion, and the final \(O(\cdot)\) bound follows from \(\E[g_{0,\lambda}]=O(1)\).
\end{proof}

\begin{proposition}\label{prop:appendix-I0-fluctuation-phi}
    For every \(0<\delta<1\), with probability at least \(1-(r+1)\delta\),
    \begin{align*}
              & \left|
        \Tr\bigl(
        T_r^{1/2} T_{r,\lambda}^{-1}
        (\hat{G}_0-G_0)
        T_{r,\lambda}^{-1} T_r^{1/2}
        \bigr)
        \right|        \\
        \le\  &
        C\Biggl(
        \sqrt{
            \frac{\mcaN_4^{(r)}(\lambda)\ln(2/\delta)}{nm}
        }
        +
        \frac{\mcaN_2^{(r)}(\lambda)\ln(2/\delta)}{nm}
        +
        \frac{\mcaN_2^{(r)}(\lambda)}{m}
        \sqrt{
            \frac{\ln(2/\delta)}{nm}
        }
        \Biggr).
    \end{align*}
    Therefore,
    \[
        \begin{aligned}
             & \Tr\bigl(
            T_r^{1/2} T_{r,\lambda}^{-1}
            (\hat{G}_0-G_0)
            T_{r,\lambda}^{-1} T_r^{1/2}
            \bigr)       \\
             & \qquad=
            O_\bbP \left(
            \frac{\lambda^{-1/(2\beta)}(\ln \frac{1}{\lambda})^{(r-1)/2}}{m^{1/2} n^{1/2}}
            +
            \lambda^{-1/\beta}
            \left(\ln \frac{1}{\lambda}\right)^{r-1}
            \left(
            \frac{1}{mn}
            +
            \frac{1}{m^{3/2} n^{1/2}}
            \right)
            \right).
        \end{aligned}
    \]
\end{proposition}

\begin{proof}
    Subtract the expectation in \Cref{eq:appendix-g0-trace-expansion}:
    \[
        \Tr\bigl(
        T_r^{1/2} T_{r,\lambda}^{-1}
        (\hat{G}_0-G_0)
        T_{r,\lambda}^{-1} T_r^{1/2}
        \bigr)
        =
        \sum_{q=0}^r
        c_{r,q,m}
        \bigl(
        U_{m,2r-q}((g_{0,\lambda})_q)
        -
        \E[(g_{0,\lambda})_q]
        \bigr).
    \]
    For \(q=0\), \Cref{lem:appendix-U-bernstein-scalar} gives
    \[
        \left|
        U_{m,2r}(g_{0,\lambda})
        -
        \E[g_{0,\lambda}]
        \right|
        \le
        \sqrt{
            \frac{2\E[g_{0,\lambda}^2]\ln(2/\delta)}{n\lfloor m/(2r)\rfloor}
        }
        +
        \frac{2\|g_{0,\lambda}\|_\infty\ln(2/\delta)}{n\lfloor m/(2r)\rfloor}.
    \]
    For each \(q\ge 1\), \Cref{lem:appendix-U-hoeffding-scalar} yields
    \[
        \left|
        U_{m,2r-q}((g_{0,\lambda})_q)
        -
        \E[(g_{0,\lambda})_q]
        \right|
        \le
        \|g_{0,\lambda}\|_\infty
        \sqrt{
            \frac{2\ln(2/\delta)}{n\lfloor m/(2r-q)\rfloor}
        }.
    \]
    Insert these bounds, use \(c_{r,q,m}=O(m^{-q})\) for \(q\ge 1\), and then substitute the estimates from \Cref{lem:appendix-I0-kernel-bounds}.
    Finally use \Cref{lem:appendix-effective-dimension} to rewrite the result in powers of \(\lambda\).
\end{proof}

\begin{proposition}\label{prop:appendix-I0-pop-resolvent}
    Assume \(\lambda\to0\) satisfies \(\left(\ln \frac{1}{\lambda}\right)^{r-1}=o(m)\).
    Then
    \[
        \Tr\bigl(
        T_{r,\lambda}^{-1/2} G_0 T_{r,\lambda}^{-1/2}
        \bigr)
        =
        \frac{(m)_{2r}}{(m)_r^2}\E[f_{0,\lambda}]
        +
        O \left(
        \frac{\lambda^{-1/\beta}}{m}
        \right).
    \]
    In particular,
    \[
        \Tr\bigl(
        T_{r,\lambda}^{-1/2} G_0 T_{r,\lambda}^{-1/2}
        \bigr)
        =
        O \left(
        1+\frac{\lambda^{-1/\beta}}{m}
        \right).
    \]
\end{proposition}

\begin{proof}
    This is the same argument as in \Cref{prop:appendix-I0-pop-phi}, now using \Cref{eq:appendix-f0-trace-expansion} and the bounds for \(f_{0,\lambda}\) from \Cref{lem:appendix-I0-kernel-bounds,lem:appendix-I0-expectations}.
\end{proof}

\begin{proposition}\label{prop:appendix-I0-fluctuation-resolvent}
    For every \(0<\delta<1\), with probability at least \(1-(r+1)\delta\),
    \begin{align*}
              & \left|
        \Tr\bigl(
        T_{r,\lambda}^{-1/2}
        (\hat{G}_0-G_0)
        T_{r,\lambda}^{-1/2}
        \bigr)
        \right|        \\
        \le\  &
        C\Biggl(
        \sqrt{
            \frac{\mcaN_2^{(r)}(\lambda)\ln(2/\delta)}{nm}
        }
        +
        \frac{\mcaN_1^{(r)}(\lambda)\ln(2/\delta)}{nm}
        +
        \frac{\mcaN_1^{(r)}(\lambda)}{m}
        \sqrt{
            \frac{\ln(2/\delta)}{nm}
        }
        \Biggr).
    \end{align*}
    Therefore,
    \[
        \begin{aligned}
             & \Tr\bigl(
            T_{r,\lambda}^{-1/2}
            (\hat{G}_0-G_0)
            T_{r,\lambda}^{-1/2}
            \bigr)       \\
             & \qquad=
            O_\bbP \left(
            \frac{\lambda^{-1/(2\beta)}(\ln \frac{1}{\lambda})^{(r-1)/2}}{m^{1/2} n^{1/2}}
            +
            \lambda^{-1/\beta}
            \left(\ln \frac{1}{\lambda}\right)^{r-1}
            \left(
            \frac{1}{mn}
            +
            \frac{1}{m^{3/2} n^{1/2}}
            \right)
            \right).
        \end{aligned}
    \]
\end{proposition}

\begin{proof}
    Repeat the proof of \Cref{prop:appendix-I0-fluctuation-phi} with \(f_{0,\lambda}\) in place of \(g_{0,\lambda}\), and use \Cref{lem:appendix-I0-kernel-bounds}.
\end{proof}
 \section{Trace Bounds for the First Noise Term}
\label{app:i1}

This appendix records the bounds used in \Cref{prop:I1-main}.
As in the preceding trace appendix, it is helpful to separate the proof into
three layers:
the overlap decomposition, the estimates for the two auxiliary kernels,
and finally the four trace estimates that feed back into the main text.
Throughout this appendix, assume $m\ge 2r-1$.

\subsection{Overlap with One Distinguished Point}

We reuse the notation \(U_{m,q}^{(i)}\) and \(U_{m,q}\) from \Cref{lem:appendix-I0-overlap}.
For a measurable kernel \(h:\mca{I}\times \mca{I}^{r-1} \times \mca{I}^{r-1} \to\R\), define the separate symmetrization
\[
    h^{\mr{sym}}(\xi;\tilde{\bm{t}};\tilde{\bm{t}}')
    \coloneqq
    \frac{1}{((r-1)!)^2}
    \sum_{\pi,\rho\in \mathfrak{S}_{r-1}}
    h\bigl(
    \xi;
    t_{\pi(1)},\dots,t_{\pi(r-1)};
    t_{\rho(1)}',\dots,t_{\rho(r-1)}'
    \bigr).
\]
For \(0\le q\le r-1\), define the \((2r-1-q)\)-variable kernel
\[
    h_q(u_0,\dots,u_{2r-2-q})
    \coloneqq
    h^{\mr{sym}}\bigl(
    u_0;
    (u_1,\dots,u_{r-1});
    (u_1,\dots,u_q,u_r,\dots,u_{2r-2-q})
    \bigr).
\]
We also abbreviate
\[
    d_{r,q,m}
    \coloneqq
    \binom{r-1}{q}^2 q!
    \frac{(m)_{2r-1-q}}{m(m-1)_{r-1}^2},
    \qquad
    0\le q\le r-1.
\]

\begin{lemma}\label{lem:appendix-I1-overlap}
    For every measurable \(h:\mca{I}\times \mca{I}^{r-1} \times \mca{I}^{r-1} \to\R\) and every \(1\le i\le n\),
    \[
        \frac{1}{m(m-1)_{r-1}^2}
        \sum_{l=1}^m
        \sum_{\mathbf{j},\mathbf{k}\in I_{m,l}^{r-1}}
        h(t_{il};t_{i\mathbf{j}};t_{i\mathbf{k}})
        =
        \sum_{q=0}^{r-1}
        d_{r,q,m}
        U_{m,2r-1-q}^{(i)}(h_q).
    \]
\end{lemma}

\begin{proof}
    Fix \(l\).
    The two tuples \(\mathbf{j},\mathbf{k}\in I_{m,l}^{r-1}\) may overlap in exactly \(q\) indices, where \(0\le q\le r-1\).
    For fixed \(q\), there are \(\binom{r-1}{q}^2 q!\) ways to choose the overlapping positions and match them.
    Once the overlap pattern is fixed, the total number of distinct sample indices is \(2r-1-q\), so the remaining choices are counted by \((m)_{2r-1-q}\).
    Averaging over the separate permutations of the two \((r-1)\)-blocks produces \(h_q\), which gives the stated formula after normalization.
\end{proof}

\subsection{Two Auxiliary Kernels}

Recall that
\[
    S_r(\xi;\tilde{\bm{t}})
    =
    \sum_{a=1}^r
    K(t_1,\dots,t_{a-1},\xi,t_a,\dots,t_{r-1})
    \in
    \caH^{\otimes r}.
\]
For \(\xi\in\mca{I}\) and
\(\tilde{\bm{t}},\tilde{\bm{t}}'\in\mca{I}^{r-1}\), define
\begin{align}
    g_{1,\lambda}(\xi;\tilde{\bm{t}};\tilde{\bm{t}}')
     & \coloneqq
    \mu_{2r-2}(\tilde{\bm{t}};\tilde{\bm{t}}')
    \ang{
        T_{r,\lambda}^{-1} S_r(\xi;\tilde{\bm{t}}),
        T_{r,\lambda}^{-1} S_r(\xi;\tilde{\bm{t}}')
    }_{L^2},
    \label{eq:I1-gpop} \\
    f_{1,\lambda}(\xi;\tilde{\bm{t}};\tilde{\bm{t}}')
     & \coloneqq
    \mu_{2r-2}(\tilde{\bm{t}};\tilde{\bm{t}}')
    \ang{
        T_{r,\lambda}^{-1/2} S_r(\xi;\tilde{\bm{t}}),
        T_{r,\lambda}^{-1/2} S_r(\xi;\tilde{\bm{t}}')
    }_{\caH^{\otimes r}}.
    \label{eq:I1-fpop}
\end{align}

\begin{lemma}\label{lem:appendix-I1-kernel-bounds}
    Let \(C_{\mu_{2r-2}}\) be the constant supplied by
    Assumption~\ref{ass:source}\textup{(iv)}, so that
    \(\abs{\mu_{2r-2}}\le C_{\mu_{2r-2}}\).
    Then
    \[
        \|g_{1,\lambda}\|_\infty
        \le
        r^2 C_{\mu_{2r-2}} M^r \mcaN_2^{(r)}(\lambda),
    \]
    and
    \[
        \|f_{1,\lambda}\|_\infty
        \le
        r^2 C_{\mu_{2r-2}} M^r \mcaN_1^{(r)}(\lambda).
    \]
\end{lemma}

\begin{proof}
    For every \(\xi\in \mca I\) and \(\tilde{\bm{t}}\in \mca I^{r-1}\),
    \begin{align*}
        \|T_{r,\lambda}^{-1} S_r(\xi;\tilde{\bm{t}})\|_{L^2}
         & \le
        \sum_{a=1}^r
        \left\|
        T_{r,\lambda}^{-1}
        K(t_1,\dots,t_{a-1},\xi,t_a,\dots,t_{r-1})
        \right\|_{L^2} \\
         & \le
        r\sup_{\mathbf{u}\in \mca I^r}
        \|T_{r,\lambda}^{-1} K(\mathbf{u})\|_{L^2}
        \le
        r\sqrt{M^r \mcaN_2^{(r)}(\lambda)},
    \end{align*}
    by \Cref{lem:appendix-phi-T}.
    Therefore
    \begin{align*}
        |g_{1,\lambda}(\xi;\tilde{\bm{t}};\tilde{\bm{t}}')|
         & \le
        C_{\mu_{2r-2}}
        \|T_{r,\lambda}^{-1} S_r(\xi;\tilde{\bm{t}})\|_{L^2}
        \|T_{r,\lambda}^{-1} S_r(\xi;\tilde{\bm{t}}')\|_{L^2} \\
         & \le
        r^2 C_{\mu_{2r-2}} M^r \mcaN_2^{(r)}(\lambda).
    \end{align*}
    The proof for \(f_{1,\lambda}\) is identical, using
    \[
        \|T_{r,\lambda}^{-1/2} S_r(\xi;\tilde{\bm{t}})\|_{\caH^{\otimes r}}
        \le
        r\sup_{\mathbf{u}\in \mca I^r}
        \|T_{r,\lambda}^{-1/2} K(\mathbf{u})\|_{\caH^{\otimes r}}
        \le
        r\sqrt{M^r \mcaN_1^{(r)}(\lambda)}
    \]
    from \Cref{lem:appendix-phi-T}. This gives the second bound.
\end{proof}

\begin{lemma}\label{lem:appendix-I1-expectations}
    For \(1\le a,b\le r\), define
    \begin{align*}
        V_{\mathrm{noise},\lambda}^{(a,b)}
         & \coloneqq
        \E\Bigl[
            \mu_{2r-2}(\tilde{\bm{t}};\tilde{\bm{t}}')
            \left\langle
            \begin{aligned}
                 &
                T_{r,\lambda}^{-1}
                K(t_1,\dots,t_{a-1},\xi,t_a,\dots,t_{r-1}),
                \\
                 &
                T_{r,\lambda}^{-1}
                K(t_1',\dots,t_{b-1}',\xi,t_b',\dots,t_{r-1}')
            \end{aligned}
            \right\rangle_{L^2}
            \Bigr],
    \end{align*}
    where the expectation is taken over independent
    \(\xi,t_1,\dots,t_{r-1},t_1',\dots,t_{r-1}'\).
    By symmetry, \(V_{\mathrm{noise},\lambda}^{(a,a)}\) does not depend on \(a\).
    By definition
    \begin{equation}\label{eq:appendix-I1-diagonal-main}
        V_{\mathrm{noise}}(\lambda;\mu_{2r-2})
        =
        \sum_{a=1}^r V_{\mathrm{noise},\lambda}^{(a,a)}
        =
        r V_{\mathrm{noise},\lambda}^{(1,1)}.
    \end{equation}
    We have
    \[
        \E[g_{1,\lambda}]
        =
        V_{\mathrm{noise}}(\lambda;\mu_{2r-2})
        +
        O(1),
        \qquad
        V_{\mathrm{noise}}(\lambda;\mu_{2r-2})
        =
        O(\lambda^{-1/\beta}),
        \qquad
        \E[f_{1,\lambda}]
        =
        O(\lambda^{-1/\beta}).
    \]
    If \(\mu_{2r-2} \neq0\), then
    \[
        V_{\mathrm{noise}}(\lambda;\mu_{2r-2})
        \asymp
        \lambda^{-1/\beta},
        \qquad
        \E[g_{1,\lambda}]
        \asymp
        \lambda^{-1/\beta}.
    \]
\end{lemma}

\begin{proof}
    Expand \(\mu_{2r-2}\) in the tensor basis of
    \(([H]^{s_2})^{\otimes(2r-2)}\):
    \[
        \mu_{2r-2}(\tilde{\bm{t}};\tilde{\bm{t}}')
        =
        \sum_{\mathbf{j},\mathbf{k}\in \bbN^{r-1}}
        a_{\mathbf{j};\mathbf{k}}
        \prod_{\ell=1}^{r-1} \lambda_{j_\ell}^{s_2/2} e_{j_\ell}(t_\ell)
        \prod_{\ell=1}^{r-1} \lambda_{k_\ell}^{s_2/2} e_{k_\ell}(t_\ell').
    \]
    Since \(\mu_{2r-2} \in ([H]^{s_2})^{\otimes(2r-2)}\), the coefficient array satisfies
    \[
        \sum_{\mathbf{j},\mathbf{k}\in \bbN^{r-1}} a_{\mathbf{j};\mathbf{k}}^2 <\infty,
    \]
    hence every coefficient is bounded by
    \(\|\mu_{2r-2}\|_{([H]^{s_2})^{\otimes(2r-2)}}\).
    Moreover, because \(\mu_{2r-2}\) is a moment kernel, the matrix
    \((a_{\mathbf{j};\mathbf{k}})_{\mathbf{j},\mathbf{k}\in\bbN^{r-1}}\)
    is positive semidefinite.

    For \(\mathbf{i}=(i_1,\dots,i_r)\in\bbN^r\) and \(1\le a\le r\), write
    \[
        \mathbf{i}^{\setminus a}
        \coloneqq
        (i_1,\dots,i_{a-1},i_{a+1},\dots,i_r)\in \bbN^{r-1}.
    \]
    Expanding the insertion sum in \(g_{1,\lambda}\), then using orthonormality in
    \(\xi,\tilde{\bm{t}},\tilde{\bm{t}}'\), gives
    \begin{align}
        \E[g_{1,\lambda}]
        =
        \sum_{a,b=1}^r
        \sum_{\mathbf{i}\in\bbN^r}
        \mathbf{1}_{\{i_a=i_b\}}
        a_{\mathbf{i}^{\setminus a};\mathbf{i}^{\setminus b}}
        \Bigl(
        \prod_{\ell\neq a} \lambda_{i_\ell}^{s_2/2}
        \Bigr)
        \Bigl(
        \prod_{\ell\neq b} \lambda_{i_\ell}^{s_2/2}
        \Bigr)
        \bigl(
        \frac{\Lambda_{\mathbf{i}}}{\Lambda_{\mathbf{i}}+\lambda}
        \bigr)^2.
        \label{eq:appendix-I1-g-expectation}
    \end{align}
    The key simplification, compared with \(\E[(g_{0,\lambda})_1]\), is that the shared variable
    \(\xi\) only contributes the product \(e_{i_a}(\xi)e_{i_b}(\xi)\).
    After integration in \(\xi\), this turns into the indicator of the event \(i_a=i_b\).
    So no fourth power of a basis function appears here.

    The diagonal part \(a=b\) in \Cref{eq:appendix-I1-g-expectation} is precisely
    \(V_{\mathrm{noise}}(\lambda;\mu_{2r-2})\).
    Indeed, expanding \(V_{\mathrm{noise},\lambda}^{(a,a)}\) in the same way gives
    \[
        V_{\mathrm{noise}}(\lambda;\mu_{2r-2})
        =
        \sum_{a=1}^r
        \sum_{\mathbf{i}\in\bbN^r}
        a_{\mathbf{i}^{\setminus a};\mathbf{i}^{\setminus a}}
        \Bigl(
        \prod_{\ell\neq a} \lambda_{i_\ell}^{s_2}
        \Bigr)
        \bigl(
        \frac{\Lambda_{\mathbf{i}}}{\Lambda_{\mathbf{i}}+\lambda}
        \bigr)^2.
    \]
    Using \(|a_{\mathbf{j};\mathbf{j}}|\le \|\mu_{2r-2}\|_{([H]^{s_2})^{\otimes(2r-2)}}\),
    \begin{align*}
        V_{\mathrm{noise}}(\lambda;\mu_{2r-2})
        \le\  &
        \|\mu_{2r-2}\|_{([H]^{s_2})^{\otimes(2r-2)}}
        \sum_{a=1}^r
        \sum_{\mathbf{i}\in\bbN^r}
        \Bigl(
        \prod_{\ell\neq a} \lambda_{i_\ell}^{s_2}
        \Bigr)
        \bigl(
        \frac{\Lambda_{\mathbf{i}}}{\Lambda_{\mathbf{i}}+\lambda}
        \bigr)^2 \\
        =\    &
        O(\lambda^{-1/\beta}),
    \end{align*}
    by the second part of \Cref{lem:appendix-effective-dimension} with \(p=2\).

    Next consider the off-diagonal contribution \(a\neq b\).
    Parameterizing the equality constraint by the common value \(q \coloneqq i_a=i_b\)
    and the remaining \(r-2\) free indices \(u_1,\ldots,u_{r-2}\), we obtain
    \begin{align*}
              & \sum_{\substack{a,b=1 \\a\neq b}}^r
        \sum_{\mathbf{i}\in\bbN^r}
        \mathbf{1}_{\{i_a=i_b\}}
        \left|
        a_{\mathbf{i}^{\setminus a};\mathbf{i}^{\setminus b}}
        \right|
        \Bigl(
        \prod_{\ell\neq a} \lambda_{i_\ell}^{s_2/2}
        \Bigr)
        \Bigl(
        \prod_{\ell\neq b} \lambda_{i_\ell}^{s_2/2}
        \Bigr)
        \bigl(
        \frac{\Lambda_{\mathbf{i}}}{\Lambda_{\mathbf{i}}+\lambda}
        \bigr)^2                      \\
        \le\  &
        \|\mu_{2r-2}\|_{([H]^{s_2})^{\otimes(2r-2)}}
        \sum_{\substack{a,b=1         \\a\neq b}}^r
        \sum_{\mathbf{i}\in\bbN^r}
        \mathbf{1}_{\{i_a=i_b\}}
        \Bigl(
        \prod_{\ell\neq a} \lambda_{i_\ell}^{s_2/2}
        \Bigr)
        \Bigl(
        \prod_{\ell\neq b} \lambda_{i_\ell}^{s_2/2}
        \Bigr)                        \\
        =\    &
        \|\mu_{2r-2}\|_{([H]^{s_2})^{\otimes(2r-2)}}
        \sum_{\substack{a,b=1         \\a\neq b}}^r
        \sum_{q=1}^{\infty}
        \sum_{u_1,\dots,u_{r-2}=1}^{\infty}
        \lambda_q^{s_2}
        \prod_{\ell=1}^{r-2} \lambda_{u_\ell}^{s_2}
        <\infty,
    \end{align*}
    because \(s_2>1/\beta\) implies \(\sum_j \lambda_j^{s_2}<\infty\).
    Thus the off-diagonal contribution is \(O(1)\).
    Combining the diagonal and off-diagonal parts proves
    \[
        \E[g_{1,\lambda}]
        =
        V_{\mathrm{noise}}(\lambda;\mu_{2r-2})
        +
        O(1).
    \]

    The same decomposition works for \(f_{1,\lambda}\).
    One only replaces
    \(\bigl(\frac{\Lambda_{\mathbf{i}}}{\Lambda_{\mathbf{i}}+\lambda}\bigr)^2\)
    by
    \(\frac{\Lambda_{\mathbf{i}}}{\Lambda_{\mathbf{i}}+\lambda}\).
    The diagonal terms are then \(O(\lambda^{-1/\beta})\) by the second part of
    \Cref{lem:appendix-effective-dimension} with \(p=1\), while the off-diagonal terms are still \(O(1)\).
    Hence
    \[
        \E[f_{1,\lambda}]
        =
        O(\lambda^{-1/\beta}).
    \]

    Finally, assume \(\mu_{2r-2} \neq0\).
    Since the coefficient matrix is positive semidefinite and nonzero, there exists
    \(\mathbf{j}^\ast=(j_1^\ast,\dots,j_{r-1}^\ast)\in\bbN^{r-1}\)
    such that \(a_{\mathbf{j}^\ast;\mathbf{j}^\ast}>0\).
    Every diagonal coefficient \(a_{\mathbf{j};\mathbf{j}}\) is nonnegative.
    Therefore, for a lower bound it is enough to keep only the single diagonal piece
    \(a=b=1\) with \(i_2=j_1^\ast,\dots,i_r=j_{r-1}^\ast\):
    \[
        V_{\mathrm{noise}}(\lambda;\mu_{2r-2})
        \ge
        a_{\mathbf{j}^\ast;\mathbf{j}^\ast}
        \prod_{\ell=1}^{r-1} \lambda_{j_\ell^\ast}^{s_2}
        \sum_{q=1}^{\infty}
        \left(
        \frac{\prod_{\ell=1}^{r-1} \lambda_{j_\ell^\ast} \lambda_q}{\prod_{\ell=1}^{r-1} \lambda_{j_\ell^\ast} \lambda_q+\lambda}
        \right)^2.
    \]
    By the \(r=1\) case of \Cref{lem:appendix-effective-dimension}, applied with regularization parameter \(\lambda/\prod_{\ell=1}^{r-1} \lambda_{j_\ell^\ast}\),
    \[
        \sum_{q=1}^{\infty}
        \left(
        \frac{\prod_{\ell=1}^{r-1} \lambda_{j_\ell^\ast} \lambda_q}{\prod_{\ell=1}^{r-1} \lambda_{j_\ell^\ast} \lambda_q+\lambda}
        \right)^2
        \asymp
        \lambda^{-1/\beta}.
    \]
    Hence
    \[
        V_{\mathrm{noise}}(\lambda;\mu_{2r-2})
        \gtrsim
        \lambda^{-1/\beta}
    \]
    for all sufficiently small \(\lambda\).
    Together with the already proved upper bound, this gives
    \[
        V_{\mathrm{noise}}(\lambda;\mu_{2r-2})
        \asymp
        \lambda^{-1/\beta}.
    \]
    Since \(\E[g_{1,\lambda}]=V_{\mathrm{noise}}(\lambda;\mu_{2r-2})+O(1)\), we also have
    \(\E[g_{1,\lambda}]\asymp \lambda^{-1/\beta}\).
\end{proof}

\subsection{The Four Trace Estimates}

By \Cref{lem:appendix-I1-overlap}, applying the overlap decomposition to \(g_{1,\lambda}\) and \(f_{1,\lambda}\) gives
\begin{align}
    \Tr\bigl(
    T_r^{1/2} T_{r,\lambda}^{-1} \hat{G}_1 T_{r,\lambda}^{-1} T_r^{1/2}
    \bigr)
     & =
    \sum_{q=0}^{r-1}
    d_{r,q,m}
    U_{m,2r-1-q} \bigl((g_{1,\lambda})_q\bigr),
    \label{eq:appendix-g1-trace-expansion} \\
    \Tr\bigl(
    T_{r,\lambda}^{-1/2} \hat{G}_1 T_{r,\lambda}^{-1/2}
    \bigr)
     & =
    \sum_{q=0}^{r-1}
    d_{r,q,m}
    U_{m,2r-1-q} \bigl((f_{1,\lambda})_q\bigr).
    \label{eq:appendix-f1-trace-expansion}
\end{align}

\begin{proposition}\label{prop:appendix-I1-pop-phi}
    Assume \(\lambda\to0\) satisfies \(\left(\ln \frac{1}{\lambda}\right)^{r-1}=o(m)\).
    Then
    \[
        \Tr\bigl(
        T_r^{1/2} T_{r,\lambda}^{-1} G_1 T_{r,\lambda}^{-1} T_r^{1/2}
        \bigr)
        =
        \frac{(m)_{2r-1}}{m(m-1)_{r-1}^2}\E[g_{1,\lambda}]
        +
        O \left(
        \frac{\lambda^{-1/\beta}
            \left(\ln \frac{1}{\lambda}\right)^{r-1}}{m}
        \right).
    \]
    In particular, if \(\mu_{2r-2} \neq0\), then
    \[
        \Tr\bigl(
        T_r^{1/2} T_{r,\lambda}^{-1} G_1 T_{r,\lambda}^{-1} T_r^{1/2}
        \bigr)
        =
        \frac{(m)_{2r-1}}{m(m-1)_{r-1}^2}
        V_{\mathrm{noise}}(\lambda;\mu_{2r-2})
        +
        o(\lambda^{-1/\beta})
        \asymp
        \lambda^{-1/\beta}.
    \]
\end{proposition}

\begin{proof}
    Taking expectations in \Cref{eq:appendix-g1-trace-expansion} gives
    \[
        \Tr\bigl(
        T_r^{1/2} T_{r,\lambda}^{-1} G_1 T_{r,\lambda}^{-1} T_r^{1/2}
        \bigr)
        =
        \sum_{q=0}^{r-1}
        d_{r,q,m}
        \E[(g_{1,\lambda})_q].
    \]
    The \(q=0\) term is exactly
    \[
        \frac{(m)_{2r-1}}{m(m-1)_{r-1}^2}\E[g_{1,\lambda}].
    \]
    For \(q\ge1\),
    \[
        |\E[(g_{1,\lambda})_q]|
        \le
        \|g_{1,\lambda}\|_\infty,
    \]
    so
    \[
        \sum_{q=1}^{r-1}
        d_{r,q,m}
        \E[(g_{1,\lambda})_q]
        =
        O \left(
        \frac{\|g_{1,\lambda}\|_\infty}{m}
        \right),
    \]
    because \(d_{r,q,m}=O(m^{-q})\) for \(q\ge1\).
    By \Cref{lem:appendix-I1-kernel-bounds,lem:appendix-effective-dimension},
    \[
        \|g_{1,\lambda}\|_\infty
        =
        O \left(
        \lambda^{-1/\beta}
        \left(\ln \frac{1}{\lambda}\right)^{r-1}
        \right),
    \]
    which proves the first expansion.
    If \(\mu_{2r-2} \neq0\), then \Cref{lem:appendix-I1-expectations} gives
    \[
        \E[g_{1,\lambda}]
        =
        V_{\mathrm{noise}}(\lambda;\mu_{2r-2})
        +
        O(1),
        \qquad
        V_{\mathrm{noise}}(\lambda;\mu_{2r-2})
        \asymp
        \lambda^{-1/\beta}.
    \]
    Since \((\ln \frac{1}{\lambda})^{r-1}=o(m)\), the overlap remainder is
    \(o(\lambda^{-1/\beta})\), and the additional \(O(1)\) term is also
    \(o(\lambda^{-1/\beta})\). Therefore
    \[
        \Tr\bigl(
        T_r^{1/2} T_{r,\lambda}^{-1} G_1 T_{r,\lambda}^{-1} T_r^{1/2}
        \bigr)
        =
        \frac{(m)_{2r-1}}{m(m-1)_{r-1}^2}
        V_{\mathrm{noise}}(\lambda;\mu_{2r-2})
        +
        o(\lambda^{-1/\beta})
        \asymp
        \lambda^{-1/\beta}.
    \]
\end{proof}

\begin{proposition}\label{prop:appendix-I1-fluctuation-phi}
    For every \(0<\delta<1\), with probability at least \(1-r\delta\),
    \[
        \left|
        \Tr\bigl(
        T_r^{1/2} T_{r,\lambda}^{-1}
        (\hat{G}_1-G_1)
        T_{r,\lambda}^{-1} T_r^{1/2}
        \bigr)
        \right|
        \le
        C
        \mcaN_2^{(r)}(\lambda)
        \sqrt{
            \frac{\ln(2/\delta)}{mn}
        }.
    \]
    Therefore,
    \[
        \Tr\bigl(
        T_r^{1/2} T_{r,\lambda}^{-1}
        (\hat{G}_1-G_1)
        T_{r,\lambda}^{-1} T_r^{1/2}
        \bigr)
        =
        O_\bbP \left(
        \frac{\lambda^{-1/\beta}
            \left(\ln \frac{1}{\lambda}\right)^{r-1}}{\sqrt{mn}}
        \right).
    \]
\end{proposition}

\begin{proof}
    Subtract expectations in \Cref{eq:appendix-g1-trace-expansion}:
    \[
        \Tr\bigl(
        T_r^{1/2} T_{r,\lambda}^{-1}
        (\hat{G}_1-G_1)
        T_{r,\lambda}^{-1} T_r^{1/2}
        \bigr)
        =
        \sum_{q=0}^{r-1}
        d_{r,q,m}
        \bigl(
        U_{m,2r-1-q}((g_{1,\lambda})_q)
        -
        \E[(g_{1,\lambda})_q]
        \bigr).
    \]
    For each \(0\le q\le r-1\), \Cref{lem:appendix-U-hoeffding-scalar} gives, with probability at least \(1-\delta\),
    \[
        \left|
        U_{m,2r-1-q}((g_{1,\lambda})_q)
        -
        \E[(g_{1,\lambda})_q]
        \right|
        \le
        \|g_{1,\lambda}\|_\infty
        \sqrt{
            \frac{2\ln(2/\delta)}{n\lfloor m/(2r-1-q)\rfloor}
        }.
    \]
    Taking a union bound over \(q=0,\dots,r-1\), we obtain with probability at least \(1-r\delta\),
    \[
        \left|
        \Tr\bigl(
        T_r^{1/2} T_{r,\lambda}^{-1}
        (\hat{G}_1-G_1)
        T_{r,\lambda}^{-1} T_r^{1/2}
        \bigr)
        \right|
        \le
        C\|g_{1,\lambda}\|_\infty
        \sqrt{
            \frac{\ln(2/\delta)}{mn}
        }
        \sum_{q=0}^{r-1} d_{r,q,m}.
    \]
    Since \(\sum_{q=0}^{r-1} d_{r,q,m}=O(1)\), the first claim follows from \Cref{lem:appendix-I1-kernel-bounds}.
    The \(O_\bbP\)-statement is then immediate from \Cref{lem:appendix-effective-dimension}.
\end{proof}

\begin{proposition}\label{prop:appendix-I1-pop-resolvent}
    Assume \(\lambda\to0\) satisfies \(\left(\ln \frac{1}{\lambda}\right)^{r-1}=o(m)\).
    Then
    \[
        \Tr\bigl(
        T_{r,\lambda}^{-1/2} G_1 T_{r,\lambda}^{-1/2}
        \bigr)
        =
        \frac{(m)_{2r-1}}{m(m-1)_{r-1}^2}\E[f_{1,\lambda}]
        +
        O \left(
        \frac{\lambda^{-1/\beta}
            \left(\ln \frac{1}{\lambda}\right)^{r-1}}{m}
        \right),
    \]
    and in particular
    \[
        \Tr\bigl(
        T_{r,\lambda}^{-1/2} G_1 T_{r,\lambda}^{-1/2}
        \bigr)
        =
        O \left(
        \lambda^{-1/\beta}
        \right).
    \]
\end{proposition}

\begin{proof}
    Taking expectations in \Cref{eq:appendix-f1-trace-expansion}, we get
    \[
        \Tr\bigl(
        T_{r,\lambda}^{-1/2} G_1 T_{r,\lambda}^{-1/2}
        \bigr)
        =
        \sum_{q=0}^{r-1}
        d_{r,q,m}
        \E[(f_{1,\lambda})_q].
    \]
    The \(q=0\) term is exactly
    \[
        \frac{(m)_{2r-1}}{m(m-1)_{r-1}^2}\E[f_{1,\lambda}].
    \]
    For \(q\ge1\),
    \[
        |\E[(f_{1,\lambda})_q]|
        \le
        \|f_{1,\lambda}\|_\infty,
    \]
    so
    \[
        \sum_{q=1}^{r-1}
        d_{r,q,m}
        \E[(f_{1,\lambda})_q]
        =
        O \left(
        \frac{\|f_{1,\lambda}\|_\infty}{m}
        \right)
        =
        O \left(
        \frac{\lambda^{-1/\beta}
            \left(\ln \frac{1}{\lambda}\right)^{r-1}}{m}
        \right),
    \]
    by \Cref{lem:appendix-I1-kernel-bounds,lem:appendix-effective-dimension}.
    This proves the expansion.
    Since \(\E[f_{1,\lambda}]=O(\lambda^{-1/\beta})\) by \Cref{lem:appendix-I1-expectations}, the whole trace is \(O(\lambda^{-1/\beta})\).
\end{proof}

\begin{proposition}\label{prop:appendix-I1-fluctuation-resolvent}
    For every \(0<\delta<1\), with probability at least \(1-r\delta\),
    \[
        \left|
        \Tr\bigl(
        T_{r,\lambda}^{-1/2}
        (\hat{G}_1-G_1)
        T_{r,\lambda}^{-1/2}
        \bigr)
        \right|
        \le
        C
        \mcaN_1^{(r)}(\lambda)
        \sqrt{
            \frac{\ln(2/\delta)}{mn}
        }.
    \]
    Therefore,
    \[
        \Tr\bigl(
        T_{r,\lambda}^{-1/2}
        (\hat{G}_1-G_1)
        T_{r,\lambda}^{-1/2}
        \bigr)
        =
        O_\bbP \left(
        \frac{\lambda^{-1/\beta}
            \left(\ln \frac{1}{\lambda}\right)^{r-1}}{\sqrt{mn}}
        \right).
    \]
\end{proposition}

\begin{proof}
    Repeat the proof of \Cref{prop:appendix-I1-fluctuation-phi}, starting from \Cref{eq:appendix-f1-trace-expansion} and using \(\|f_{1,\lambda}\|_\infty\) in place of \(\|g_{1,\lambda}\|_\infty\).
    The bound from \Cref{lem:appendix-I1-kernel-bounds} then gives the claimed inequality, and \Cref{lem:appendix-effective-dimension} turns it into the stated \(O_\bbP\)-estimate.
\end{proof}
 \section{Omitted Proofs}
\label{app:omitted-proofs}

\par\noindent{\bf Proof of \Cref{ex:lower-regular-sphere}.}\par\noindent
Let \(\mca H_q^d(\bbS^d)\) be the space of degree-\(q\) spherical
harmonics, and write \(a_q=\dim \mca H_q^d(\bbS^d)\).  Its reproducing
kernel is
\[
    Z_q(x,y)
    =
    \sum_{\ell=1}^{a_q} Y_{q,\ell}(x)Y_{q,\ell}(y),
\]
for any orthonormal basis \((Y_{q,\ell})_{\ell=1}^{a_q}\).  The addition
formula for spherical harmonics gives
\(Z_q(x,x)=a_q\) for every \(x\in\bbS^d\)
\citep[Corollary~1.27]{dai2013ApproximationTheoryHarmonic}.  By the
Funk-Hecke formula, \(T\) acts on \(\mca H_q^d(\bbS^d)\) as multiplication
by a scalar \(\gamma_q\).  Therefore, for each positive eigenvalue
\(\theta_m\) of \(T\), the eigenspace \(V_m\) is the direct sum of the
harmonic spaces with \(\gamma_q=\theta_m\).  Since the diagonal of the
orthogonal projection kernel is additive over orthogonal direct sums,
\[
    k_m(x,x)
    =
    \sum_{\gamma_q=\theta_m} Z_q(x,x)
    =
    \sum_{\gamma_q=\theta_m} a_q
    =
    d_m,
    \qquad x\in\bbS^d.
\]
Hence, for every \(N\ge1\),
\[
    \sum_{m=1}^N k_m(x,x)
    =
    \sum_{m=1}^N d_m,
    \qquad x\in\bbS^d.
\]
This proves Assumption~\ref{ass:prelim-ae-lower-regularity}.  The same
identity also gives the regular RKHS condition.
\hfill\BlackBox\\[2mm]

\par\noindent{\bf Proof of \Cref{ex:lower-regular-ball}.}\par\noindent
Let \(V_q^d\) be the space of orthogonal polynomials of degree exactly
\(q\) on \(\bbB^d\) with respect to the weight
\(W(x)=(1-\norm{x}^2)^{-1/2}\), and set
\(a_q=\dim V_q^d=\binom{q+d-1}{q}\).  The Funk-Hecke formula on the ball
implies that \(T\) acts on \(V_q^d\) as multiplication by a scalar
\(\gamma_q\) for a dot-product kernel
\citep[Theorem~11.1.9]{dai2013ApproximationTheoryHarmonic}.  Let \(P_q\)
denote the reproducing kernel of \(V_q^d\).  Thus, for each positive
eigenvalue \(\theta_m\),
\[
    k_m(x,x)
    =
    \sum_{\gamma_q=\theta_m} P_q(x,x),
    \qquad
    d_m
    =
    \sum_{\gamma_q=\theta_m} a_q .
\]
With
\(\eta=(d-1)/2\) and \(u=2\norm{x}^2-1\), the explicit formula of
\citet[Corollary~11.1.8]{dai2013ApproximationTheoryHarmonic} yields
\[
    P_q(x,x)
    =
    \frac{q+\eta}{2\eta}
    \left\{
    C_q^\eta(1)+C_q^\eta(u)
    \right\},
    \qquad -1\le u\le1,
\]
where \(C_q^\eta\) is the Gegenbauer polynomial.  We use the standard
Gegenbauer bound \citep[Appendix~B.2]{dai2013ApproximationTheoryHarmonic}
\[
    \abs{C_q^\eta(u)}
    \le
    C_q^\eta(1),
    \qquad -1\le u\le1,
\]
and the standard value formula for \(C_q^\eta(1)\):
\[
    C_q^\eta(1)
    =
    \binom{q+d-2}{q},
    \qquad
    a_q
    =
    \binom{q+d-1}{q}
    =
    \frac{q+d-1}{d-1}\binom{q+d-2}{q}.
\]
Since \(\eta=(d-1)/2\), this gives the exact comparison
\[
    \frac{q+\eta}{\eta}C_q^\eta(1)
    =
    \frac{2q+d-1}{q+d-1}a_q .
\]
Fix \(x\neq0\), and set \(u=2\norm{x}^2-1\).  If \(\norm{x}<1\), then
\(u\in(-1,1)\), and the classical Darboux asymptotic formula for Gegenbauer polynomials
\citep[Appendix~B.2]{dai2013ApproximationTheoryHarmonic} gives
\[
    \frac{C_q^\eta(u)}{C_q^\eta(1)}
    =
    O(q^{-\eta})
    \to0 .
\]
If \(\norm{x}=1\), then \(u=1\).  In both cases, the displayed formula for
\(P_q(x,x)\) implies that there exist \(q_x\) and \(c_x>0\) such that
\[
    P_q(x,x)
    \ge
    c_x a_q,
    \qquad
    q\ge q_x .
\]
Using \(k_m(x,x)=\sum_{\gamma_q=\theta_m} P_q(x,x)\), we obtain, for this fixed
\(x\neq0\),
\[
    \sum_{m=1}^N k_m(x,x)
    \ge
    c_x
    \sum_{m=1}^N d_m
    -
    O(1),
\]
where the \(O(1)\) term accounts for the finitely many \(q<q_x\) that may contribute to the sum.
Since \(\sum_{m=1}^N d_m \to\infty\), this yields
\[
    \liminf_{N\to\infty}
    \frac{\sum_{m=1}^N k_m(x,x)}
    {\sum_{m=1}^N d_m}
    \ge
    c_x,
    \qquad
    x\neq0 .
\]
Since \(\rho(\{0\})=0\), Assumption~\ref{ass:prelim-ae-lower-regularity}
follows.
\hfill\BlackBox\\[2mm]

\par\noindent{\bf Proof of \Cref{ex:exact-power-law-mu}.}\par\noindent
Fix \(s'<s\).  By the tensor-product definition of the interpolation
spaces,
\[
    \norm{\mu_r}_{([H]^{s'})^{\otimes r}}^2
    =
    \sum_{\bm{i}\in\bbN^r} c_{\bm{i}}^2 \Lambda_{\bm{i}}^{-s'}.
\]
Since \(c_{\bm{i}}^2 \asymp \Lambda_{\bm{i}}^{s+\frac{1}{\beta}}\), the last
display is bounded by a constant multiple of
\[
    \sum_{\bm{i}\in\bbN^r}
    \Lambda_{\bm{i}}^{s+\frac{1}{\beta}-s'}
    =
    \left(
    \sum_{j=1}^{\infty} \lambda_j^{s+\frac{1}{\beta}-s'}
    \right)^r
    <\infty,
\]
because \(s+\frac{1}{\beta}-s'>1/\beta\).  Hence
\(\mu_r \in([H]^{s'})^{\otimes r}\), which proves
Assumption~\ref{ass:source}\textup{(i)}.

If \(s\ge2\), Assumption~\ref{ass:source}\textup{(ii)} is void.  Suppose
now that \(s<2\).  Let \(N=N_r(\lambda)\), and let
\((\lambda_j^{(r)})_{j\ge1}\) be the non-increasing rearrangement of
\(\{\Lambda_{\bm{i}}:\bm{i}\in\bbN^r\}\).  By
\Cref{lem:appendix-effective-dimension},
\[
    N
    \asymp
    \lambda^{-1/\beta}
    \left(\ln\frac{1}{\lambda}\right)^{r-1},
    \qquad
    \lambda_j^{(r)}
    \asymp
    j^{-\beta}(\ln(ej))^{\beta(r-1)}.
\]
Therefore,
\begin{align*}
    \sum_{\Lambda_{\bm{i}}<\lambda}c_{\bm{i}}^2
     & \asymp
    \sum_{j>N} \bigl(\lambda_j^{(r)} \bigr)^{s+1/\beta}                 \\
     & \asymp
    \sum_{j>N} j^{-\beta (s+1/\beta)}(\ln(ej))^{\beta(s+1/\beta)(r-1)} \\
     & \asymp
    \int_N^\infty x^{-(s\beta+1)}(\ln(ex))^{(s\beta+1)(r-1)}\dd x .
\end{align*}
Since \(\beta s+1>1\), the tail estimate in
\Cref{lem:appendix-log-power-integrals} yields
\[
    \int_N^\infty x^{-(s\beta+1)}(\ln(ex))^{(s\beta+1)(r-1)}\dd x
    \asymp
    N^{-\beta s}(\ln(eN))^{(\beta s+1)(r-1)}.
\]
Using the estimate for \(N\) and
\(\ln(eN)\asymp\ln(1/\lambda)\), we obtain
\[
    \sum_{\Lambda_{\bm{i}}<\lambda}c_{\bm{i}}^2
    \asymp
    \lambda^s
    \left(\ln\frac{1}{\lambda}\right)^{r-1}
    =
    \Omega(\lambda^s).
\]
This proves Assumption~\ref{ass:source}\textup{(ii)}.
\hfill\BlackBox\\[2mm]

\par\noindent{\bf Proof of \Cref{ex:bias-borderline-coefficients}.}\par\noindent
By the coefficient assumption,
\[
    B(\lambda;\mu_r)^2
    \asymp
    \sum_{\mathbf{i}=1}^{\infty}
    \Lambda_{\mathbf{i}}^{s+1/\beta}
    \left(
    \frac{\lambda}{\Lambda_{\mathbf{i}}+\lambda}
    \right)^2.
\]

Let \((\lambda_j^{(r)})_{j\ge1}\) be the non-increasing rearrangement of
the tensor eigenvalues \(\{\Lambda_{\mathbf{i}}:\mathbf{i}\in\bbN^r\}\),
and define
\[
    N
    \coloneqq
    \#\{j\ge1:\lambda_j^{(r)} \ge\lambda\}.
\]
By \Cref{lem:appendix-effective-dimension},
\[
    N
    \asymp
    \lambda^{-1/\beta}
    \left(\ln\frac{1}{\lambda}\right)^{r-1},
    \qquad
    \ln(eN)\asymp\ln\frac{1}{\lambda}.
\]

We next record the two summation estimates needed below.  By
\Cref{lem:appendix-effective-dimension},
\[
    \bigl(\lambda_j^{(r)} \bigr)^q
    \asymp
    j^{-\beta q}(\ln(ej))^{\beta q(r-1)}
\]
for each fixed \(q\in\mathbb{R}\).  Integral comparison gives
\[
    \sum_{j\le N} \bigl(\lambda_j^{(r)} \bigr)^q
    \asymp
    1+
    \int_1^N
    x^{-\beta q}(\ln(ex))^{\beta q(r-1)}
    \dd x.
\]
Applying \Cref{lem:appendix-log-power-integrals} to the shifted-log
integral, equivalently after discarding the bounded initial interval
and using \(\ln(ex)\asymp\ln x\) for \(x\ge2\), yields
\[
    \sum_{j\le N} \bigl(\lambda_j^{(r)} \bigr)^q
    \asymp
    \begin{cases}
        N^{1-\beta q}(\ln(eN))^{\beta q(r-1)},
         & q<1/\beta, \\
        (\ln(eN))^r,
         & q=1/\beta, \\
        1,
         & q>1/\beta,
    \end{cases}
\]
where the constant \(1\) accounts for the finitely many small indices in
the convergent case.  Similarly, for \(q>1/\beta\), another integral
comparison gives
\[
    \sum_{j>N} \bigl(\lambda_j^{(r)} \bigr)^q
    \asymp
    \int_N^\infty
    x^{-\beta q}(\ln(ex))^{\beta q(r-1)}
    \dd x.
\]
Applying the tail part of \Cref{lem:appendix-log-power-integrals} gives
\[
    \sum_{j>N} \bigl(\lambda_j^{(r)} \bigr)^q
    \asymp
    N^{1-\beta q}(\ln(eN))^{\beta q(r-1)}.
\]

Split the bias according to the two regions
\(\Lambda_{\mathbf{i}}\ge\lambda\) and \(\Lambda_{\mathbf{i}}<\lambda\):
\[
    B(\lambda;\mu_r)^2
    \asymp
    B_{\ge}(\lambda)+B_{<}(\lambda),
\]
where
\[
    B_{\ge}(\lambda)
    \coloneqq
    \sum_{\Lambda_{\mathbf{i}}\ge\lambda}
    \Lambda_{\mathbf{i}}^{s+1/\beta}
    \left(
    \frac{\lambda}{\Lambda_{\mathbf{i}}+\lambda}
    \right)^2
\]
and
\[
    B_{<}(\lambda)
    \coloneqq
    \sum_{\Lambda_{\mathbf{i}}<\lambda}
    \Lambda_{\mathbf{i}}^{s+1/\beta}
    \left(
    \frac{\lambda}{\Lambda_{\mathbf{i}}+\lambda}
    \right)^2.
\]
On the first region,
\[
    B_{\ge}(\lambda)
    \asymp
    \lambda^2
    \sum_{\Lambda_{\mathbf{i}}\ge\lambda}
    \Lambda_{\mathbf{i}}^{s+1/\beta-2}
    =
    \lambda^2
    \sum_{j\le N}
    \bigl(\lambda_j^{(r)} \bigr)^{s+1/\beta-2}.
\]
Applying the partial-sum estimate above with
\(q=s+1/\beta-2\), we obtain
\[
    B_{\ge}(\lambda)
    \asymp
    \begin{cases}
        \lambda^s \left(\ln\frac{1}{\lambda}\right)^{r-1},
         & s<2, \\
        \lambda^2 \left(\ln\frac{1}{\lambda}\right)^r,
         & s=2, \\
        \lambda^2,
         & s>2.
    \end{cases}
\]

On the second region,
\[
    B_{<}(\lambda)
    \asymp
    \sum_{\Lambda_{\mathbf{i}}<\lambda}
    \Lambda_{\mathbf{i}}^{s+1/\beta}
    =
    \sum_{j>N}
    \bigl(\lambda_j^{(r)} \bigr)^{s+1/\beta}
    \asymp
    \lambda^s
    \left(
    \ln\frac{1}{\lambda}
    \right)^{r-1}.
\]
Since this term is of lower order than \(B_{\ge}(\lambda)\) when
\(s\ge2\), and has the same order when \(s<2\), we conclude that
\[
    B(\lambda;\mu_r)^2
    \asymp
    \begin{cases}
        \lambda^s \left(\ln\frac{1}{\lambda}\right)^{r-1},
         & s<2, \\
        \lambda^2 \left(\ln\frac{1}{\lambda}\right)^r,
         & s=2, \\
        \lambda^2,
         & s>2.
    \end{cases}
\]
\hfill\BlackBox\\[2mm]

\par\noindent{\bf Proof of \Cref{thm:main}.}\par\noindent
Since \(\lambda=\Omega((mn)^{-\theta})\) for some \(\theta<\beta\), we have
\(v_r=\lambda^{-1/\beta}(\log(1/\lambda))^r/(mn)\to0\), and
Assumption~\ref{ass:source}\textup{(iv)} gives
\(\Sigma_r \neq0\) and \(\mu_{2r-2} \neq0\), so \Cref{thm:variance-summary}
applies.  Combining the variance expansion in \Cref{thm:variance-summary},
the bias expansion in \Cref{thm:bias-main}, and the bias--variance
decomposition \Cref{eq:bias-variance} gives the result.
\hfill\BlackBox\\[2mm]

\par\noindent{\bf Proof of \Cref{cor:main-noiseless}.}\par\noindent
As in the proof of \Cref{thm:main}, the condition on \(\lambda\) implies
\(v_r=\lambda^{-1/\beta}(\log(1/\lambda))^r/(mn)\to0\).  Combining the
noise-free variance expansion in \Cref{cor:variance-noiseless}, the bias
expansion in \Cref{thm:bias-main}, and the bias--variance decomposition
\Cref{eq:bias-variance} proves the claim.
\hfill\BlackBox\\[2mm]

\par\noindent{\bf Proof of \Cref{thm:minimax-lower-bound}.}\par\noindent
We prove the two terms by restricting to two submodels of
\(\mca P_{r,s}(R)\).

For the sampling term, let \(A\) take the values \(1\) and \(2\) with equal
probability, independently across curves, and consider the submodel
\[
    X=A f.
\]
Its target is \(\mu_r=\vartheta_r f^{\otimes r}\), where
\(\vartheta_r\coloneqq\E A^r>0\).
Let \(N=mn\), choose
\[
    J\asymp N^{1/(s\beta+1)},
    \qquad
    a_J=\eta J^{-(s\beta+1)/2},
\]
and write
\[
    f_{\omega}
    =
    b e_1+u_{\omega},
    \qquad
    u_{\omega}
    =
    a_J\sum_{j=J+1}^{2J}\omega_j e_j,
    \qquad
    \omega\in\{-1,1\}^{J}.
\]
Here \(b,\eta>0\) are fixed sufficiently small constants.  By
Assumption~\ref{ass:prelim-eigen-decay},
\[
    \norm{u_{\omega}}_{[H]^s}^2
    =
    a_J^2\sum_{j=J+1}^{2J}\lambda_j^{-s}
    \lesssim
    a_J^2J^{s\beta+1}
    =
    \eta^2.
\]
Since \(A\le2\) almost surely,
\[
    \left(
        \E\norm{A f_{\omega}}_{[H]^s}^{2r}
    \right)^{1/(2r)}
    \le
    2\left(b\lambda_1^{-s/2}+C\eta\right).
\]
Thus \(b\) and \(\eta\) may be chosen, independently of \(n,m\), so that
the laws of \(A f_{\omega}\) all belong to \(\mca P_{r,s}(R)\).

The Varshamov--Gilbert bound gives a subset
\(\Omega\subset\{-1,1\}^{J}\) such that
\[
    \log\abs{\Omega}\gtrsim J,
    \qquad
    d_{\mathrm H}(\omega,\omega')\ge J/8
    \quad
    \text{for distinct }\omega,\omega'\in\Omega.
\]
To compare the corresponding moment functions, let
\(E_J=\spn\{e_{J+1},\ldots,e_{2J}\}\).  Projecting
\(f_{\omega}^{\otimes r}-f_{\omega'}^{\otimes r}\) onto the orthogonal sum
of the \(r\) tensor subspaces having one coordinate in \(E_J\) and all other
coordinates in \(\spn\{e_1\}\) gives
\[
    b^{r-1}\sum_{\ell=1}^r
    e_1^{\otimes(\ell-1)}
    \otimes(u_{\omega}-u_{\omega'})
    \otimes e_1^{\otimes(r-\ell)}.
\]
The summands are mutually orthogonal.  Consequently,
\begin{align*}
    \norm{
        \vartheta_r f_{\omega}^{\otimes r}
        -\vartheta_r f_{\omega'}^{\otimes r}
    }_{L^2(\rho^{\otimes r})}^2
    & \ge
    \vartheta_r^2 r b^{2r-2}
    \norm{u_{\omega}-u_{\omega'}}_{L^2(\rho)}^2 \\
    & =
    4\vartheta_r^2 r b^{2r-2}a_J^2
    d_{\mathrm H}(\omega,\omega')
    \gtrsim
    a_J^2J.
\end{align*}

Let \(P_{\omega}\) denote the distribution of the observed data under
\(f_{\omega}\).  Revealing the latent amplitudes \(A_1,\ldots,A_n\) can only
increase the information.  Conditional on these amplitudes and the design
points, the observations are Gaussian, so the data-processing inequality gives
\begin{align*}
    \operatorname{KL}(P_{\omega},P_{\omega'})
    & \le
    \frac{mn\E A^2}{2\sigma^2}
    \norm{f_{\omega}-f_{\omega'}}_{L^2(\rho)}^2 \\
    & \lesssim
    mn a_J^2J
    \asymp
    \eta^2J.
\end{align*}
Taking \(\eta\) sufficiently small makes this a fixed small fraction of
\(\log\abs{\Omega}\).  Fano's lemma therefore yields
\[
    \inf_{\widehat\mu}
    \sup_{P\in\mca P_{r,s}(R)}
    \E_P
    \norm{\widehat\mu-\mu_r(P)}_{L^2(\rho^{\otimes r})}^2
    \gtrsim
    a_J^2J
    \asymp
    (mn)^{-s\beta/(s\beta+1)}.
\]

For the curve-level term, fix a nonzero \(h\in[H]^s\) with sufficiently
small norm and let \(X=A h\), where
\[
    \bbP_p(A=2)=p,
    \qquad
    \bbP_p(A=1)=1-p,
\]
for \(p\) in a fixed neighborhood of \(1/2\).  These laws also lie in
\(\mca P_{r,s}(R)\), and
\[
    \mu_{r,p}
    =
    \bigl(1+(2^r-1)p\bigr)h^{\otimes r}.
\]
The observed data are obtained from the independent amplitudes
\(A_1,\ldots,A_n\) through a Markov kernel that does not depend on \(p\).
Hence, for \(p,q\) near \(1/2\),
\[
    \operatorname{KL}(P_p,P_q)
    \le
    n\operatorname{KL}\bigl(\operatorname{Bern}(p),
    \operatorname{Bern}(q)\bigr)
    \lesssim
    n(p-q)^2.
\]
Choose \(p,q\) with \(\abs{p-q}=c_0n^{-1/2}\), where \(c_0>0\) is a
sufficiently small constant.  The last display is then bounded by a fixed
small constant, whereas
\[
    \norm{\mu_{r,p}-\mu_{r,q}}_{L^2(\rho^{\otimes r})}^2
    =
    (2^r-1)^2(p-q)^2\norm{h}_{L^2(\rho)}^{2r}
    \asymp
    n^{-1}.
\]
Le Cam's two-point lemma gives the lower bound \(n^{-1}\).  Taking the maximum
of the lower bounds from the two submodels, and using
\(a\vee b\ge(a+b)/2\), proves the theorem.
\hfill\BlackBox\\[2mm]

\par\noindent{\bf Proof of \Cref{cor:sparse-saturation}.}\par\noindent
Write
\[
    \mu_r=\sum_{\bm{i}} c_{\bm{i}} e_{\bm{i}}.
\]
By Assumption~\ref{ass:source}\textup{(ii)}, choose \(\bm{i}_\ast\) such that
\(c_{\bm{i}_\ast} \neq0\).  The eigenvalues of
\(\lambda T_{r,\lambda}^{-1}\) are
\(\lambda/(\Lambda_{\bm{i}}+\lambda)\).  Hence, for all small enough
\(\lambda\),
\[
    B(\lambda;\mu_r)^2
    =
    \sum_{\bm{i}}
    c_{\bm{i}}^2
    \left(
    \frac{\lambda}{\Lambda_{\bm{i}}+\lambda}
    \right)^2
    \ge
    c_{\bm{i}_\ast}^2
    \left(
    \frac{\lambda}{\Lambda_{\bm{i}_\ast}+\lambda}
    \right)^2
    \gtrsim
    \lambda^2 .
\]
On the other hand, \(\Cref{lem:bias-source-upper}\) gives
\(B(\lambda;\mu_r)^2=O(\lambda^2)\).  Thus
\(B(\lambda;\mu_r)^2\asymp\lambda^2\).
By \Cref{thm:main} in the noisy case, or \Cref{cor:main-noiseless} in the noiseless case, we have
\[
    \operatorname{MSE}_{r,\lambda}(\mca T)
    =
    \Theta_\bbP \left(
    n^{-1}
    +
    \frac{\lambda^{-1/\beta}}{mn}
    +
    \lambda^2
    \right).
\]
Since
\[
    \lambda^2+\frac{\lambda^{-1/\beta}}{mn}
    \gtrsim
    (mn)^{-2\beta/(2\beta+1)}
\]
for every \(\lambda>0\), the lower bound follows.
\hfill\BlackBox\\[2mm]

\par\noindent{\bf Proof of \Cref{lem:variance-E-expansion}.}\par\noindent
Since \(\mathbf{j}\) and \(\mathbf{k}\) have no repeated indices inside
themselves, each noise variable \(\ep_{i\ell}\) can appear at most once in
\(Y_{i\mathbf{j}}\) and at most once in \(Y_{i\mathbf{k}}\).  Therefore, in
the product \(Y_{i\mathbf{j}}Y_{i\mathbf{k}}\), every \(\ep_{i\ell}\) can
appear only \(0\), \(1\), or \(2\) times.  Terms in which a noise variable
appears exactly once have conditional expectation zero; terms in which a noise
variable appears twice must use a common sample location in
\(J(\mathbf{j})\cap J(\mathbf{k})\), and each such matched pair contributes
\(\sigma^2\).  Hence
\begin{equation}\label{eq:raw-moment-polynomial}
    \E[Y_{i\mathbf{j}}Y_{i\mathbf{k}} \mid \mca{T}]
    =
    \sum_{A\subseteq J(\mathbf{j})\cap J(\mathbf{k})}
    \sigma^{2|A|}
    \mu_{2r-2|A|} \bigl(\bm{t}^{A}_{i,\mathbf{j},\mathbf{k}}\bigr).
\end{equation}
Subtracting \(\mu_r(t_{i\mathbf{j}})\mu_r(t_{i\mathbf{k}})\) only changes the
coefficient of \(\sigma^0\), and that coefficient becomes
\(\Sigma_r(t_{i\mathbf{j}};t_{i\mathbf{k}})\).  Grouping the remaining terms
according to \(d=|A|\) gives \cref{eq:b-polynomial}.
\hfill\BlackBox\\[2mm]

\par\noindent{\bf Proof of \Cref{lem:appendix-operator-concentration}.}\par\noindent
This is the tensor-product analogue of \citet[Proposition~4.6 and 4.7]{li2024GeneralizationErrorCurves}.
We split the proof into two auxiliary steps.

\medskip
\noindent\textbf{Auxiliary Lemma 1.}
For every \(\delta\in(0,1)\), with probability at least \(1-\delta\),
\begin{equation}\label{eq:appendix-operator-concentration-explicit}
    \|T_{r,\lambda}^{-1/2}(\hat{T}_r-T_r)T_{r,\lambda}^{-1/2}\|
    \le
    \sqrt{u_{r,\delta}(\lambda)}
    +
    \frac{2}{3}u_{r,\delta}(\lambda),
\end{equation}
where
\begin{equation}\label{eq:appendix-operator-concentration-u}
    u_{r,\delta}(\lambda)
    \coloneqq
    \frac{
        2M^r \mcaN_1^{(r)}(\lambda)
    }{
        n\lfloor m/r\rfloor
    }
    \ln \left(
    \frac{
            4(\|T_r\|+\lambda)\mcaN_1^{(r)}(\lambda)
        }{
            \delta\|T_r\|
        }
    \right).
\end{equation}

\smallskip
\emph{Proof.}
Define the operator-valued kernel
\[
    Z_{r,\lambda}(\bm{t})
    \coloneqq
    T_{r,\lambda}^{-1/2}
    \bigl(
    K(\bm{t})\otimes K(\bm{t})
    \bigr)
    T_{r,\lambda}^{-1/2},
    \qquad
    \bm{t}\in \mca I^r.
\]
Then \(Z_{r,\lambda}(\bm{t})\) is a positive self-adjoint Hilbert-Schmidt operator on \(\caH^{\otimes r}\), and
\[
    \frac{1}{n}\sum_{i=1}^n
    \frac{1}{(m)_r}
    \sum_{\mathbf{j}\in I_m^r}
    Z_{r,\lambda}(t_{i\mathbf{j}})
    =
    T_{r,\lambda}^{-1/2} \hat{T}_r T_{r,\lambda}^{-1/2}.
\]
Also,
\[
    \E Z_{r,\lambda}(\bm{t})
    =
    T_{r,\lambda}^{-1/2} T_r T_{r,\lambda}^{-1/2},
\]
so the centered \(U\)-statistic in \Cref{lem:appendix-U-bernstein-operator} is exactly
\[
    T_{r,\lambda}^{-1/2}(\hat{T}_r-T_r)T_{r,\lambda}^{-1/2}.
\]

By \Cref{lem:appendix-phi-T},
\[
    \|Z_{r,\lambda}(\bm{t})\|
    =
    \|T_{r,\lambda}^{-1/2} K(\bm{t})\|_{\caH^{\otimes r}}^2
    \le
    M^r \mcaN_1^{(r)}(\lambda)
    =:
    L_{r,\lambda}.
\]
Taking expectations gives
\[
    \|T_{r,\lambda}^{-1/2} T_r T_{r,\lambda}^{-1/2}\|
    \le
    \E\|Z_{r,\lambda}(\bm{t})\|
    \le
    L_{r,\lambda}.
\]
Hence
\[
    \|Z_{r,\lambda}(\bm{t})-\E Z_{r,\lambda}(\bm{t})\|
    \le
    2L_{r,\lambda}
    =:
    L.
\]

Next, since \(Z_{r,\lambda}(\bm{t})\) is positive and rank one,
\begin{equation*}
    Z_{r,\lambda}(\bm{t})^2
    =
    \|T_{r,\lambda}^{-1/2} K(\bm{t})\|_{\caH^{\otimes r}}^2
    Z_{r,\lambda}(\bm{t})
    \preceq
    L_{r,\lambda} Z_{r,\lambda}(\bm{t}).
\end{equation*}
Using \(\E(B-\E B)^2\preceq \E B^2\), we obtain
\begin{align*}
    \E\bigl[
        (Z_{r,\lambda}(\bm{t})-\E Z_{r,\lambda}(\bm{t}))^2
        \bigr]
     & \preceq
    \E[Z_{r,\lambda}(\bm{t})^2]  \\
     & \preceq
    L_{r,\lambda} \E[Z_{r,\lambda}(\bm{t})] \\
     & =
    L_{r,\lambda} T_{r,\lambda}^{-1/2} T_r T_{r,\lambda}^{-1/2}
    =:
    S_{r,\lambda}.
\end{align*}
Since \(T_r\) and \(T_{r,\lambda}\) commute,
\[
    S_{r,\lambda}
    =
    L_{r,\lambda} T_r T_{r,\lambda}^{-1}.
\]
Therefore
\[
    \|S_{r,\lambda}\|
    =
    L_{r,\lambda}\|T_r T_{r,\lambda}^{-1}\|
    =
    L_{r,\lambda} \frac{\|T_r\|}{\|T_r\|+\lambda},
\]
and
\[
    \Tr S_{r,\lambda}
    =
    L_{r,\lambda} \Tr(T_r T_{r,\lambda}^{-1})
    =
    L_{r,\lambda} \mcaN_1^{(r)}(\lambda).
\]
Hence
\[
    \frac{\Tr S_{r,\lambda}}{\|S_{r,\lambda}\|}
    =
    \frac{(\|T_r\|+\lambda)\mcaN_1^{(r)}(\lambda)}{\|T_r\|}.
\]

Apply \Cref{lem:appendix-U-bernstein-operator} with \(k=r\) and \(D=\E Z_{r,\lambda}(\bm{t})\).
With probability at least \(1-\delta\),
\begin{align*}
          & \|T_{r,\lambda}^{-1/2}(\hat{T}_r-T_r)T_{r,\lambda}^{-1/2}\| \\
    \le\  &
    \sqrt{
    \frac{
    2\|S_{r,\lambda}\|
    }{
    n\lfloor m/r\rfloor
    }
    \ln \left(
    \frac{
        4\Tr S_{r,\lambda}
    }{
        \delta\|S_{r,\lambda}\|
    }
    \right)
    }
    +
    \frac{
        2L
    }{
        3n\lfloor m/r\rfloor
    }
    \ln \left(
    \frac{
        4\Tr S_{r,\lambda}
    }{
        \delta\|S_{r,\lambda}\|
    }
    \right).
\end{align*}
Since
\[
    \|S_{r,\lambda}\|
    \le L_{r,\lambda}=
    M^r \mcaN_1^{(r)}(\lambda),
    \qquad
    \frac{
    4\Tr S_{r,\lambda}
    }{
    \delta\|S_{r,\lambda}\|
    }
    =
    \frac{
        4(\|T_r\|+\lambda)\mcaN_1^{(r)}(\lambda)
    }{
        \delta\|T_r\|
    },
\]
this is exactly \cref{eq:appendix-operator-concentration-explicit}.

\medskip
\noindent\textbf{Auxiliary Lemma 2.}
If \(A\) is a self-adjoint operator with \(\|A\|\le \frac{1}{2}\), then \(I+A\) is invertible and
\[
    \|(I+A)^{-1}\|
    \le
    \frac{1}{1-\|A\|}
    \le 2.
\]

\smallskip
\emph{Proof.}
The spectrum of \(A\) is contained in \([-\|A\|,\|A\|]\subset[-1/2,1/2]\), so the spectrum of \(I+A\) lies in
\([1/2,3/2]\). Therefore \(I+A\) is invertible and
\[
    \|(I+A)^{-1}\|
    =
    \frac{1}{\min\sigma(I+A)}
    \le 2.
\]

We now return to the main claim.
By \Cref{lem:appendix-effective-dimension},
\[
    \mcaN_1^{(r)}(\lambda)
    \asymp
    \lambda^{-1/\beta}
    \left(
    \ln \frac{1}{\lambda}
    \right)^{r-1}.
\]
Hence, for every fixed \(\delta\in(0,1)\),
\begin{align*}
    u_{r,\delta}(\lambda)
     & \le
    C_\delta
    \frac{
        \lambda^{-1/\beta}
        (\ln \frac{1}{\lambda})^{r-1}
    }{
        n\lfloor m/r\rfloor
    }
    \ln \left(
    C_\delta
    \lambda^{-1/\beta}
    \left(
        \ln \frac{1}{\lambda}
        \right)^{r-1}
    \right) \\
     & \le
    C_\delta
    \frac{
        \lambda^{-1/\beta}
        (\ln \frac{1}{\lambda})^{r}
    }{
        n\lfloor m/r\rfloor
    },
\end{align*}
for all sufficiently small \(\lambda\), because
\[
    \ln \left(
    C_\delta
    \lambda^{-1/\beta}
    \left(
        \ln \frac{1}{\lambda}
        \right)^{r-1}
    \right)
    =
    O_\delta \left(
    \ln \frac{1}{\lambda}
    \right).
\]
Since \(m\ge r\), we have \(\lfloor m/r\rfloor\ge m/(2r)\), hence
\begin{equation}\label{eq:appendix-operator-concentration-u-vr}
    u_{r,\delta}(\lambda)
    \le
    C_{\delta,r}
    \frac{
        \lambda^{-1/\beta}
        (\ln \frac{1}{\lambda})^{r}
    }{
        mn
    }
    =
    C_{\delta,r} v_r.
\end{equation}
Combining \cref{eq:appendix-operator-concentration-explicit} with \cref{eq:appendix-operator-concentration-u-vr},
we obtain
\[
    \|T_{r,\lambda}^{-1/2}(\hat{T}_r-T_r)T_{r,\lambda}^{-1/2}\|
    =
    O_\bbP(\sqrt{v_r}),
\]
because \(u_{r,\delta}(\lambda)\to0\) and
\[
    \sqrt{u_{r,\delta}(\lambda)}
    +
    \frac{2}{3}u_{r,\delta}(\lambda)
    \le
    2\sqrt{u_{r,\delta}(\lambda)}
    \le
    C_{\delta,r} \sqrt{v_r}
\]
for all sufficiently small \(\lambda\).

Since this implies
\[
    \|T_{r,\lambda}^{-1/2}(\hat{T}_r-T_r)T_{r,\lambda}^{-1/2}\|
    =
    o_\bbP(1),
\]
the event
\[
    \left\{
    \|T_{r,\lambda}^{-1/2}(\hat{T}_r-T_r)T_{r,\lambda}^{-1/2}\|
    \le \frac{1}{2}
    \right\}
\]
has probability tending to one.
On this event,
\[
    T_{r,\lambda}^{-1/2} \hat{T}_{r,\lambda} T_{r,\lambda}^{-1/2}
    =
    I
    +
    T_{r,\lambda}^{-1/2}(\hat{T}_r-T_r)T_{r,\lambda}^{-1/2}.
\]
Hence, by Auxiliary Lemma 2,
\[
    \|T_{r,\lambda}^{1/2} \hat{T}_{r,\lambda}^{-1} T_{r,\lambda}^{1/2}\|
    =
    \left\|
    \bigl(
    T_{r,\lambda}^{-1/2} \hat{T}_{r,\lambda} T_{r,\lambda}^{-1/2}
    \bigr)^{-1}
    \right\|
    \le 2.
\]
\hfill\BlackBox\\[2mm]

\par\noindent{\bf Proof of \Cref{lem:variance-phi-hatT}.}\par\noindent
The resolvent identity gives
\[
    \hat{T}_{r,\lambda}^{-1}-T_{r,\lambda}^{-1}
    =
    \hat{T}_{r,\lambda}^{-1}(T_r-\hat{T}_r)T_{r,\lambda}^{-1}.
\]
For \(h\in \caH^{\otimes r}\), we use the identity
\[
    \|h\|_{L^2(\mca I^r)}
    =
    \|T_r^{1/2} h\|_{\caH^{\otimes r}}.
\]
Therefore, for every \(\bm{t}\in \mca I^r\),
\begin{align*}
    \|\hat{T}_{r,\lambda}^{-1} K(\bm{t})\|_{L^2}
    \le\  &
    \|T_{r,\lambda}^{-1} K(\bm{t})\|_{L^2} \\
          & +
    \|T_r^{1/2} \hat{T}_{r,\lambda}^{-1}
    (\hat{T}_r-T_r)T_{r,\lambda}^{-1} K(\bm{t})\|_{\caH^{\otimes r}}.
\end{align*}

Set
\[
    \Delta_r
    \coloneqq
    T_{r,\lambda}^{-1/2}(\hat{T}_r-T_r)T_{r,\lambda}^{-1/2}.
\]
Then
\begin{align*}
          &
    \|T_r^{1/2} \hat{T}_{r,\lambda}^{-1}
    (\hat{T}_r-T_r)T_{r,\lambda}^{-1} K(\bm{t})\|_{\caH^{\otimes r}} \\
    \le\  &
    \|T_r^{1/2} T_{r,\lambda}^{-1/2}\|
    \|T_{r,\lambda}^{1/2} \hat{T}_{r,\lambda}^{-1} T_{r,\lambda}^{1/2}\|
    \|\Delta_r\|
    \|T_{r,\lambda}^{-1/2} K(\bm{t})\|_{\caH^{\otimes r}}.
\end{align*}
Since \(0\le T_r \le T_{r,\lambda}\),
\[
    \|T_r^{1/2} T_{r,\lambda}^{-1/2}\|\le 1.
\]
On the high-probability event from
\Cref{lem:appendix-operator-concentration},
\[
    \|T_{r,\lambda}^{1/2} \hat{T}_{r,\lambda}^{-1} T_{r,\lambda}^{1/2}\|
    \le 2.
\]
Hence, on that event,
\[
    \|T_r^{1/2} \hat{T}_{r,\lambda}^{-1}
    (\hat{T}_r-T_r)T_{r,\lambda}^{-1} K(\bm{t})\|_{\caH^{\otimes r}}
    \le
    2\|\Delta_r\|
    \|T_{r,\lambda}^{-1/2} K(\bm{t})\|_{\caH^{\otimes r}}.
\]
By \Cref{lem:appendix-phi-T},
\[
    \sup_{\bm{t}\in \mca I^r}
    \|T_{r,\lambda}^{-1/2} K(\bm{t})\|_{\caH^{\otimes r}}^2
    \le
    M^r \mcaN_1^{(r)}(\lambda),
\]
and
\[
    \sup_{\bm{t}\in \mca I^r}
    \|T_{r,\lambda}^{-1} K(\bm{t})\|_{L^2}^2
    \le
    M^r \mcaN_2^{(r)}(\lambda).
\]
Since \(v_r \to0\), \Cref{lem:appendix-operator-concentration} also gives
\[
    \|\Delta_r\|
    =
    o_\bbP(1),
\]
so with probability tending to one we may additionally assume
\(\|\Delta_r\|\le1\).  On that event,
\[
    \sup_{\bm{t}\in \mca I^r}
    \|T_r^{1/2} \hat{T}_{r,\lambda}^{-1}
    (\hat{T}_r-T_r)T_{r,\lambda}^{-1} K(\bm{t})\|_{\caH^{\otimes r}}
    \le
    2\sqrt{M^r \mcaN_1^{(r)}(\lambda)}.
\]
Since
\(\mcaN_1^{(r)}(\lambda)\asymp \mcaN_2^{(r)}(\lambda)\) by
\Cref{lem:appendix-effective-dimension}, this is bounded by
\(C\sqrt{M^r \mcaN_2^{(r)}(\lambda)}\) for all sufficiently small \(\lambda\).
This proves \cref{eq:phi-hatT-bound-main}.
\hfill\BlackBox\\[2mm]

\par\noindent{\bf Proof of \Cref{lem:variance-overlap-count}.}\par\noindent
Classify pairs \((\mathbf{j},\mathbf{k})\) by the overlap size
\(q=q(\mathbf{j},\mathbf{k})\).  For a fixed \(q\), the number of ordered
pairs \((\mathbf{j},\mathbf{k})\in I_m^r \times I_m^r\) with exactly \(q\)
common sample indices is
\[
    \binom{r}{q}^2 q! (m)_{2r-q}.
\]
Indeed, choose the \(q\) shared positions in \(\mathbf{j}\), choose the \(q\)
shared positions in \(\mathbf{k}\), match them by a bijection, and then choose
the remaining \(2r-q\) distinct sample indices in order.  For each such pair
there are \(\binom{q}{d}\) choices of the \(d\) common locations that enter the
inner sum.  Summing over \(q=d,\ldots,r\) gives
\cref{eq:variance-overlap-count}.
\hfill\BlackBox\\[2mm]
 \section{Table of Notation}
\label{app:notation}

Table~\ref{tab:notation} lists the symbols used most frequently in the paper.

{\small
\begin{longtable}{p{0.28\textwidth}p{0.62\textwidth}}
\caption{Frequently used notation.}
\label{tab:notation}\\
\toprule
Symbol & Meaning \\
\midrule
\endfirsthead
\caption[]{Frequently used notation (continued).}\\
\toprule
Symbol & Meaning \\
\midrule
\endhead
\midrule
\multicolumn{2}{r}{\emph{Continued on next page}}\\
\endfoot
\bottomrule
\endlastfoot
\(\mca{I},\rho\) & compact sampling space and design distribution \\
\(n,m,r,\lambda,\sigma^2\) & number of sample paths, number of observations per path, moment order, regularization parameter, and noise variance \\
\(X_i,t_{ij},\ep_{ij},y_{ij}\) & latent curve, design point, measurement noise, and observation \\
\(\mca{T}\) & collection of all design points \((t_{ij})_{1\le i\le n, 1\le j\le m}\) \\
\(\mu_r,\mu_0,\Sigma_r\) & \(r\)-th moment function, the convention \(\mu_0=1\), and the covariance kernel of the \(r\)-fold product process \\
\(\mca P_{r,s}(R)\) & uniform \([H]^s\)-moment class used for the minimax lower bound \\
\(I_m^r,\falling{m}{r}\) & ordered distinct coordinate set and its cardinality \\
\(Y_{i\bm{j}},t_{i\bm{j}}\) & product pseudo-response and corresponding \(r\)-tuple of design points \\
\(\caH,k_x,K(\bm{x}),\caH^{\otimes r}\) & RKHS, scalar feature \(k(x,\cdot)\), tensor feature \(k_{x_1} \otimes\cdots\otimes k_{x_r}\), and tensor-product RKHS \\
\([H]^s,([H]^s)^{\otimes r}\) & interpolation space and its tensor-product analogue \\
\(M_\alpha,\alpha_0\) & embedding norm of \([H]^\alpha\hookrightarrow L^\infty(\mca I,\rho)\) and embedding index \\
\((\lambda_j,e_j)\) & Mercer eigenpairs of \(T\), with eigenvalues repeated according to multiplicity \\
\((\theta_m,V_m,d_m,k_m)\) & distinct eigenvalues, associated eigenspaces, their dimensions, and spectral projection kernels \\
\(e_{\bm{i}},\Lambda_{\bm{i}},c_{\bm{i}}\) & tensor eigenfunction, tensor eigenvalue, and coefficient in the expansion of \(\mu_r\) \\
\(T,T_r,\hat{T}_r\) & base integral operator, tensor-product operator \(T_r=T^{\otimes r}\), and empirical tensor operator \\
\(T_{r,\lambda},\hat{T}_{r,\lambda}\) & shorthand for \(T_r+\lambda I\) and \(\hat{T}_r+\lambda I\) \\
\(\hat{\zeta}_r,\tilde{\zeta}_r\) & empirical and conditional-mean right-hand sides for KRR \\
\(\hat{\mu}_{r,\lambda},\tilde{\mu}_{r,\lambda}\) & KRR estimator and its conditional mean \\
\(\operatorname{MSE}_{r,\lambda}(\mca T)\) & conditional generalization error given the design points \\
\(\Var(\lambda),\Bias(\lambda)^2\) & conditional variance and conditional bias terms in the proof decomposition \\
\(B(\lambda;\mu_r)^2\) & deterministic squared bias scale \\
\(V_{\mathrm{curve}},V_{\mathrm{point}},V_{\mathrm{noise}}\) & deterministic variance coefficients in the main expansion, with prefactors \(1/n\), \(1/(mn)\), and \(\sigma^2/(mn)\), respectively \\
\(\hat{V}_d,\tilde{V}_d,V_d\) & variance contribution of order \(\sigma^{2d}\), with mixed and population versions used for the leading terms \\
\(\hat{G}_0,G_0,\hat{G}_1,G_1\) & empirical and population operators packaging the noise-free and first-noise trace terms \\
\(\lambda_j^{(r)},N_r(u)\) & non-increasing rearrangement and counting function of the tensor eigenvalues \(\Lambda_{\bm{i}}\) \\
\(\mcaN_p^{(r)}(\lambda)\) & tensor effective dimension \\
\(\Delta_r,v_r\) & normalized empirical-operator fluctuation and its scale \(\lambda^{-1/\beta}(\ln(1/\lambda))^r/(mn)\) \\
\end{longtable}
}


\begin{thebibliography}{36}
\providecommand{\natexlab}[1]{#1}
\providecommand{\url}[1]{\texttt{#1}}
\expandafter\ifx\csname urlstyle\endcsname\relax
  \providecommand{\doi}[1]{doi: #1}\else
  \providecommand{\doi}{doi: \begingroup \urlstyle{rm}\Url}\fi

\bibitem[Arcones(1995)]{arcones1995BernsteintypeInequalityUstatistics}
Miguel~A. Arcones.
\newblock A {{Bernstein-type}} inequality for {{U-statistics}} and
  {{U-processes}}.
\newblock \emph{Statistics \& Probability Letters}, 22\penalty0 (3):\penalty0
  239--247, February 1995.
\newblock \doi{10.1016/0167-7152(94)00072-G}.

\bibitem[Bauer et~al.(2007)Bauer, Pereverzev, and
  Rosasco]{bauer2007RegularizationAlgorithmsLearning}
Frank Bauer, Sergei Pereverzev, and Lorenzo Rosasco.
\newblock On regularization algorithms in learning theory.
\newblock \emph{Journal of Complexity}, 23\penalty0 (1):\penalty0 52--72,
  February 2007.
\newblock \doi{10.1016/j.jco.2006.07.001}.

\bibitem[Bordelon et~al.(2020)Bordelon, Canatar, and
  Pehlevan]{bordelon2020SpectrumDependentLearning}
Blake Bordelon, Abdulkadir Canatar, and Cengiz Pehlevan.
\newblock Spectrum {{Dependent Learning Curves}} in {{Kernel Regression}} and
  {{Wide Neural Networks}}.
\newblock In \emph{Proceedings of the 37th {{International Conference}} on
  {{Machine Learning}}}, pages 1024--1034. PMLR, November 2020.
\newblock URL \url{https://proceedings.mlr.press/v119/bordelon20a.html}.

\bibitem[Cai and Yuan(2010)]{cai2010NonparametricCovarianceFunction}
T~Tony Cai and Ming Yuan.
\newblock Nonparametric {{Covariance Function Estimation}} for {{Functional}}
  and {{Longitudinal Data}}.
\newblock Technical report, {University of Pennsylvania and Georgia Institute
  of Technology}, Atlanta, GA, 2010.

\bibitem[Cai and Yuan(2011)]{cai2011OptimalEstimationMean}
T.~Tony Cai and Ming Yuan.
\newblock Optimal estimation of the mean function based on discretely sampled
  functional data: {{Phase}} transition.
\newblock \emph{The Annals of Statistics}, 39\penalty0 (5), October 2011.
\newblock \doi{10.1214/11-AOS898}.

\bibitem[Caponnetto and De~Vito(2007)]{caponnetto2007OptimalRatesRegularized}
A.~Caponnetto and E.~De~Vito.
\newblock Optimal {{Rates}} for the {{Regularized Least-Squares Algorithm}}.
\newblock \emph{Foundations of Computational Mathematics}, 7\penalty0
  (3):\penalty0 331--368, July 2007.
\newblock \doi{10.1007/s10208-006-0196-8}.

\bibitem[Cui et~al.(2022)Cui, Loureiro, Krzakala, and
  Zdeborov{\'a}]{cui2022GeneralizationErrorRates}
Hugo Cui, Bruno Loureiro, Florent Krzakala, and Lenka Zdeborov{\'a}.
\newblock Generalization error rates in kernel regression: The crossover from
  the noiseless to noisy regime*.
\newblock \emph{Journal of Statistical Mechanics: Theory and Experiment},
  2022\penalty0 (11):\penalty0 114004, November 2022.
\newblock \doi{10.1088/1742-5468/ac9829}.

\bibitem[Dai and Xu(2013)]{dai2013ApproximationTheoryHarmonic}
Feng Dai and Yuan Xu.
\newblock \emph{Approximation {{Theory}} and {{Harmonic Analysis}} on
  {{Spheres}} and {{Balls}}}.
\newblock Springer {{Monographs}} in {{Mathematics}}. Springer New York, New
  York, NY, 2013.
\newblock \doi{10.1007/978-1-4614-6660-4}.

\bibitem[DLMF()]{NIST:DLMF}
DLMF.
\newblock {{{\emph{NIST}}}}{\emph{ digital library of mathematical functions}},
  March 2026.
\newblock URL \url{https://dlmf.nist.gov/}.

\bibitem[Fischer and Steinwart(2020)]{fischer2020SobolevNormLearning}
Simon Fischer and Ingo Steinwart.
\newblock Sobolev norm learning rates for regularized least-squares algorithms.
\newblock \emph{Journal of Machine Learning Research}, 21\penalty0
  (205):\penalty0 1--38, 2020.
\newblock URL \url{http://jmlr.org/papers/v21/19-734.html}.

\bibitem[Giraudo(2025)]{giraudo2025ExponentialInequalityHilbertvalued}
Davide Giraudo.
\newblock An exponential inequality for {{Hilbert-valued U}} -statistics of
  i.i.d. data.
\newblock \emph{Journal of Multivariate Analysis}, 207:\penalty0 105406, May
  2025.
\newblock \doi{10.1016/j.jmva.2025.105406}.

\bibitem[Gupta and Sriperumbudur(2026)]{gupta2026MinimaxOptimalEstimation}
Naveen Gupta and Bharath~K. Sriperumbudur.
\newblock Minimax {{Optimal Estimation}} of {{Mean}} and {{Covariance
  Functions}} with {{Spectral Regularization}}, March 2026.

\bibitem[Hall et~al.(2006)Hall, M{\"u}ller, and
  Wang]{hall2006PropertiesPrincipalComponent}
Peter Hall, Hans-Georg M{\"u}ller, and Jane-Ling Wang.
\newblock Properties of principal component methods for functional and
  longitudinal data analysis.
\newblock \emph{The Annals of Statistics}, 34\penalty0 (3), June 2006.
\newblock \doi{10.1214/009053606000000272}.

\bibitem[Hoeffding(1963)]{hoeffding1963ProbabilityInequalitiesSums}
Wassily Hoeffding.
\newblock Probability {{Inequalities}} for {{Sums}} of {{Bounded Random
  Variables}}.
\newblock \emph{Journal of the American Statistical Association}, 58\penalty0
  (301):\penalty0 13--30, March 1963.
\newblock \doi{10.1080/01621459.1963.10500830}.

\bibitem[Hsing and Eubank(2015)]{hsing2015TheoreticalFoundationsFunctional}
Tailen Hsing and Randall Eubank.
\newblock \emph{Theoretical {{Foundations}} of {{Functional Data Analysis}},
  with an {{Introduction}} to {{Linear Operators}}}.
\newblock Wiley {{Series}} in {{Probability}} and {{Statistics}}. Wiley, 1
  edition, May 2015.
\newblock \doi{10.1002/9781118762547}.

\bibitem[Li et~al.(2026)Li, Lindquist, Gunning, and
  Crainiceanu]{li2026FunctionalMomentsRegression}
Mingyuan Li, Martin~A. Lindquist, Edward Gunning, and Ciprian Crainiceanu.
\newblock Functional {{Moments Regression}}, 2026.

\bibitem[Li and Hsing(2010)]{li2010UniformConvergenceRates}
Yehua Li and Tailen Hsing.
\newblock Uniform convergence rates for nonparametric regression and principal
  component analysis in functional/longitudinal data.
\newblock \emph{The Annals of Statistics}, 38\penalty0 (6), December 2010.
\newblock \doi{10.1214/10-AOS813}.

\bibitem[Li et~al.(2023)Li, Zhang, and Lin]{li2023AsymptoticLearningCurves}
Yicheng Li, Haobo Zhang, and Qian Lin.
\newblock On the {{Asymptotic Learning Curves}} of {{Kernel Ridge Regression}}
  under {{Power-law Decay}}.
\newblock In \emph{Thirty-Seventh {{Conference}} on {{Neural Information
  Processing Systems}}}, 2023.
\newblock URL \url{https://openreview.net/forum?id=E4P5kVSKlT}.

\bibitem[Li et~al.(2024{\natexlab{a}})Li, Gan, Shi, and
  Lin]{li2024GeneralizationErrorCurves}
Yicheng Li, Weiye Gan, Zuoqiang Shi, and Qian Lin.
\newblock Generalization {{Error Curves}} for {{Analytic Spectral Algorithms}}
  under {{Power-law Decay}}, 2024{\natexlab{a}}.

\bibitem[Li et~al.(2024{\natexlab{b}})Li, Zhang, and
  Lin]{li2024SaturationEffectKernel}
Yicheng Li, Haobo Zhang, and Qian Lin.
\newblock On the {{Saturation Effect}} of {{Kernel Ridge Regression}},
  2024{\natexlab{b}}.

\bibitem[Lin et~al.(2020)Lin, Rudi, Rosasco, and
  Cevher]{lin2020OptimalRatesSpectral}
Junhong Lin, Alessandro Rudi, Lorenzo Rosasco, and Volkan Cevher.
\newblock Optimal rates for spectral algorithms with least-squares regression
  over {{Hilbert}} spaces.
\newblock \emph{Applied and Computational Harmonic Analysis}, 48\penalty0
  (3):\penalty0 868--890, May 2020.
\newblock \doi{10.1016/j.acha.2018.09.009}.

\bibitem[Minsker(2017)]{minsker2017ExtensionsBernsteinsInequality}
Stanislav Minsker.
\newblock On some extensions of {{Bernstein}}'s inequality for self-adjoint
  operators.
\newblock \emph{Statistics \& Probability Letters}, 127:\penalty0 111--119,
  August 2017.
\newblock \doi{10.1016/j.spl.2017.03.020}.

\bibitem[Montgomery and
  Vaughan(2006)]{montgomery2006MultiplicativeNumberTheory}
Hugh~L. Montgomery and Robert~C. Vaughan.
\newblock \emph{Multiplicative {{Number Theory I}}: {{Classical Theory}}}.
\newblock Cambridge University Press, 1 edition, November 2006.
\newblock \doi{10.1017/CBO9780511618314}.

\bibitem[Ramsay and Silverman(1997)]{ramsay1997FunctionalDataAnalysis}
James~O Ramsay and Bernard~W Silverman.
\newblock \emph{Functional Data Analysis}.
\newblock Springer, 1997.

\bibitem[Rice and Silverman(1991)]{rice1991EstimatingMeanCovariance}
John~A. Rice and B.~W. Silverman.
\newblock Estimating the {{Mean}} and {{Covariance Structure Nonparametrically
  When}} the {{Data}} are {{Curves}}.
\newblock \emph{Journal of the Royal Statistical Society Series B: Statistical
  Methodology}, 53\penalty0 (1):\penalty0 233--243, September 1991.
\newblock \doi{10.1111/j.2517-6161.1991.tb01821.x}.

\bibitem[Sriperumbudur and Sterge(2022)]{sriperumbudur2022ApproximateKernelPCA}
Bharath~K. Sriperumbudur and Nicholas Sterge.
\newblock Approximate kernel {{PCA}}: {{Computational}} versus statistical
  trade-off.
\newblock \emph{The Annals of Statistics}, 50\penalty0 (5), October 2022.
\newblock \doi{10.1214/22-AOS2204}.

\bibitem[Steinwart and Christmann(2008)]{steinwart2008SupportVectorMachines}
Ingo Steinwart and Andreas Christmann.
\newblock \emph{Support {{Vector Machines}}}.
\newblock Information {{Science}} and {{Statistics}}. Springer New York, New
  York, NY, 2008.
\newblock \doi{10.1007/978-0-387-77242-4}.

\bibitem[Steinwart and Scovel(2012)]{steinwart2012MercersTheoremGeneral}
Ingo Steinwart and Clint Scovel.
\newblock Mercer's {{Theorem}} on {{General Domains}}: {{On}} the
  {{Interaction}} between {{Measures}}, {{Kernels}}, and {{RKHSs}}.
\newblock \emph{Constructive Approximation}, 35\penalty0 (3):\penalty0
  363--417, June 2012.
\newblock \doi{10.1007/s00365-012-9153-3}.

\bibitem[Terada and Yara(2026)]{terada2026TheoryNonparametricCovariance}
Yoshikazu Terada and Atsutomo Yara.
\newblock A {{Theory}} of {{Nonparametric Covariance Function Estimation}} for
  {{Discretely Observed Data}}, 2026.

\bibitem[Tropp(2015)]{tropp2015IntroductionMatrixConcentration}
Joel~A. Tropp.
\newblock An {{Introduction}} to {{Matrix Concentration Inequalities}}, 2015.

\bibitem[Wang et~al.(2016)Wang, Chiou, and
  M{\"u}ller]{wang2016FunctionalDataAnalysis}
Jane-Ling Wang, Jeng-Min Chiou, and Hans-Georg M{\"u}ller.
\newblock Functional {{Data Analysis}}.
\newblock \emph{Annual Review of Statistics and Its Application}, 3\penalty0
  (1):\penalty0 257--295, June 2016.
\newblock \doi{10.1146/annurev-statistics-041715-033624}.

\bibitem[Wang et~al.(2022)Wang, Wong, and
  Zhang]{wang2022LowRankCovarianceFunction}
Jiayi Wang, Raymond K.~W. Wong, and Xiaoke Zhang.
\newblock Low-{{Rank Covariance Function Estimation}} for {{Multidimensional
  Functional Data}}.
\newblock \emph{Journal of the American Statistical Association}, 117\penalty0
  (538):\penalty0 809--822, April 2022.
\newblock \doi{10.1080/01621459.2020.1820344}.

\bibitem[Wong and
  Zhang(2019)]{wong2019NonparametricOperatorregularizedCovariance}
Raymond~K.W. Wong and Xiaoke Zhang.
\newblock Nonparametric operator-regularized covariance function estimation for
  functional data.
\newblock \emph{Computational Statistics \& Data Analysis}, 131:\penalty0
  131--144, March 2019.
\newblock \doi{10.1016/j.csda.2018.05.013}.

\bibitem[Yao et~al.(2005)Yao, M{\"u}ller, and
  Wang]{yao2005FunctionalDataAnalysis}
Fang Yao, Hans-Georg M{\"u}ller, and Jane-Ling Wang.
\newblock Functional {{Data Analysis}} for {{Sparse Longitudinal Data}}.
\newblock \emph{Journal of the American Statistical Association}, 100\penalty0
  (470):\penalty0 577--590, June 2005.
\newblock \doi{10.1198/016214504000001745}.

\bibitem[Zhang et~al.(2024)Zhang, Li, and
  Lin]{zhang2024OptimalityMisspecifiedSpectral}
Haobo Zhang, Yicheng Li, and Qian Lin.
\newblock On the optimality of misspecified spectral algorithms.
\newblock \emph{Journal of Machine Learning Research}, 25\penalty0
  (188):\penalty0 1--50, 2024.
\newblock URL \url{http://jmlr.org/papers/v25/23-0383.html}.

\bibitem[Zhang and Wang(2016)]{zhang2016SparseDenseFunctional}
Xiaoke Zhang and Jane-Ling Wang.
\newblock From sparse to dense functional data and beyond.
\newblock \emph{The Annals of Statistics}, 44\penalty0 (5), October 2016.
\newblock \doi{10.1214/16-AOS1446}.

\end{thebibliography}
\end{document}